\documentclass[11pt,reqno]{amsart}
\usepackage{hyperref}
\usepackage{amsmath}
\usepackage{geometry}
\usepackage{amssymb, latexsym}
\usepackage{verbatim}
\usepackage[T1]{fontenc}
\usepackage{enumerate}

\usepackage{mathtools}
\usepackage[tableposition=top]{caption}
\usepackage{booktabs,dcolumn}

\usepackage[skins]{tcolorbox}
\usepackage{lipsum}

\newtheorem{theorem}{Theorem}[section]
\newtheorem{proposition}[theorem]{Proposition}
\newtheorem{proposition1}[theorem]{Proposition}

\newtheorem{lemma}[theorem]{Lemma}
\newtheorem{corollary}[theorem]{Corollary}

\newtheorem{definition}[theorem]{Definition}
\newtheorem{remark}[theorem]{Remark}

\newtheorem{example}[theorem]{Example}

\tcolorboxenvironment{lemma}{blanker,before skip=10pt,after skip=10pt}

\newcommand\E{\mathbb{E}}
\newcommand\R{\mathbb{R}}
\newcommand\Z{\mathbb{Z}}
\newcommand\N{\mathbb{N}}
\newcommand\C{\mathbb{C}}

\newcommand\n{\mathrm{n}}
\newcommand\x{\mathrm{x}}
\newcommand\Poly{\mathrm{Poly}}

\newcommand\diam{\mathrm{diam}}
\newcommand\eps{\varepsilon}
\newcommand\m{\mathbf{m}}
\renewcommand\mod{\ \mathrm{mod}\ }

\renewcommand\P{\mathbb{P}}

\begin{document}

\title[Higher uniformity of multiplicative functions]{Higher uniformity of bounded multiplicative functions in short intervals on average}

\author[Matom\"aki]{Kaisa Matom\"aki}
\address{Department of Mathematics and Statistics \\
University of Turku, 20014 Turku\\
Finland}
\email{ksmato@utu.fi}

\author[Radziwi{\l}{\l}]{Maksym Radziwi{\l}{\l}}
\address{ Department of Mathematics,
  Caltech, 
  1200 E California Blvd,
  Pasadena, CA, 91125 \\
	USA}
\email{maksym.radziwill@gmail.com}

\author[Tao]{Terence Tao}
\address{Department of Mathematics, UCLA\\
405 Hilgard Ave\\
Los Angeles, CA, 90095\\
USA}
\email{tao@math.ucla.edu}

\author[Ter\"av\"ainen]{Joni Ter\"av\"ainen}
\address{Mathematical Institute, University of Oxford \\
Woodstock Road \\
Oxford OX2 6GG \\
United Kingdom}
\address{Department of Mathematics and Statistics \\
University of Turku, 20014 Turku\\
Finland}
\email{joni.p.teravainen@gmail.com}

\author[Ziegler]{Tamar Ziegler}
\address{Einstein Institute of Mathematics, Givaat Ram The Hebrew University of Jerusalem \\
Edmond J. Safra Campus \\
Jerusalem 91904 \\
Israel}
\email{tamarz@math.huji.ac.il}

\begin{abstract} Let $\lambda$ denote the Liouville function. We show that, as $X \rightarrow \infty$, 
  $$
  \int_{X}^{2X} \sup_{\substack{P(Y)\in \mathbb{R}[Y]\\\deg{P}\leq k}} \left | \sum_{x \leq n \leq x + H} \lambda(n) e(-P(n)) \right |\ dx = o ( X H) 
  $$
  for all fixed $k$ and $X^{\theta} \leq H \leq X$ with $0 < \theta < 1$ fixed but arbitrarily small. Previously this was only established for $k \leq 1$.  We obtain this result as a special case of the corresponding statement for (non-pretentious) $1$-bounded multiplicative functions that we prove.
	
	In fact, we are able to replace the polynomial phases $e(-P(n))$ by degree $k$ nilsequences $\overline{F}(g(n) \Gamma)$. By the inverse theory for the Gowers norms this implies the higher order asymptotic uniformity result
  $$
  \int_{X}^{2X} \| \lambda \|_{U^{k+1}([x,x+H])}\ dx = o ( X )
  $$
	in the same range of $H$.
	
We present applications of this result to patterns of various types in the Liouville sequence. Firstly, we show that the number of sign patterns of the Liouville function is superpolynomial, making progress on a conjecture of Sarnak about the Liouville sequence having positive entropy. Secondly, we obtain cancellation in averages of $\lambda$ over short polynomial progressions $(n+P_1(m),\ldots, n+P_k(m))$, which in the case of linear polynomials yields a new averaged version of Chowla's conjecture.

We are in fact able to prove our results on polynomial phases in the wider range $H\geq \exp((\log X)^{5/8+\varepsilon})$, thus strengthening also previous work on the Fourier uniformity of the Liouville function. 

\end{abstract}

\maketitle
\setcounter{tocdepth}{1}
\tableofcontents

\section{Introduction}

Let $\lambda \colon \N \to \{-1,+1\}$ denote the Liouville function, that is to say the completely multiplicative function with $\lambda(p)=-1$ for all primes $p$; we extend $\lambda$ by zero to the integers.  In~\cite{mr} it was shown that this function exhibited cancellation on almost all short intervals $[x,x+H]$ in the sense that\footnote{See Section~\ref{notation-sec} for our asymptotic notation conventions.}
\begin{equation}\label{mr-eq}
 \int_X^{2X} \left | \sum_{x \leq n \leq x + H} \lambda(n) \right |\ dx = o ( H X )
\end{equation}
as $X \to \infty$, whenever $H = H(X)$ was a function of $X$ that went to infinity as $X \to \infty$; see also~\cite{mr-p} for a simpler proof of \eqref{mr-eq} in the case of ``polynomially large intervals'', in which $H = X^\theta$ for a fixed $0 < \theta < 1$.  In~\cite{mr}, \cite{mr-p} the qualitative gain $o(HX)$ over the trivial bound $O(HX)$ was improved to a more quantitative bound, but in this paper we will focus only on qualitative estimates.  The bounds for $\lambda$ also extend to the closely related M\"obius function $\mu$, but for the sake of discussion we shall restrict attention initially to the Liouville function $\lambda$.

In~\cite{MRT} the estimate \eqref{mr-eq} was generalized to
\begin{equation}\label{mrq-eq}
\sup_{\alpha \in \R} \int_X^{2X} \left | \sum_{x \leq n \leq x + H} \lambda(n) e(-\alpha n) \right |\ dx = o ( H X )
\end{equation}
as $X \to \infty$, for any $H = H(X)$ that went to infinity as $X \to \infty$, where we adopt the usual notation $e(\alpha) \coloneqq e^{2\pi i \alpha}$.  Informally, this asserts that $\lambda$ does not asymptotically exhibit any correlation with a \emph{fixed} linear phase $n \mapsto e(\alpha n)$ in short intervals on average.  The question was then raised in~\cite[Section 4]{Tao} as to whether the stronger estimate
\begin{equation}\label{mrq-conj}
\int_X^{2X} \sup_{\alpha \in \R} \left | \sum_{x \leq n \leq x + H} \lambda(n) e(-\alpha n) \right |\ dx = o ( H X )
\end{equation}
could be established.  This is not known unconditionally, although as observed in~\cite{TaoEq} it can be deduced from the Chowla conjecture~\cite{Chowla-book}.  However, in a recent paper~\cite{mrt-fourier} the bound \eqref{mrq-conj} was established in the regime where $H = X^\theta$ for a fixed $0 < \theta < 1$; the case $\theta > 5/8$ without needing the $x$-average was previously established by Zhan in~\cite{Zhan} (and Zhan's result was recently improved to $\theta > 3/5$ in~\cite{MatoTera}).

For any non-negative integer $k \geq 0$, any interval $[x,x+H]$, and any function $f \colon \Z \to \C$, define the \emph{weak Gowers uniformity norm}
\begin{equation}\label{wgow}
 \| f \|_{u^{k+1}([x,x+H])} \coloneqq \sup_{P \in \mathrm{Poly}_{\leq k}(\R \to \R)} \frac{1}{H} \left| \sum_{x \leq n \leq x+H} \lambda(n) e(-P(n)) \right|
\end{equation}
where $\mathrm{Poly}_{\leq k}(\R \to \R)$ is the $k+1$-dimensional vector space of polynomial maps\footnote{In the sum in \eqref{wgow}, only the values of $P$ on the integers $\Z$ are relevant, but in our later analysis it will be convenient to evaluate such polynomials at non-integer values as well.} $P \colon \R \to \R$ of degree at most $k$.   This norm is indeed much weaker than the usual Gowers norm, in the sense that it is well known (see~\cite[\S 4]{gowers-long-aps}) that it does not control linear patterns of complexity $\geq 2$. Nevertheless, we will need the weak Gowers uniformity result in Theorem~\ref{mult-poly} below in order to establish the strong Gowers uniformity result in Theorem~\ref{mult-pret}.

The result in~\cite{mrt-fourier} is then equivalent to the bound
$$ \int_X^{2X} \| \lambda \|_{u^2([x,x+H])}\ dx = o(X)$$
as $X \to \infty$, with $H = X^\theta$ for a fixed $0 < \theta < 1$; the corresponding (and weaker) bound for the $u^1$ norm follows from the earlier result in~\cite{mr} or~\cite{mr-p}.  Our first main result extends these bounds to higher orders of uniformity:

\begin{corollary}[Liouville does not correlate with polynomial phases on short intervals on average]\label{lam-poly} Let $k \geq 0$ be a non-negative integer, and let $0 < \theta < 1$ be fixed.  Then we have
\begin{align}\label{eqtn1}
\int_X^{2X} \| \lambda \|_{u^{k+1}([x,x+H])}\ dx = o(X)
\end{align}
as $X \to \infty$, where $H \coloneqq X^\theta$.
\end{corollary}
\begin{remark}
In Theorem~\ref{lower} below we show that Corollary~\ref{lam-poly} holds for $H$ as small as $\exp((\log X)^{5/8 + \varepsilon})$ for any fixed $\varepsilon > 0$. 
\end{remark}

We remark that previously this was known in the $k\geq 1$ cases for $\theta>2/3$ by~\cite[Theorem 1.4]{matomaki-shao}.  In fact, in this regime a uniform bound $\sup_{x \in [X,2X]} \| \lambda \|_{u^{k+1}([x,x+H])} = o(1)$ is established.  It is natural to conjecture that such uniform bounds extend to all $\theta>0$, but this seems well beyond the reach of the methods in this paper.

In fact (as in~\cite{mrt-fourier}), we can generalize Corollary~\ref{lam-poly} to the case where the Liouville function $\lambda$ is replaced by a more general ``non-pretentious'' $1$-bounded multiplicative function.  Recall that a multiplicative function $f \colon \N \to \C$ is said to be $1$-bounded if $|f(n)| \leq 1$ for all $n \in \N$.  To motivate the ``non-pretentiousness'' hypothesis, we consider (as was done in~\cite{mrt-fourier} in the $k=1$ case) the character
\begin{equation}\label{fchar}
 f(n) \coloneqq n^{it} \chi(n),
\end{equation}
formed by multiplying an ``Archimedean character'' $n \mapsto n^{it}$ for some real number $t$ with $|t| \leq \eps X^{k+1} / H^{k+1}$ for some small $\eps>0$, and a Dirichlet character $\chi$ of some bounded conductor $q$. Observe that $f$ is completely multiplicative and $1$-bounded, and a Taylor expansion with remainder of the phase $n \mapsto \frac{t}{2\pi} \log n$ of the Archimedean character $n^{it} = e(\frac{t}{2\pi} \log n)$ around a given point $x \in [X,2X]$ yields a decomposition of the form
\begin{equation}\label{fapprox}
 n^{it} = e(P_x(n)) + O_k(\eps)
 \end{equation}
for all $n \in [x,x+H]$ and some polynomial $P_x \in \mathrm{Poly}_{\leq k}(\R \to \R)$ depending on $x$ (and $t$). This together with the $q$-periodicity of $\chi$ can be used to imply that
$$ \int_X^{2X} \| f \|_{u^{k+1}([x,x+H])}\ dx \gg_{k,q} X$$
if $1 \leq H \leq X$ are sufficiently large.

Our next result asserts that this is essentially the only obstruction to extending Corollary~\ref{lam-poly} to more general $1$-bounded multiplicative functions.  Following Granville and Soundararajan~\cite{GS}, we define the distance function
$$
\mathbb{D}(f,g;X):=\Big(\sum_{p \leq X} \frac{1 - \mathrm{Re}(f(p)\overline{g(p)})}{p}\Big)^{1/2},
$$
and further define the quantity
$$
  M(f; X, Q) \coloneqq \inf_{|t|\leq X}\,\inf_{\substack{\chi \mod{q} \\ q \leq Q }} \mathbb{D}(f,n\mapsto \chi(n)n^{it};X).
$$
Informally, $M(f; X, Q)$ is small whenever $f$ is close to a function of the form \eqref{fchar} with $|t| \leq X$ and $\chi$ of conductor at most $Q$.  We then have

\begin{theorem}[Non-pretentious multiplicative functions do not correlate with polynomial phases on short intervals on average]\label{mult-poly}  Let $k \geq 0$ be a non-negative integer, and let $0 < \theta < 1/2$.  Suppose that $f\colon \N \to \C$ is a multiplicative $1$-bounded function, and suppose that $X\geq 1$, $X^\theta \leq H \leq X^{1-\theta}$, and $\eta > 0$ are such that
$$ \int_X^{2X} \| f \|_{u^{k+1}([x,x+H])}\ dx \geq \eta X.$$
Then one has
$$ M(f; C X^{k+1} / H^{k+1}, Q ) \ll_{k,\eta,\theta} 1$$
for some $Q, C \ll_{k,\eta,\theta} 1$.
\end{theorem}

The upper bound $H \leq X^{1-\theta}$ here is for minor technical reasons and it is likely that one can replace it with $H \leq X$; however our main interest is in the opposite regime when $H$ is as small as possible.
Standard calculations regarding the ``non-pretentious'' nature of the Liouville function (using the Vinogradov--Korobov zero-free region for $L$-functions) allow one to deduce Corollary~\ref{lam-poly} from Theorem~\ref{mult-poly}; see for instance~\cite[(1.12)]{MRT}.  The $k=0$ case of this theorem follows from the results in~\cite{mr}, and the $k=1$ case is established\footnote{In that paper the constant $C$ appearing in the above theorem was worsened to $H^\rho$ for some arbitrarily small constant $\rho>0$, but we have found a way to modify the arguments to eliminate that power loss in this result. In fact, it will be important in the induction arguments used to establish Theorem~\ref{mult-pret} below that such losses are avoided.} in~\cite[Theorem 1.4]{mrt-fourier}. Our focus here shall accordingly be on the higher order case $k \geq 2$, which we will establish by generalizing the techniques in~\cite{mrt-fourier} to the polynomial phase setting (and in fact further to nilsequences, which are needed in proving our Theorem~\ref{mult-pret} on genuine Gowers norms of multiplicative functions).

As a corollary of Theorem~\ref{mult-poly} and the decomposition \eqref{fapprox} we can also control the correlation of non-pretentious multiplicative functions with Archimedean characters on short intervals on average:

\begin{corollary}[Non-pretentious multiplicative functions do not correlate with Archimedean characters  on short intervals on average]\label{mult-poly-it}  Let $k \geq 0$ be a non-negative integer, and let $0 < \theta < 1/2$.  Suppose that $f\colon \N \to \C$ is a multiplicative $1$-bounded function, and suppose that $X\geq 1$, $\varepsilon>0$, $X^\theta \leq H \leq X^{1-\theta}$, and $\eta > 0$ are such that
$$ \int_X^{2X}  \sup_{|t| \le \varepsilon X^{k+1}/H^{k+1}}\left| \sum_{x\leq  n \leq x+H} f(n)n^{it}\right|\, dx \geq  \eta HX.$$
Then one has
$$ M(f; C X^{k+1} / H^{k+1}, Q ) \ll_{k,\eta, \varepsilon, \theta} 1$$
for some $Q, C \ll_{k,\eta,\varepsilon, \theta} 1$.
\end{corollary}

We also note that He and Wang~\cite{HeWang} recently proved that 
\begin{align*}
\sup_{P\in \mathrm{Poly}_{\leq k}(\mathbb{R}\to \mathbb{R})}\int_{X}^{2X}\left|\sum_{x\leq n\leq x+H}\lambda(n)e(-P(n))\right|\, dx=o(HX)   
\end{align*}
for any $H$ tending to infinity with $X$, and they also proved an analogous estimate for nilsequences. This statement with the supremum \emph{outside} the integral unfortunately does not lead to control on Gowers norms (or weak Gowers norms) of $\lambda$ over short intervals and is accordingly closer in spirit to~\cite{MRT} than to the current paper. It is the case with the supremum \emph{inside} the integral (as in Theorems~\ref{mult-poly} and~\ref{mult-nil}) that we need for the applications in this paper, and such estimates would lead to a proof of the logarithmically averaged Chowla and Sarnak conjectures (via~\cite[Theorem 1.8]{TaoEq}) if one was able to take the interval length $H$ to grow sufficiently slowly in them; see Proposition~\ref{entropy}.

As indicated above, we can strengthen Theorem~\ref{mult-poly} further.  For any non-negative integer $k \geq 0$, and any function $f \colon \Z \to \C$ with finite support, define the (unnormalized) Gowers uniformity norm
$$ \| f \|_{U^{k+1}(\Z)} \coloneqq \left( \sum_{y,h_1,\dots,h_{k+1} \in \Z} \prod_{\omega \in \{0,1\}^{k+1}} \mathcal{C}^{|\omega|} f(y+\omega_1 h_1+\dots+\omega_{k+1} h_{k+1}) \right)^{1/2^{k+1}}$$
where $\omega = (\omega_1,\dots,\omega_{k+1})$, $|\omega| \coloneqq \omega_1+\dots+\omega_{k+1}$, and $\mathcal{C} \colon z \mapsto \overline{z}$ is the complex conjugation map. Then for any interval $[x,x+H]$ with $H \geq 1$ and any $f \colon \Z \to \C$ (not necessarily of finite support), define the \emph{Gowers uniformity norm over} $[x,x+H]$ by 
\begin{equation}\label{gow}
 \| f \|_{U^{k+1}([x,x+H])} \coloneqq \| f 1_{[x,x+H]} \|_{U^{k+1}(\Z)} /  \| 1_{[x,x+H]} \|_{U^{k+1}(\Z)}
\end{equation}
where $1_{[x,x+H]} \colon \Z \to \C$ is the indicator function of $[x,x+H]$.  We then have

\begin{theorem}[Non-pretentious multiplicative functions are Gowers uniform on short intervals on average]\label{mult-pret}  Let $k \geq 0$ be a non-negative integer, and let $0 < \theta < 1/2$.  Suppose that $f\colon \N \to \C$ is a multiplicative $1$-bounded function (extended by zero to the remaining integers), and suppose that $X \geq 1$, $X^\theta \leq H \leq X^{1-\theta}$, and $\eta > 0$ are such that
$$ \int_X^{2X} \| f \|_{U^{k+1}([x,x+H])}\ dx \geq \eta X.$$
Then one has
\begin{align}\label{MfC}
 M(f; C X^{k+1} / H^{k+1}, Q ) \ll_{k,\eta,\theta} 1
\end{align}
for some $Q, C \ll_{k,\eta,\theta} 1$.
\end{theorem}

The corresponding statement on correlations of $f$ with nilsequences $n \mapsto F(g(n)\Gamma)$ on intervals $[x,x+H]$, which we will use to derive Theorem~\ref{mult-pret} (and which in fact is equivalent to it), is given as Theorem~\ref{mult-nil}. 

In particular, using the non-pretentious nature of the Liouville function, this theorem yields the following corollary.
\begin{corollary}[Gowers uniformity of Liouville on short intervals on average]\label{cor: mult-pret}
Let an integer $k\geq 0$ and $0 < \theta \leq 1$ be fixed. Then for $H\geq X^{\theta}$ we have
\begin{equation}\label{fox}
 \int_X^{2X} \| \lambda \|_{U^{k+1}([x,x+H])}\ dx = o(X).
\end{equation}
\end{corollary}
Note that in the corollary above the case of larger values of $H\geq X^{1-o(1)}$ follows from the case $H=X^{\theta}$ by a simple averaging argument (by first using the inverse theorem for the Gowers norms to express \eqref{fox} in terms of the correlation of $\lambda$ with nilsequences on $[x,x+H]$, and then partitioning this interval into subintervals of length $\asymp X^{1-\varepsilon}$).  This partially verifies~\cite[Conjecture 1.6]{TaoEq}, which asserted that this estimate (or more precisely, a slightly weaker logarithmically averaged version of this estimate) held whenever $H = H(X)$ went to infinity as $X \to \infty$.  Fully resolving this conjecture would have many implications, including the (logarithmically averaged) Chowla and Sarnak conjectures; see~\cite{Tao}, \cite{TaoTeravainenGeneral} and~\cite{fh-sarnak} for the best currently known results in this direction).  Correspondingly, the partial result \eqref{fox} allows us to make progress on some problems concerning the Liouville function, including its word complexity and an averaged version of Chowla's conjecture, which we discuss in Subsection~\ref{sub:apps}. 

Regarding previous results on Gowers norms of non-pretentious multiplicative functions, a result of Frantzikinakis and Host~\cite{FH-Fourier} (generalizing work of Green and Tao~\cite{gt-mobius}) establishes the ``long sum'' endpoint case of Theorem~\ref{mult-pret} (corresponding to the case $H=X$, which is strictly speaking not covered by the above theorem), showing that $\|f\|_{U^{k+1}[1,X]}=o(1)$ under the assumption that $\mathbb{D}(f,n\mapsto \chi(n)n^{it};X)\to \infty$ as $X\to \infty$ for any fixed real number $t$ and Dirichlet character $\chi$. 

It is not difficult to establish a general estimate of the form
$$ \| f \|_{u^{k+1}([x,x+H])} \ll_k \| f \|_{U^{k+1}([x,x+H])}$$
for any $f\colon \Z \to \C$; this can be established for instance by a minor modification of the arguments in~\cite[\S 11.2]{tao-vu}. Thus Theorem~\ref{mult-pret} implies Theorem~\ref{mult-poly}.  The converse implication is also routine for $k=0,1$, but as is now well known (see e.g.,~\cite[Proposition 11.8]{tao-vu}), for higher $k$ the polynomial phases $n \mapsto e(P(n))$ appearing in the definition of the weak Gowers norms \eqref{wgow} are insufficient to control the full Gowers norms \eqref{gow}.  To bridge the gap, one needs to replace these polynomial phases by more general \emph{nilsequences} $n \mapsto F(g(n) \Gamma)$. The polynomial phases correspond to nilsequences on filtered nilmanifolds $G/\Gamma$ with $G$ abelian. We will thus first prove Theorem~\ref{mult-poly} in Section~\ref{sec: polyphases} to treat the case of abelian $G$, and then use a 
different and more delicate argument (presented in Section~\ref{nilseq} and outlined in Subsection~\ref{sub: sketch}) to handle the non-abelian case.

\subsection{Connection with the Chowla and Sarnak conjectures}

As already mentioned, estimates such as \eqref{fox} with slowly growing $H$ are closely tied to the Chowla and Sarnak conjectures. The logarithmically averaged Chowla conjecture states that
\begin{align*}
\sum_{n\leq x}\frac{\lambda(a_1n+b_1)\cdots \lambda(a_kn+b_k)}{n}=o(\log x)    
\end{align*}
whenever $a_i,b_i$ are natural numbers\footnote{In this paper the natural numbers $\N = \{1,2,3,\dots\}$ begin with $1$.} with $a_ib_j\neq a_jb_i$ for $i\neq j$. The logarithmically averaged Sarnak conjecture in turn is the statement that
\begin{align*}
\sum_{n\leq x}\frac{\lambda(n)a(n)}{n}=o(\log x)
\end{align*}
for every bounded, deterministic sequence $a \colon \N \to \C$ (in the sense that $a$ has zero topological entropy). See~\cite{sarnaksurvey} for a survey of previous work on these two conjectures.

In~\cite{TaoEq}, it was shown that the logarithmically averaged Chowla conjecture and the logarithmically averaged Sarnak conjecture are equivalent, and that both would also follow from \eqref{fox} being true for every $H=H(X)$ tending to infinity with $X$.  In fact these two conjectures are equivalent to the logarithmic version of \eqref{fox} in this regime, which states that
\begin{align}\label{log-liouville}
\int_{1}^{X}\frac{\|\lambda\|_{U^{k+1}[x,x+H]}}{x}\, dx=o(\log X)    
\end{align}
whenever $H = H(X)$ goes to infinity with $X$.  Thus, a potential strategy towards proving the logarithmic Chowla and Sarnak conjectures emerges from the possibility of lowering the value of $H=H(X)$ in Theorem~\ref{mult-pret}. We observe in Section~\ref{uss} that we in fact do not need \eqref{log-liouville} for arbitrarily slowly growing $H$ to deduce the logarithmic Chowla conjecture; it instead suffices to prove it for $H\geq (\log X)^{\eta}$ for every $\eta>0$.

\begin{proposition}\label{entropy} Suppose that for every natural number $k$ and every $\eta>0$ for $H=H(X)=(\log X)^{\eta}$ we have
\begin{align*}
\int_{1}^{X}\frac{\|\lambda\|_{U^{k+1}[x,x+H]}}{x}\,dx=o(\log X).    
\end{align*}
Then the logarithmic Chowla conjecture holds.
\end{proposition}

This proposition will be proved in Subsection~\ref{uss}.

Thus, in order to prove the logarithmic Chowla conjecture, it would suffice to bridge the gap between $H\geq X^{\eta}$ (which is the range where Corollary~\ref{cor: mult-pret} is valid) and $H\leq (\log X)^{\eta}$ in Proposition~\ref{entropy}. In Section~\ref{sec: lowering}, we already show that, at least in the case of our result on the weak Gowers norms (Theorem~\ref{mult-poly}), we may lower the admissible $H$ to $H\geq \exp((\log X)^{c})$ for some $c>0$.

\begin{theorem}[Shortening the intervals]\label{lower}
Let $k$ be a natural number, and let $\theta > 5/8$ and $\rho > 0$ be fixed.  Suppose that $f\colon \N \to \C$ is a multiplicative $1$-bounded function (extended by zero to the remaining integers), and suppose that $X\geq 1$, $X^{\theta} \geq H \geq \exp((\log X)^\theta) $, and $\eta > 0$ are such that
\begin{equation}
\label{eq:lowerMassump}
\int_X^{2X} \| f \|_{u^{k+1}([x,x+H])}\ dx \geq \eta X.
\end{equation}
Then one has
\begin{equation}
\label{eq:lowerMClaim} 
M(f; X^{k+1} / H^{k+1-\rho}, Q ) \ll_{k,\eta,\theta} 1
\end{equation}
for some $Q \ll_{k,\eta, \theta, \rho} 1$.
\end{theorem}

It is conceivable that a careful reworking of the nilsequence part of our arguments in Section~\ref{nilseq} would yield a similar regime $H\geq \exp((\log X)^{1-\delta})$ for Theorem~\ref{mult-pret}; we do not pursue this here, however (see however Remark~\ref{rem: loweringH}).

The exponent $5/8$ appearing in Theorem~\ref{lower} is significant as it shows that it is possible to control $\| f \|_{u^{k + 1}[x, x + H]}$ on average over $x$ without establishing cancellations in short sums over primes of the form $\sum_{H \leq p \leq 2H} p^{it}$ (with $t$ of size $X^{k}$). Instead, we show using general Dirichlet polynomial techniques that the set of points $t$ at which the above Dirichlet polynomial exhibits no cancellation is sparse. We note that the smallest $H$ for which $\sum_{H \leq p \leq 2H} p^{it}$ is known to exhibit cancellations for $t$ of size $X^k$ is $H = \exp((\log X)^{2/3 + \varepsilon})$. We also note that the proof of Theorem~\ref{lower} crucially relies on cancellation in short sums of multiplicative functions outside a power-saving exceptional set, proved in~\cite{mrPart2} as an improvement to~\cite{mr}. See Remark~\ref{rem:specShorter} on how in the case of $f=\lambda$ the weaker range $H\geq \exp((\log X)^{2/3+\varepsilon})$ and  can be obtained using only the method of~\cite{mr-p}.

It seems nonetheless that the lower bound for $H$ in Theorem~\ref{lower} is close to the breaking point of several arguments in our proof. Firstly, for $H$ much smaller than $\exp((\log X)^c)$ with $c > 0$ it appears difficult to show (using general Dirichlet polynomial techniques) that for a large proportion of values $|t| \leq X^{O(1)}$ the sum $\sum_{H \leq p \leq 2H} p^{it}$ exhibits cancellations. Secondly, in the graph-theoretic part of our arguments factors of the type $\ell!$ with $\ell \asymp \frac{\log X}{\log H}$ arise, and while these are harmless for $H\geq X^{\eta}$, they become problematic in the regime $H\leq \exp((\log X)^{\theta})$, in particular if $\theta < 1/2$. Despite these limitations, at least if one works with certain model cases of the problem (such as a ``99\% version'' of Theorem~\ref{mult-pret}, where $\eta$ is very close to $1$) and assumes GRH, then one should be able to push $H$ further down. 

Handling the regime $H\in [(\log X)^{\eta}, (\log X)^{\eta^{-1}}]$, at the very least, would likely necessitate an entirely new idea for several reasons. Firstly, even under GRH cancellation in the Dirichlet polynomials $\sum_{H^{\varepsilon}\leq p\leq 2H^{\varepsilon}}\chi(p)p^{it}$ is known essentially only for $H\gg_{\varepsilon,\kappa} (\log X)^{(2+\kappa)\varepsilon^{-1}}$. Secondly, the arguments for solving ``approximate functional equations'' involving phase functions that are used in this paper do not seem to work (even in model cases) for such $H$, as such arguments rely on the ``modulus'' $\prod_{H^{\varepsilon}\leq p\leq 2H^{\varepsilon}}p$ being much larger than $X$ (see footnote~\ref{foot1}). Thirdly, the entropy decrement argument (which is applied to prove Proposition~\ref{entropy} that the $H\geq (\log X)^{\eta}$ range of \eqref{log-liouville} implies the logarithmic Chowla conjecture) is restricted to the regime $H\leq (\log X)^{\eta}$, as it is based on equidistribution of the integers in $[1,X]$ modulo $\prod_{H^A\leq p\leq 2H^A}p$ for $A\geq 1$ large enough (see however the recent work~\cite{helfgott-radziwill} for a quantitatively stronger alternative replacement to the entropy decrement method in the case of two-point correlations).

\subsection{Applications}\label{sub:apps}

\subsubsection{Sign patterns of the Liouville function}

Let 
\begin{equation}\label{sk-def}
s(k) \coloneqq |\{v\in \{-1,+1\}^k:\,\, v=(\lambda(n+1),\ldots, \lambda(n+k))\,\,\textnormal{for some}\,\, n\in \mathbb{N}\}|    
\end{equation}
be the number of sign patterns of length $k$ in the Liouville sequence. A direct consequence of Chowla's conjecture (or its logarithmic version) is that $s(k)=2^k$ for all $k$ and that each pattern of length $k$ occurs with positive lower density; yet, this remains unknown (apart from the $k\leq 4$ cases handled in~\cite{TaoTeravainenGeneral}). In fact, known lower bounds on $s(k)$ are far from exponential; Frantzikinakis and Host~\cite{fh-sarnak} proved that $s(k)/k\to \infty$ as $k\to \infty$, and recently this was improved by McNamara~\cite{mcnamara} to $s(k)\gg k^2$. In fact, both in~\cite{mcnamara} and~\cite{fh-sarnak} a stronger result was proved, namely that $\lambda$ is orthogonal (with logarithmic averages) to any sequence having $o(k^2)$ (respectively $O(k)$) sign patterns of length $k$. Let us also remark that the validity of the $2j$-point Chowla conjecture for any fixed $j$ implies by a simple moment computation that there are $\gg k^j$ sign patterns of length $k$ that occur with positive density (so, in particular, $s(k)\gg k^j$). As an application of Theorem~\ref{mult-pret}, we prove a superpolynomial lower bound on $s(k)$.

\begin{theorem}[The Liouville function has superpolynomial number of patterns]\label{superpolynomial} We have $s(k)\gg_A k^A$ for every $A\geq 1$.
\end{theorem}

In fact, we prove a more general result (Theorem~\ref{theo-signs}), which shows that any improvement in the range of validity of \eqref{fox} leads to an improvement in the lower bound on $s(k)$. In particular, if \eqref{fox} holds for $H\geq \exp((\log X)^{1-\delta})$, then $s(k)\gg_{\varepsilon} k^{(\log k)^{\delta/(1-\delta)-\varepsilon}}$. See also Theorem~\ref{theo-multsigns} for a generalization to multiplicative functions other than the Liouville function. 

Theorem~\ref{superpolynomial} can be viewed as progress towards a conjecture of Sarnak in~\cite{sarnak} that the Furstenberg systems of the Liouville function have positive entropy (so that in particular $s(k)\gg c^k$ for some $c>1$). Sarnak highlighted this as a key special case of his M\"obius randomness conjecture. It is worth noting that, as was observed in~\cite{sarnak}, one easily sees that the M\"obius system has positive entropy, but this amounts solely to the fact that the distribution of squarefree numbers is very well understood and therefore this does not imply anything about the Liouville system (indeed, Sarnak says in~\cite{sarnak} that it appears ``pretty hard to show that $\lambda$ is not deterministic''). In this connection, it would be very interesting to say more about the frequency of the superpolynomially many patterns produced by Theorem~\ref{superpolynomial}.

The proof of Theorem~\ref{superpolynomial} involves a different approach than the previous sign pattern arguments, utilizing a type of ``structure and randomness'' dichotomy (meaning that if there are few sign patterns, then the Liouville function is easier to understand, and we can leverage this to eventually get a contradiction); see Section~\ref{sec: signpatterns} for the proof and Subsection~\ref{sss:pfSignPatSkecth} for its outline.

\subsubsection{Polynomial averages of the Liouville function}

As another application of Theorem~\ref{mult-pret}, we use it to establish cancellation in averages 
\begin{align*}
\mathbb{E}_{n\leq X}\mathbb{E}_{m\leq X^{1/d}}\lambda(n+P_1(m))\cdots \lambda(n+P_k(m))    
\end{align*}
of the Liouville function along polynomial progressions $(n+P_1(m),\ldots, n+P_k(m))$ (with $d$ being the maximum degree of the polynomials $P_i$). Averages along polynomial progressions are natural objects in additive combinatorics and ergodic theory, and a particularly important result concerning them is the polynomial Szemer\'edi theorem of Bergelson and Leibman~\cite{bl} that guarantees for any non-constant polynomials $P_i(x)\in \mathbb{Z}[x]$ with $P_i(0)=0$ the existence of a polynomial progression $n+P_1(m),\ldots, n+P_k(m)$ inside any positive density subset of the integers. This was generalized to polynomial progressions inside the primes in~\cite{tao-ziegler1}. However, when one is considering polynomial progressions weighted by an oscillating function (such as $\lambda$), these results do not apply (as they are lower bound results). 

It was later shown in~\cite[Theorem 1.4]{tao-ziegler2} that if the assumption $P_i(0)=0$ for all $i$ is replaced with the polynomials $P_i-P_j$ having degree $d$ for all $i\neq j$ (where $d$ is the maximum of the degrees of $P_l$) one has an asymptotic for polynomial patterns $(n+P_1(m),\ldots, n+P_k(m))$ weighted by the von Mangoldt function (and the same argument works for the Liouville function). Here we remove this assumption on the degree $d$ coefficients of the $P_i$ being distinct in the case of the Liouville weight, thus obtaining a result that works for any polynomial patterns (that are not of ``infinite complexity'', such as the pattern $(n+1,n+2,\ldots, n+k)$). Moreover, we can take the $m$ average in our results to be of subpolynomial size, which is important for Corollary~\ref{cor_chowla} below.  

\begin{theorem}[Polynomial averages of the Liouville function]\label{poly1}
Let $k,r\geq 1$ be integers, and let $P_1,\ldots, P_k$ be polynomials in $\mathbb{Z}[x_1,\ldots, x_r]$ with degrees $\leq d$. Suppose that $P_i-P_j$ is nonconstant for all $i\neq j$. Then for any fixed $0<\varepsilon< 1/d$ we have
\begin{align*}
\mathbb{E}_{\m\in [X^{\varepsilon}]^r}|\mathbb{E}_{n\leq X}\lambda(n+P_1(\m))\cdots \lambda(n+P_k(\m))|=o(1).    
\end{align*}
Here, $[N]^{r}$ stands for the $r$-dimensional discrete box $\{1,\ldots, N\}^r$.
\end{theorem}

Specializing to linear polynomials, the following result on Chowla's conjecture with a short one-variable average is an immediate corollary (in fact, this corollary could also be obtained more directly from our Gowers uniformity result, Corollary~\ref{cor: mult-pret}; see footnote~\ref{foot:polypattern}). 

\begin{corollary}[Chowla's conjecture with a short average]\label{cor_chowla}
Let $k\geq 1$ be an integer, and let $a_1,\ldots, a_k\geq 0$ be distinct. Let $\varepsilon>0$ be arbitrary. Then we have
\begin{align*}
\mathbb{E}_{h\leq X^{\varepsilon}}|\mathbb{E}_{n\leq X}\lambda(n+a_1h)\cdots \lambda(n+a_kh)|=o(1).    
\end{align*}
\end{corollary}

We remark that the Theorem~\ref{poly1} (and hence Corollary~\ref{cor_chowla}) continues to hold, with essentially the same proof, if $k-1$ of the $k$ occurrences of $\lambda$ in the correlation average are replaced with arbitrary fixed $1$-bounded sequences.

Taking $h$ bounded in Corollary~\ref{cor_chowla} would amount to settling Chowla's conjecture. Previously, the result of Corollary~\ref{cor_chowla} was only known for $k\leq 2$ (using the main result of~\cite{MRT}), and for $k=3$ without the absolute values around the $n$ average (using~\cite[Corollary 1.5]{mrt-fourier}). Note that for $k\geq 3$ the averaged Chowla conjecture of~\cite{MRT}  is not applicable in the setting above, since that result requires averaging over $k-1$ independent short variables. 

We can also prove an asymptotic similar to the one in Theorem~\ref{poly1} for averages of the von Mangoldt function if one of the terms in the progression is assigned the Liouville weight (but perhaps surprisingly the proof does not apply if the weight $\lambda$ is replaced with the constant weight $1$).

\begin{theorem}[Polynomial averages of the von Mangoldt function with Liouville twist]\label{poly2}
Let $k,r\geq 1$ be integers, and let $P_1,\ldots, P_k$ be non-constant  polynomials in $\mathbb{Z}[x_1,\ldots, x_r]$ with degrees $\leq d$. Suppose that $P_i-P_j$ is nonconstant for all $i\neq j$. Let $\Lambda$ be the von Mangoldt function. Then for any fixed $0<\varepsilon< 1/d$ we have
\begin{align*}
\mathbb{E}_{\m\in [X^{\varepsilon}]^r}|\mathbb{E}_{n\leq X} \lambda(n+P_1(\m)) \Lambda(n+P_2(\m)) \dotsm \Lambda(n+P_k(\m))|=o(1).    
\end{align*}
\end{theorem}

We remark that the theorem continues to hold, with essentially the same proof, when the occurrences of $\Lambda$ in the correlation average are replaced with arbitrary fixed sequences that are bounded by $\Lambda$ in modulus.

These results will be established in Section~\ref{sec:pattern}.

\subsection{Overview of proofs}\label{sub: sketch}

\subsubsection{The higher order uniformity theorem}

Let us outline the proof of Corollary~\ref{cor: mult-pret}; the proof of the more general Theorem~\ref{mult-pret} follows along similar lines. By the inverse theorem for the Gowers norms,  Corollary~\ref{cor: mult-pret} is equivalent to a discorrelation estimate between the Liouville function and nilsequences; more precisely
\begin{equation}\label{discor}
 \int_X^{2X} \sup_{g \in \Poly( \Z \to G )} \left| \sum_{n \in [x,x+H]} \lambda(n) \overline{F}(g(n) \Gamma)\right| \ dx =o(H X),
\end{equation}
where $G/\Gamma$ is any fixed\footnote{We note that the notion of ``complexity'' of nilmanifolds plays no role in this paper, unlike in e.g.~\cite{gt-mobius}, since the inverse theorem supplies us with a single nilmanifold $G_{\eta}/\Gamma_{\eta}$ such that $\|f\|_{U^k[N]}\geq \eta$ with $f$ $1$-bounded implies that $f$ correlates with a nilsequence on $G_{\eta}/\Gamma_{\eta}$.} degree $k$ filtered nilmanifold, $F:G/\Gamma\to \mathbb{C}$ is any fixed Lipschitz function, and the supremum is over all polynomial sequences $g(n)$ taking values in the Lie group $G$ (for all the relevant definitions and for the precise statement, see Section~\ref{nilseq}).

By using an induction on the dimension of $G$ we may assume that the function $F$ is ``irreducible'' in a certain technical sense, which roughly means that the nilsequences $n \mapsto F(g(n) \Gamma)$ are ``orthogonal'' to all lower dimensional nilsequences. We split the proof of this estimate \eqref{discor} into two cases that are analyzed separately, the case of abelian $G$ and the case of non-abelian $G$.

For abelian $G$, the nilsequences that arise on the filtered nilmanifold $G/\Gamma$ are (Lipschitz functions of) polynomial phases $n\mapsto e(P(n))$ with $\deg(P)\leq k$, so this case reduces to the polynomial phase case. This case is handled in Section~\ref{sec: polyphases} and is already sufficient for proving Corollary~\ref{lam-poly}. Here the task is to establish structure in phase functions $P_x\in \Poly_{\leq k}( \Z \to \R)$ satisfying
\begin{align}\label{eq12}
\left| \sum_{n \in [x,x+H]} \lambda(n) e(P_x(n))\right| \ dx \gg H    
\end{align}
for $\gg X$ choices of $x\in [X,2X]\cap\mathbb{Z}$, and eventually to exploit that structure to show that such functions do not exist. In order to talk about polynomials being equal up to negligible contributions, we introduce an equivalence relation on them; in this sketch, we say that $P_x\sim Q_x$ if\footnote{The actual equivalence relation used in Section~\ref{sec: polyphases} is slightly more elaborate; it also allows for a factor $\gamma(n)$ which is a rational polynomial. To show ideas, let us work with this slightly simpler equivalence in which we allow the ``Archimedean'' error $\varepsilon$ but not the ``non-Archimedean'' error $\gamma$.} $P_x(n)=\varepsilon(n)Q_x(n)$ holds on the underlying interval $[x,x+H]$ for some polynomial $\varepsilon(n)$ which is ``smooth'' in the sense that $|\varepsilon^{(\ell)}(n)|\ll H^{-\ell}$ for all $\ell\leq k$. Note that if we can show that
\begin{align}\label{eq14}
e(P_x(n))\approx e\left(\frac{T}{2\pi}\log n+\gamma(n)\right),\quad n\in [x,x+H],    
\end{align}
with $\gamma\in \Poly_{\leq k}(\R\to \R)$ being a $O(1)$-integral polynomial (that is, it maps from $q\Z$ to $\Z$ for some $q=O(1)$) and $T$ independent of $x$ and of polynomial size in $X$,
then $e(P_x(n))$ is essentially a twist of the Archimedean character $n\mapsto n^{iT}$, so we can use the results from~\cite{mr}, \cite{MRT} to obtain the desired contradiction. 

As in the linear phase case handled in~\cite{mrt-fourier}, we begin by establishing an ``approximate functional equation''\footnote{Our use of the term approximate functional equation of course differs from its meaning in the context of $L$-functions.} for the polynomial function $P_x(n)$ in \eqref{eq12}. Note that if $p\leq H^{\varepsilon}$ is a prime, then if $\lambda$ correlates with $e(P_x(n))$ on $[x,x+H]$, then $\lambda$ correlates with $e(P_{x}(pn))$ on $[x/p, (x+H)/p]$, for ``most'' choices of $p$ (this is a standard Tur\'an--Kubilius argument; see Proposition~\ref{Scaling-down}). Similarly, for ``most'' $y\in [X,2X]$ and primes $q\leq H^{\varepsilon}$, we must have that $\lambda$ correlates with $e(P_{y}(qn))$ on $[y/q, (y+H)/q]$. Now, if $|x/p- y/q|\leq H/(2\max\{p,q\})$, then the intervals $[x/p, (x+H)/p], [y/q,(y+H)/q]$ have large intersection, and since by the large sieve for polynomial phases (Proposition~\ref{lsieve}) there can only be boundedly many polynomial phases that $\lambda$ correlates with on an interval, we can say that
\begin{align*}
e(P_{x}(pn))\approx e(P_{y}(qn)),\quad n\in [x/p,(x+H)/p]
\end{align*}
for ``most'' $p,q\in [P,2P]$ and $x,y\in [X,2X]$ with $x/p=y/q+O(H/P)$ for some $P\leq H^{\varepsilon}$. This corresponds to an approximate equality of polynomials modulo $1$, but using a suitable version of the Chinese remainder theorem (Proposition~\ref{chinese}), and shifting $P_x, P_y$ by integer amounts, which is always allowable, we can eventually upgrade this to an equality modulo the product $\prod_{p'\in \mathcal{P}}p'$, where $\mathcal{P}$ is a ``large'' set of primes in $[P,2P]$, and thus with our choice of $H$ the modulus is enormous compared to $X$, so we can essentially treat this as a genuine equality in $\mathbb{R}$. In this way, we can essentially pass to the approximate functional equation
\begin{align}\label{eq15}
P_{x}(pn)\sim P_y(qn)    
\end{align}
for ``most'' $p,q\in [P,2P]$ and $x,y\in [X,2X]$ with $x/p= y/q+O(H/P)$.

 If we now form a graph $\mathcal{G}$ on $[X,2X]\cap \mathbb{Z}$ by connecting $x,y$ whenever $x/p=y/q+O(H/P)$ and $x,y,p,q$ are satisfying the above conditions, we obtain a graph whose structure governs the solutions to \eqref{eq15}. In particular, (a known case of) Sidorenko's conjecture tells us that $\mathcal{G}$ contains many configurations $\mathcal{C}$ consisting of two $\ell$-cycles and an edge between them, for $\ell> \log X/\log P$. When we unwrap what this means in terms of approximate functional equations, we obtain (in Proposition~\ref{prop41-analog}) the approximate dilation invariance
\begin{align}\label{eq13}
P_x(a_xn)\sim P_x(b_x n)    
\end{align}
for many pairs $(a_x,b_x)$ that are of polynomial size in $X$ (more precisely, products of $\ell$ primes from $[P,2P]$), and relatively close to each other (with $\frac{a_x-b_x}{a_x} \asymp \frac{H}{X}$).

We then ``solve'' the approximate equation \eqref{eq13} using properties of the underlying polynomial algebra, with the conclusion that 
$P_x$ must locally ``pretend'' to be a character:
\begin{align*}
e(P_x(n))\approx e\left(\frac{T_x}{2\pi}\log n+\gamma_x(n)\right),
\end{align*}
where $\gamma_x$ is $O(X^{O(1)})$-rational (in a sense specified in Section~\ref{nilseq}) and $T_x=O(X^{k+1}/H^{k+1})$; see Proposition~\ref{sold} for a precise statement. Moreover, the quantities $T_x$ can now be shown to satisfy the approximate functional equation
\begin{align*}
T_x=T_y+O(X/H)    
\end{align*}
when $x/p=y/q+O(\frac{H}{PX})$, for ``most'' $x,y,p,q$. As in~\cite{mrt-fourier}, using mixing properties of the graph $\mathcal{G}$ arising from cancellation in $\sum_{P\leq p\leq 2P}p^{it}$ for $|t|\ll X^{O(1)}$, we may deduce from this that $T_x=T_0+O(X/H)$ for some $T_0$ of polynomial size and for a ``most'' values of $x$. Further, we also have (modulo integer-valued polynomials) the relation 
\begin{align*}
\gamma_x(pn)=\gamma_y(qn)  
\end{align*}
for the same tuples $(x,y,p,q)$, and solving this eventually leads to $\gamma_x(n)$ being $O(1)$-rational (with a bit more work than in~\cite{MRT}, where $\gamma_x(n)$ was just of the form $\frac{a}{q'}n$). Putting everything together, we reach the relation \eqref{eq14}, which was enough for finishing the proof.

For $G$ non-abelian, we can use some of the above arguments, but certain additional difficulties (indicated below) arise that necessitate a more involved analysis involving quantitative nilalgebra and some refinements on the graph theory side. Note that by the factorization theorem for nilsequences~\cite{gt-nil}, we have a similar splitting of polynomials $g:\mathbb{Z}\to G$ to a smooth part, an equidistributed part and a rational part, so we may define a similar equivalence relation for these sequences as for polynomial phases. Moreover, we can make sense of the sequence $g(n)$ evaluated at real $n$ and we can define the size of an element of $G$; see Section~\ref{nilseq} for details. 

Up until the approximate functional equation \eqref{eq15} (now with $g_x(n)$ in place of $P_x(n)$), the arguments in the polynomial phase case are sufficiently general to work equally well for nilsequences. We can also obtain the analogue of \eqref{eq13} similarly but, perhaps surprisingly, in the nilsequence setting the solutions to \eqref{eq13} for a given pair $(a_x,b_x)$ are \emph{not} all approximate characters (see \eqref{far} for a counterexample). We thus must proceed more carefully and extract more information from the fact that \eqref{eq13} holds for an extremely large family of pairs $(a_x,b_x)$. It turns out that the pathological solutions to \eqref{eq13} for a given $(a_x,b_x)$ generally do not obey \eqref{eq13} for other pairs $(a'_x,b'_x)$, but demonstrating that requires some work. 

The way we obtain the required extra information is by generalizing the graph theory argument from~\cite{mrt-fourier} a bit (to configurations of two cycles of unequal length connected by an edge), and this extra flexibility allows us to obtain
\begin{equation}\label{gmix}
g_x((1+\theta)t)\sim g_x(t)\gamma_{x,\theta}(t),\quad t\in [x,x+H]   
\end{equation}
for $t\in \mathbb{R}$ and for a ``very dense'' set of real numbers  $\theta=O(H/X)$ (as opposed to just a few such numbers), where $\gamma_{x,\theta}$ is $Q$-rational with $Q\gg \prod_{p\in [P,2P]}p^{\varepsilon}$ (this notion makes sense in Lie algebras; see Section~\ref{nilseq}). This is the outcome of Proposition~\ref{adi-dense}. 

\begin{remark} As indicated above, while in the case of polynomial phases it suffices to have equation \eqref{gmix} hold for a single $\theta$, in the more general nilsequence case this condition is insufficient due to the existence of exotic "approximately multiplicative" nilsequences. Consider for example $\phi(n)=F(g(n)\Gamma)$ where 
$$g(n) = e_1^{T_1\log n }e_2^{T_2\log n}e_{12}^{-\frac{T_1T_2}{2}(\log n)^2}$$ where here $e_1, e_2, e_{12}$ are the generators of the free $2$-step $3$ dimensional nilpotent Lie group, $\Gamma$ the standard lattice.  By Taylor approximation of the logarithm function, $g(n)$ differs from a polynomial sequence by a negligible amount. Moreover,
$g((1+\theta)t)=g(1+\theta)g(t)$ so that so one would get $\phi((1+\theta)n) \sim \phi(n)$ if $g(1+\theta)$ is very close to $\Gamma$, independent of $n$. 
\end{remark}

It is a fact (following from the Baker--Campbell--Hausdorff formula) that if $n\mapsto \gamma_{x,\theta}(n)$ is simultaneously very rational and of polynomial size, then it is a constant; thus, $\gamma_{x,\theta}(n)=:\gamma_{x,\theta}$. Make in \eqref{gmix} the change of variables $1+\theta_x=e^{\alpha/N}$ with $\alpha \sim 1$ restricted to a very dense set of numbers and with $N=X/H$. Then 
\begin{align*}
g_x(e^{\alpha/N}t)\sim g_x(t)\gamma_{x,\alpha},\quad t=x+O(H),    
\end{align*}
so by iterating 
\begin{align*}
g_x(e^{n\alpha/N} t)\sim g_x(t)\gamma_{x,\alpha}^n
\end{align*}
for all integers $n=O(1)$. In fact, by an interpolation lemma (Lemma~\ref{ber-exp}), we will be able to boost this to real $n$ as well. Now we essentially have a two-variable functional equation for $g_x$, which after some manipulation gives us
\begin{equation}\label{tmix}
g_x(y)\sim T^{N\log(y/x)},\quad y=x+O(H),    
\end{equation}
and for some $T=T_x\in G$ of polynomial size. Here, $T$ is given by the relation
\begin{align*}
T^{\alpha s}\sim \gamma_{x,\alpha}^{s}    
\end{align*}
for $s=O(1)$ and for a dense set of $\alpha\sim 1$ (cf. Proposition~\ref{prop1}). This is still not enough for us, since when $G$ is non-abelian, $y\mapsto F(T^{N\log(y/x)}\Gamma)$ need not resemble a character at all. With some extra work, which involves quantitative equidistribution theory of nilsequences and the mixing lemma to carefully analyze the compatibility between \eqref{gmix} and \eqref{tmix}, we eventually show that $T=O(1)T_0$, where $T_0$ is of polynomial size and lies either in the center of $G$ or in a proper rational subgroup of $G$.  In the case that $G$ is non-abelian, the former case is contained in the latter.  This is then finally enough, since the $O(1)$ error turns out to be negligible by Taylor expansion, and if $T$ lies in a proper rational subgroup, we ascend to a group of lower dimension, so we can apply induction to conclude. Thus $n\mapsto T^{N\log(n/x)}$ must essentially be a polynomial function on an \emph{abelian} nilmanifold, meaning that it is a classical polynomial. This reduces us back to the polynomial phase case, whose proof we outlined above. 

\subsubsection{The sign patterns result}
\label{sss:pfSignPatSkecth}
We then sketch the proof of Theorem~\ref{superpolynomial}. Suppose for the sake of contradiction that $s(k)\ll k^{A}$ for some $A$ and for $k$ belonging to an infinite set $\mathcal{K}$. Then, expanding the (logarithmic) density of each sign pattern of length $k$ as a correlation, we must have 
\begin{align*}
C \coloneqq \frac{1}{\log x}\sum_{n\leq x}\frac{\lambda(n+h_1)\cdots \lambda(n+h_j)}{n}\gg_k 1    
\end{align*}
for $k\in \mathcal{K}$ and for some distinct $h_1,\ldots, h_j\in [1,k]$. The entropy decrement argument developed in~\cite{Tao} (see also~\cite{TaoTeravainenGeneral}), allows one to write $C$ as a double average:
\begin{align}\label{doublecorr}
C=(-1)^k\frac{\log P}{P}\sum_{P\leq p\leq 2P}\frac{1}{\log x}\sum_{n\leq x}\frac{\lambda(n+ph_1)\cdots \lambda(n+ph_j)}{n}+o(1),
\end{align}
where $P=P(x)$ is suitable. However, $P$ has to be very small here (namely $P\ll (\log x)^{o(1)}$), which is by far too small in order to apply Corollary~\ref{cor: mult-pret}. Instead, we leverage the assumption that $\lambda$ is assumed to have few sign patterns to show that the entropy decrement argument can be replaced with a quantitatively much stronger method of moments computation, and this eventually allows us to obtain \eqref{doublecorr} for $P\gg X^{\varepsilon}$ (along a suitable sequence of values of $X$ depending on $\mathcal{K}$). Then we are in a position to apply Corollary~\ref{cor: mult-pret}, and we conclude from the generalized von Neumann theorem that actually $C=o(1)$, which is the desired contradiction.

\subsection{Acknowledgments}

This work was initiated at the American Institute of Mathematics workshop on Sarnak's conjecture in December 2018. KM was supported by Academy of Finland grant no. 285894. MR acknowledges the support of NSF grant DMS-1902063 and a Sloan Fellowship. TT was supported by a Simons Investigator grant, the James and Carol Collins Chair, the Mathematical Analysis \& Application Research Fund Endowment, and by NSF grant DMS-1764034. JT was supported by a Titchmarsh Fellowship. TZ was  supported by ERC grant ErgComNum 682150. 

We are grateful to the anonymous referees for their extremely careful reading of the paper and for numerous helpful comments and remarks that improved the presentation of this paper. We thank Amita Malik, Redmond McNamara and Peter Sarnak for helpful discussions.

\section{Notation and preliminaries}\label{notation-sec}

We use the asymptotic notation $X \ll Y$, $X = O(Y)$ or $Y \gg X$ to denote the estimate $|X| \leq CY$ for some absolute constant $C$ (in case of $Y \gg X$ we also require that $X \geq 0$). If we allow the constant $C$ to depend on parameters, we will indicate this by subscripts unless otherwise specified, thus for instance $X = O_k(Y)$ denotes the estimate $|X| \leq C_k Y$ for some $C_k$ depending on $k$.  We also write $X \asymp Y$ for $X \ll Y \ll X$.

Several of the concepts defined in this paper (e.g., ``large family'', ``smooth polynomial'', ``comparable interval'', etc.) will rely on the above notation, and thus involve some unspecified implicit constants.  If a proposition involves such notation in both its hypotheses and conclusion, then the implied constants in the conclusions are always permitted to depend on the implied constants in the hypotheses.

All intervals in this paper will be closed. If $I$ is an interval, we use $|I|$ to denote its Lebesgue measure and $x_I$ to denote its midpoint, thus $I = [x_I-\frac{|I|}{2}, x_I+\frac{|I|}{2}]$.  For any $x \in \R$, we define the normalized distance
\begin{equation}\label{idist}
 \langle x \rangle_I \coloneqq \frac{\diam(I \cup \{x\})}{|I|}
\end{equation}
and similarly for an interval $J$
\begin{equation}\label{jdist}
 \langle J \rangle_I \coloneqq \frac{\mathrm{diam}(I \cup J)}{|I|}.
\end{equation}
We say that two intervals $I,J$ are \emph{comparable}\footnote{Here and throughout the paper, definitions such as this one that depend on an implicit asymptotic parameter are only called in the presence of such parameters (which will be the parameters in Theorem~\ref{mult-pret}).}, and write $I \sim J$, if we have $\langle I \rangle_J, \langle J \rangle_I \ll 1$, or equivalently if $|I| \asymp |J| \asymp \mathrm{diam}(I \cup J)$. Note that this is an equivalence relation up to modification of the implied constants; for instance if $I \sim J$ and $J \sim K$ then $I \sim K$, where the implied constants in the latter relation can differ from those in the former.

If $F$ is a finite set, we use $\# F$ to denote its cardinality.  If $E$ is a set, we use $1_E$ to denote its indicator function, thus $1_E(n) = 1$ when $n \in E$ and $1_E(n)=0$ otherwise. Similarly, for any statement $S$, we define the indicator $1_S$ to equal $1$ when $S$ is true and $0$ otherwise.

For any subset $E$ of the real line, we use $a+E \coloneqq \{ a+x: x \in E \}$ to denote the translation of $E$ by a shift $a \in \R$, and $\lambda E \coloneqq \{ \lambda x: x \in E \}$ to denote the dilation of $E$ by a factor $\lambda>0$.  For instance if $I,J$ are intervals, then $I \sim J$ if and only if $\lambda I \sim \lambda J$.  If $f \colon \R \to S$ is any function taking values in some set $S$, we use $f(\lambda \cdot) \colon \R \to S$ to denote the dilated function $t \mapsto f(\lambda t)$. For an interval $I$ and function $g$, we also use the pushforward notation $\lambda_{*} (I,g) \coloneqq \left(\lambda I, g\left(\frac{1}{\lambda} \cdot\right) \right)$.

If $a, b$ are elements of an additive group $(G,+)$, and $H$ is a subgroup of $G$, we write $a = b \mod H$ to denote the claim that $a-b \in H$; by abuse of notation we also use $a \mod H$ to denote the element $a+H$ of the quotient group $G/H$.  Similarly, if $G = (G,\cdot)$ is a multiplicative group and $H$ is a normal subgroup, we write $a = b \mod H$ to denote the claim that $ab^{-1} \in H$.

Summations and products over the symbol $p$ (or $p'$, etc.) are always understood to be over primes unless otherwise specified, and similarly sums over $n$ are understood to be over integers unless otherwise specified.

In Section~\ref{sec: signpatterns}, we will need some averaging notation. For a function $f:A\to \mathbb{C}$ defined on a set $A$ with $A\subset \mathbb{N}$ nonempty, define its unweighted and logarithmic average over $A$ by
\begin{align*}
\mathbb{E}_{n\in A}f(n):=\frac{1}{|A|}\sum_{n\in A}f(n)\quad \textnormal{and}\quad \mathbb{E}_{n\in A}^{\log}f(n):=\frac{1}{\sum_{n\in A}\frac{1}{n}}\sum_{n\in A}\frac{f(n)}{n},    
\end{align*}
respectively. Thus in particular for a bounded function $f:\mathbb{N}\to \mathbb{C}$ we have
\begin{align*}
\mathbb{E}_{n\leq x}^{\log}f(n)=\frac{1}{\log x} \sum_{n\leq x}\frac{f(n)}{n}+o(1),\quad \textnormal{and}\quad  \mathbb{E}_{x\leq p\leq 2x}f(p)=\frac{1}{x/\log x}\sum_{x\leq p\leq 2x}f(p)+o(1).  
\end{align*}

If ${\mathcal P}$ is a collection of prime numbers, we use $\prod {\mathcal P}$ to denote the product of its elements:
$$\prod {\mathcal P} \coloneqq \prod_{p \in {\mathcal P}} p.$$
For any $P \geq 2$, we let $\pi_0(P)$ denote the quantity
$$ \pi_0(P) \coloneqq \frac{P}{\log P}.$$
Note that from the prime number theorem, we see that for sufficiently large $P$, the number of primes in $[P,2P]$ or $[P/2,P]$ is comparable to $\pi_0(P)$.  Accordingly, we say that a set of primes in $[P,2P]$ or $[P/2,P]$ is \emph{large} if its cardinality is $\gg \pi_0(P)$.  Observe that if ${\mathcal P}$ is a large set of primes in $[P,2P]$ or $[P/2,P]$, then we have an exponential lower bound
\begin{equation}\label{exp-lower}
\prod {\mathcal P} \gg \exp( c P )
\end{equation}
for some $c \gg 1$.  In practice, this lower bound means that $\prod {\mathcal P}$ is so large compared with the many ``polynomial size'' quantities we will encounter in the course of our arguments that this modulus is effectively infinite.

For a smooth function $f \colon \R \to \C$, we use $f^{(j)}$ to denote the $j^{\mathrm{th}}$ derivative for $j \geq 0$.  We recall the \emph{Bernstein inequality} (see e.g.,~\cite[p. 146]{Prasolov})
\begin{equation}\label{ber}
 \sup_{t \in I} |f^{(1)}(t)| \ll_k |I|^{-1} \sup_{t \in I} |f(t)|
\end{equation}
for all polynomials $f \in \Poly_{\leq k}(\R \to \R)$, and hence on iteration
\begin{equation}\label{ber-2}
 \sup_{t \in I} |f^{(j)}(t)| \ll_k |I|^{-j} \sup_{t \in I} |f(t)|
\end{equation}
for any $j \geq 0$ (note that $f^{(j)}$ vanishes for $j > k$).  From Taylor expansion we then also have
\begin{equation}\label{ber-3}
 |f^{(j)}(t')| \ll_{k} |I|^{-j} \langle t' \rangle_I^{k-j} \sup_{t \in I} |f(t)|
\end{equation}
for any $t' \in \R$ and $j \geq 0$, using the notation \eqref{idist}.

If $\delta > 0$, we use $\Poly_{\leq k}(\delta \Z \to \Z)$ to denote the subgroup of the additive group $\Poly_{\leq k}(\R \to \R)$ consisting of polynomials $\gamma$ such that $\gamma(\delta\Z) \subset \Z$; we refer to these polynomials as \emph{$\frac{1}{\delta}$-integral} polynomials.  We have the following explicit description of these groups:

\begin{lemma}[Discrete Taylor expansion]\label{dte}  For any $\delta > 0$ and $k \geq 0$, the space $\Poly_{\leq k}(\delta \Z \to \Z)$ consists precisely of those functions $\gamma: \R \to \R$ of the form
$$ \gamma(t) \coloneqq \sum_{j=0}^k c_j \binom{t/\delta}{j}$$
for some integers $c_0,\dots,c_k$, where $\binom{x}{j} \coloneqq \frac{x(x-1)\dots(x-j+1)}{j!}$.
\end{lemma}

In some parts of the paper we will also use a non-abelian version of  Lemma~\ref{dte} (see Lemma~\ref{na-dte}).

\begin{proof}  By rescaling we may take $\delta=1$.  The claim is trivial for $k=0$, so suppose inductively that $k \geq 1$ and that the claim has already been proven for $k-1$.  The polynomials $\binom{\cdot}{j}$ for $j=0,\dots,k$ all lie in $\Poly_{\leq k}(\Z \to \Z)$, and hence so do all integer linear combinations $\sum_{j=0}^k c_j \binom{\cdot}{j}$.  Conversely, suppose that $\gamma \in \Poly_{\leq k}(\Z \to \Z)$.  On taking $k^{\mathrm{th}}$ divided differences, we see that the $k^{\mathrm{th}}$ derivative $\gamma^{(k)}$ (which is a constant) is equal to an integer $c_k$.  Thus the polynomial $\gamma - c_k \binom{\cdot}{k}$ has vanishing $k^{\mathrm{th}}$ derivative and thus lies in $\Poly_{\leq k-1}(\Z \to \Z)$.  The claim now follows from the induction hypothesis. 
\end{proof}

We will need the following application of Bezout's identity:

\begin{lemma}[Bezout identity]\label{bezout}  Let $a,b$ be coprime natural numbers, and let $k \geq 0$.  Then for any $\lambda>0$ we have
$$ \Poly_{\leq k}\left(\frac{\lambda}{a} \Z \to \Z\right) + \Poly_{\leq k}\left(\frac{\lambda}{b} \Z \to \Z\right) = \Poly_{\leq k}(\lambda\Z \to \Z)$$
and
$$ \Poly_{\leq k}\left(\frac{\lambda}{a} \Z \to \Z\right) \cap \Poly_{\leq k}\left(\frac{\lambda}{b} \Z \to \Z\right) = \Poly_{\leq k}\left(\frac{\lambda}{ab} \Z \to \Z\right).$$
\end{lemma}

Thus for instance every $1$-integral polynomial can be decomposed as the sum of an $a$-integral and a $b$-integral polynomial, and a polynomial is $ab$-integral if and only if it is both $a$-integral and $b$-integral.

\begin{proof} 
See Appendix~\ref{app:b}.
\end{proof}

We will need a variant of the Bernstein inequality for exponential polynomials, that is to say real linear combinations of exponential monomials $t \mapsto t^j \exp(\alpha t)$ for some non-negative integers $j$ and real numbers $\alpha$:

\begin{lemma}[Bernstein inequality for exponential polynomials]\label{ber-exp}  Let $d_1,\dots,d_k$ be non-negative integers, and let $N_0$ be a sufficiently large natural number depending on $k,d_1,\dots,d_k$.  Let $\alpha_1,\dots,\alpha_k$ be real numbers whose absolute values are sufficiently small depending on $k,d_1,\dots,d_k,N_0$.  Let $P: \R \to \R$ be a real linear combination of the exponential monomials $t \mapsto t^j \exp(\alpha_i t)$ for $i=1,\dots,k$ and $0 \leq j \leq d_i$. Then for any interval $I$ and any non-negative integer $m$ one has, for all $t \in I$,
\begin{equation}\label{pmt}
 |P^{(m)}(t)| \ll_{k,d_1,\dots,d_k,m,N_0,I} \sup_{n=1,\dots,N_0} |P(n)|.
\end{equation}
\end{lemma}

\begin{proof} 
See Appendix~\ref{app:a}.
\end{proof}

\section{Local correlations with polynomial phases}\label{sec: polyphases}

In this section, we establish Theorem~\ref{mult-poly}, which implies Corollary~\ref{lam-poly} as a special case.  Our arguments shall follow those in~\cite{mrt-fourier} (although they will be reformulated in a more general and algebraic setting that applies to relevant collections of phase functions, such as polynomial phases and later to nilsequences in Section~\ref{nilseq}). Some familiarity with the arguments in~\cite{mrt-fourier} will be presumed in this section.

Let $k,\theta,f,X,\eta,H$ be as in Theorem~\ref{mult-poly}.  To simplify the notation we now allow all implied constants in the asymptotic notation to depend on $k,\theta,\eta$, thus for instance
\begin{equation}\label{f-uk-large}
 \int_X^{2X} \| f\|_{u^{k+1}([x,x+H])}\ dx \gg X.
\end{equation}
We can assume that $X$ is sufficiently large depending on $k,\theta,\eta$, since the claim is trivial otherwise.  We can also assume\footnote{Indeed, from the results in~\cite{mrt-fourier} we can almost assume $k \geq 2$, except for the problem that those results contain an additional loss of $H^\rho$ in the conclusion that is not conceded here.  In any case, the arguments here will also recover the $k=1$ case without difficulty.} $k \geq 1$, since the $k=0$ case follows similarly to~\cite[Theorem A.1]{MRT}\footnote{The only difference is that one needs to, in the formula below~\cite[Theorem A.2]{MRT}, treat the integral over $|t| \geq CX/(2H)$ by the mean value theorem to be able to work with $M(f; CX/H,Q)$ instead of $M(f; X,Q)$.}.

It will be convenient to abstract the properties of the polynomial phases one is testing against, as this will allow us to easily generalize many of the arguments in this section to the case of nilsequence correlations in Section~\ref{nilseq}.  Define a \emph{local polynomial phase} to be a pair $\phi = (I, P)$, where $I$ is an interval in $\R$ and $P \in \Poly_{\leq k}(\R \to \R)$ is a polynomial.  We let $\Phi$ denote the set of all local polynomial phases $(I,P)$, and $\Phi_I$ the set of local polynomial phases $(I,P)$ with a given $I$.  Intuitively, $(I,P)$ should be viewed as an abstraction of the phase function $t \mapsto e(P(t))$ on the interval $I$.  If $\phi = (I,P)$ is a local polynomial phase and $f: \Z \to \C$ is a function, we define the correlation
\begin{equation}\label{corr-def}
 \langle f, \phi \rangle \coloneqq \frac{1}{|I|} \sum_{n \in I} f(n) e(-P(n)).
\end{equation}
Thus we have
$$ \| f\|_{u^{k+1}([x,x+H])} = \sup_{\phi \in \Phi_{[x,x+H]}} |\langle f, \phi \rangle|$$
and thus from \eqref{f-uk-large}
\begin{equation}\label{fow}
 \int_X^{2X} \sup_{\phi \in \Phi_{[x,x+H]}} \left|\langle f, \phi \rangle\right|\ dx \gg X.
\end{equation}
Recall from Section~\ref{notation-sec} that given any local polynomial phase $\phi = (I,P) \in \Phi$ and a scaling factor $\lambda > 0$, we define the rescaling (or pushforward) $\lambda_{*} \phi \in \Phi$ by the formula
$$ \lambda_{*} \phi \coloneqq \left(\lambda I, P\left( \frac{1}{\lambda} \cdot\right)\right).$$
Note that this gives a multiplicative action on $\Phi$, in the sense that
$$ (\lambda_1)_{*} ((\lambda_2)_{*} \phi) = (\lambda_1 \lambda_2)_{*}\phi$$
whenever $\phi \in \Phi$ and $\lambda_1,\lambda_2 > 0$.

Following~\cite[\S 2]{mrt-fourier}, we perform a convenient discretization.  Define an \emph{$(X,H)$-family of intervals} to be a finite collection ${\mathcal I}$ of intervals of length $H$ contained in $[X/10,10X]$ such that any pair of intervals in ${\mathcal I}$ are separated by a distance at least $500H$.  We say that such a family ${\mathcal I}$ is \emph{large} if $\# {\mathcal I} \gg X/H$.
 By repeating the proof of~\cite[Lemma 2.1]{mrt-fourier} (which is a pigeonholing argument) using \eqref{fow} as a starting point, one obtains a large $(X,H)$-family of intervals ${\mathcal I}$, such that for each $I \in {\mathcal I}$ one can find $\phi_I \in \Phi_I$ such that
\begin{equation}\label{gii}
| \langle f, \phi_I \rangle| \gg 1.
\end{equation}
We remark that this step does not require any properties of the polynomial space $\Poly_{\leq k}(\R \to \R)$, as it only uses the fact that $e(P(n))$ is $1$-bounded for every $P$ in this space.

The next step is to use the multiplicativity of $f$ to relate the various $\phi_I$ to each other.  We need a key definition, given as Definition~\ref{poly-comp} below.  Given an interval $I$ in $\R$, we say that a map $\eps \in \Poly_{\leq k}(\R \to \R)$ is \emph{smooth} on $I$ if one has the bound
$$|\eps(t)| \ll 1$$
for all $t \in I$, which by \eqref{ber-3} also implies that
$$\left|\frac{d^j}{dt^j} \eps(t)\right| \ll |I|^{-j} \langle t \rangle_I^{k-j}$$
for all $j \geq 0$ and $t \in \R$.  In particular, if $\eps$ is smooth on $I$, then it is also smooth on $I'$ for any $I' \sim I$.

\begin{definition}[Comparability of polynomial phases]\label{poly-comp} Given two local polynomial phases $\phi_1 = (I_1,P_1), \phi_2 = (I_2,P_2)$ of $\Phi$ and a scaling factor $\delta>0$, we define the relation
$$ \phi_1 \sim_{\delta} \phi_2$$
to hold if $I_1 \sim I_2$, and we have a splitting
$$ P_1 = \eps + P_2 + \gamma,$$
where $\eps, \gamma \in \Poly_{\leq k}(\R \to \R)$ are polynomials obeying the following axioms:
\begin{itemize}
\item[(i)]  ($\eps$ smooth)  $\eps$ is smooth on $I_1$.
\item[(ii)]  ($\gamma$ is $\frac{1}{\delta}$-integral)  $\gamma \in \Poly_{\leq k}(\delta \Z \to \Z)$.
\end{itemize}
\end{definition}

Informally, the relation $\phi_1 \sim_{\delta} \phi_2$ asserts that $\phi_1$ ``pretends to be'' $\phi_2$ on the discrete set $I_1 \cap \delta \Z$.  Technically, this is not a single relation, but a family of relations, depending on the choices of implied constants appearing in (i), but we shall abuse notation by referring to $\sim_{\delta}$ as a single relation.  It obeys the following basic properties:

\begin{proposition}[Basic properties of $\sim_{\delta}$]\label{basic}  Let $\delta > 0$, and let $\phi, \phi', \phi'' \in \Phi$.
\begin{itemize}
\item[(i)] (Equivalence relation)  We have $\phi \sim_{\delta} \phi$, and if $\phi \sim_{\delta} \phi'$ then $\phi' \sim_{\delta} \phi$.  Finally, if $\phi \sim_{\delta} \phi'$ and $\phi' \sim_{\delta} \phi''$ then $\phi \sim_{\delta} \phi''$, where we allow the implied constants in the latter relations to depend on the implied constants in the former relations.
\item[(ii)]  (Dilation invariance)  If  $\phi \sim_{\delta} \phi'$ and $\lambda > 0$, then $\lambda_{*} \phi \sim_{\lambda \delta} \lambda_{*} \phi'$.
\item[(iii)]  (Sparsification)  If $\phi \sim_{\delta} \phi'$, then $\phi \sim_{\ell\delta} \phi'$ for any natural number $\ell$.
\end{itemize}
\end{proposition}

\begin{proof} These are immediate from Definition~\ref{poly-comp}, together with the previously made observation that a polynomial smooth on an interval $I$ is automatically smooth on all comparable intervals $I' \sim I$.
\end{proof}

The relevance of this relation to the correlations \eqref{gii} comes from the following lemma.

\begin{proposition}[Large sieve]\label{lsieve}  Let $I$ be an interval of some length $|I| \geq 1$, and let $f: \Z \to \C$ be a function bounded in magnitude by $1$.  Suppose that for each $i=1,\dots,K$ there is an interval $I_i \sim I$ and a local polynomial phase $\phi_i \in \Phi_{I_i}$ such that
$$ |\langle f, \phi_i \rangle| \gg 1.$$
Then either
$$ K \ll 1$$
or there exists $1 \leq i < j \leq K$ such that
$$ \phi_i \sim_{1} \phi_j.$$ 
\end{proposition}

\begin{proof}  Write $\phi_i = (I_i,P_i)$ and $H = |I|$. By \eqref{corr-def}, for each $1 \leq i \leq K$, we can find a real number $\theta_i$ such that
$$ \mathrm{Re}\left( e(\theta_i) \sum_{n \in I_i} f(n) e(-P_i(n)) \right) \gg H$$
and hence on summing in $i$ and rearranging
$$ \mathrm{Re}\left( \sum_{n \in I} f(n) \sum_{i=1}^K 1_{I_i}(n) e(\theta_i) e(-P_i(n)) \right)\gg H K.$$
By Cauchy-Schwarz we conclude that
$$ \sum_{n \in I} \left| \sum_{i=1}^K 1_{I_i}(n) e(\theta_i) e(-P_i(n))  \right|^2 \gg H K^2.$$
The left-hand side can be rearranged as
$$ \sum_{i=1}^K \sum_{j=1}^K e(\theta_j-\theta_i) \sum_{n \in I_i \cap I_j} e(P_i(n) - P_j(n)).$$
Thus, by the pigeonhole principle and triangle inequality, there exists $1 \leq i \leq K$ such that
$$ \sum_{j=1}^K \left| \sum_{n \in I_i \cap I_j} e(P_i(n) - P_j(n))\right| \gg HK,$$
and hence
\begin{equation}
\label{eq:Pi-Pjest}
\left| \sum_{n \in I_i \cap I_j} e(P_i(n) - P_j(n))\right| \gg H
\end{equation}
for $\gg K$ choices of $j=1,\dots,K$. Fix this choice of $i$.

Let $n_I$ denote an integer point in $I$.  For each $j$ such that~\eqref{eq:Pi-Pjest} holds, we write
$$P_i(t) - P_j(t) = \sum_{l=0}^k \alpha_{j,l} (t-n_I)^l$$
for some real coefficients $\alpha_{j,l}$. Then we have
$$ \left| \sum_{n \in (I_i-n_I) \cap (I_j - n_I)} e\left( \sum_{l=0}^k \alpha_{j,l} n^l \right)\right| \gg H$$
Applying Weyl sum estimates such as~\cite[Lemma 1.1.16]{tao-higher}, we conclude that there exists a natural number $1 \leq q_j \ll 1$ such that
$$ \| q_j \alpha_{j,l} \|_{\R/\Z} \ll H^{-l}$$
for $l=0,\dots,k$, where $\|x\|_{\R/\Z}$ denotes the distance of $x$ to the nearest integer.  In particular there exist natural numbers $1 \leq a_{j,l} \leq q_j$ such that
$$\left\| \alpha_{j,l} - \frac{a_{j,l}}{q_j} \right\|_{\R/\Z} \ll H^{-l}.$$
The total number of tuples $(q_j,a_{j,1},\dots,a_{j,k})$ is $O(1)$.  Thus by the pigeonhole principle, either $K \ll 1$, or else there exist $1 \leq j < j' \leq K$ such that $q_j = q_{j'}$ and $a_{j,l} = a_{j',l}$ for all $l=0,\dots,K$.  In particular, by the triangle inequality we have
$$\| \alpha_{j,l} - \alpha_{j',l} \|_{\R/\Z} \ll H^{-l}$$
for $l=0,\dots,K$, so we can write $\alpha_{j',l} = \eps_{j, j', l} + \alpha_{j,l} + \gamma_{j, j', l}$ for some integer $\gamma_{j,j', l}$ and some real number $\eps_{j,j',l} = O(H^{-l})$. This gives the decomposition
$$ P_j(t) = \sum_{l=0}^k \eps_{j,j',l} (t-n_I)^l + P_{j'}(t) +  \sum_{l=0}^k \gamma_{j,j',l} (t-n_I)^l.$$
Comparing this with Definition~\ref{poly-comp}, we see that
$$ \phi_j \sim_{1} \phi_{j'},$$
and the proposition follows.
\end{proof}

Using this proposition, we can obtain

\begin{proposition}[Scaling down]\label{Scaling-down} Let $2 \leq P \leq Q \leq H \leq X$ and let $f: \N \to \C$ be a $1$-bounded multiplicative function.   Suppose there exists a large $(X,H)$-family ${\mathcal I}$ and a local polynomial phase $\phi_I \in \Phi_I$ associated to each interval $I \in {\mathcal I}$ such that
$$ |\langle f, \phi_I \rangle| \gg 1$$
for all $I \in {\mathcal I}$. Assuming that $P, \frac{\log Q}{\log P}$ are sufficiently large (depending on the implied constants in the above hypotheses), there exist $P' \in [P,Q/2]$, a large $(\frac{X}{P'}, \frac{H}{P'})$-family ${\mathcal I}'$, and a function $\phi'_{I'} \in \Phi_{I'}$ associated to each $I' \in {\mathcal I}'$, such that
$$ |\langle f, \phi'_{I'} \rangle| \gg 1$$
for all $I' \in {\mathcal I}'$.  Furthermore, for each $I' \in {\mathcal I}'$, one can find $\gg \pi_0(P')$ pairs $(I, p')$, where $I \in {\mathcal I}$ and $p'$ is a prime in $[P', 2P']$, such that the rescaled interval $\frac{1}{p'} I$ lies within $3\frac{H}{P'}$ of $I'$, and such that
\begin{equation}\label{peo}
(\frac{1}{p'})_{*} \phi_I \sim_1 \phi'_{I'}.
\end{equation}
\end{proposition}

\begin{proof}  From Proposition~\ref{lsieve} and the greedy algorithm, we can associate to each interval $I$ of length $H \geq 1$ and any $\eta' > 0$ a family $\phi_1,\dots,\phi_K \in \Phi_I$ of local polynomial phases with $K = O_{\eta'}(1)$ such that whenever one has
$$ | \langle f, \phi \rangle| \geq \eta'$$
for some $\phi \in \Phi_J$ with $J \subset I$ and $|J| \geq \eta' |I|$, then one has
$$ \phi \sim_{1} \phi_i$$ 
for some $i=1,\dots,K$ (if we permit implied constants in the $\sim_1$ notation to depend on $\eta'$).  The claim now follows by repeating the proof of~\cite[Proposition 3.1]{mrt-fourier} (which is a Tur\'an--Kubilius argument), using the above claim as a substitute for~\cite[Lemma 2.2]{mrt-fourier}.  For the convenience of the reader we sketch the main ideas of this argument as follows.  First, by using~\cite[Proposition 2.5]{mrt-fourier} and the multiplicative nature of $f$, one can deduce that
$$ |\langle f, (\frac{1}{p'})_* \phi_I \rangle| \gg 1$$
for many $I \in {\mathcal I}$ and many primes $p' \in [P,Q]$, and thence (by the pigeonhole principle) for many $I \in {\mathcal I}$ and $p' \in [P',2P']$ for a suitable $P'$.  By further pigeonholing, we may arrange matters so that the intervals $\frac{1}{p'} I$ lie close to intervals $I'$ in a suitable large $(\frac{X}{P'}, \frac{H}{P'})$-family ${\mathcal I}'$.  Using the previously mentioned claim, one can then show that many of the $(\frac{1}{p'})_* \phi_I$ associated to a given interval $I'$ are related via the $\sim_1$ relation to a suitable phase $\phi'_{I'}$, which will give the claim.
\end{proof}

We also need the following version of the Chinese remainder theorem\footnote{The reason we call this a Chinese remainder theorem is that it allows us to combine $\mod p$ conditions for different primes $p$.}. This proposition turns out to be very useful in what follows, since it allows us to upgrade equivalences between different exponential phases up to the point where the modulus is so large that we must have a genuine equality in $\mathbb{R}$.

\begin{proposition}[Chinese remainder theorem]\label{chinese} Let $I$ be an interval of some length $|I| \geq 1$, and let ${\mathcal P}$ be a finite collection of primes.
\begin{itemize}
\item[(i)]  Suppose that $\phi \in \Phi_I$, and that for each $p \in {\mathcal P}$ there exists $\phi_p \in \Phi$ such that
$$ \phi_p \sim_{1} \phi.$$
Then there exists $\tilde \phi \in \Phi_I$ such that
$$ \phi_p \sim_{\frac{1}{p}} \tilde \phi$$
for all $p \in {\mathcal P}$, and furthermore $\langle f, \phi \rangle = \langle f, \tilde \phi \rangle$ for all $f: \Z \to \C$.
\item[(ii)]  Suppose that $\phi \in \Phi_I$ and $\phi' \in \Phi$ are such that
$$ \phi \sim_{\frac{1}{p}} \phi'$$
for all $p \in {\mathcal P}$, and suppose $|I|$ is sufficiently large (depending on the implied constants in the $\sim_{\frac{1}{p}}$ notation). Then
$$ \phi \sim_{\frac{1}{\prod {\mathcal P}}} \phi'.$$
\end{itemize}
\end{proposition}

\begin{proof}  
See Appendix~\ref{app:b}.
\end{proof}

One can now conclude

\begin{proposition}[Building a family of related local polynomial phases]\label{prop32-analog}  Let the hypotheses be as in Theorem~\ref{mult-poly}.  Let $\eps > 0$ be sufficiently small depending on $k,\theta,\eta$, and suppose that $X$ is sufficiently large depending on $\theta,\eta,\eps,k$.  Then there exist $P', P'' \in [X^{\eps^2/2},X^\eps]$, a large $(\frac{X}{P'P''}, \frac{H}{P'P''})$-family ${\mathcal I}''$, and local polynomial phases $\phi''_{I''} \in \Phi_{I''}$ for each $I'' \in {\mathcal I}''$ such that
\begin{equation}\label{falp}
|\langle f, \phi''_{I''} \rangle| \gg 1
\end{equation}
for all $I'' \in {\mathcal I}''$.  Furthermore, there exists a collection ${\mathcal Q}$ of $\gg \pi_0(P')^2 \frac{X}{H}$ quadruples $(I''_1,I''_2, p'_1,p'_2)$ with $I''_1, I''_2$ distinct intervals in ${\mathcal I}''$ and $p'_1,p'_2$ distinct primes in $[P',2P']$, such that $I''_1$ lies within $50 \frac{H}{P'P''}$ of $\frac{p'_2}{p'_1} I''_2$ (so in particular $\frac{1}{p'_2} I''_1 \sim \frac{1}{p'_1} I''_2$), and such that
\begin{equation}\label{phip}
(\frac{1}{p'_2})_{*} \phi''_{I''_1} \sim_{\frac{1}{p''}} (\frac{1}{p'_1})_{*} \phi''_{I''_2}
\end{equation}
for a large set of primes $p''$ in $[P''/2, P'']$. (The implied constants in the conclusions may depend on the implied constants in the hypotheses.)
\end{proposition}

\begin{proof}  One basically repeats~\cite[Proof of Proposition 3.2]{mrt-fourier} more or less verbatim, but replacing~\cite[Proposition 3.1]{mrt-fourier} by Proposition~\ref{Scaling-down}.  For the convenience of the reader we now outline some more details of the argument.  By two applications of Proposition~\ref{Scaling-down} (arguing exactly as in the proof of~\cite[Proposition 3.2]{mrt-fourier} down to~\cite[(41)]{mrt-fourier}), we can find $P' \in [X^{\eps^2},X^{\eps}]$ and $P'' \in [(X/P')^{\eps^2},(X/P')^\eps] \subset [X^{\eps^2/2},X^{\eps}]$, an $(X/P',H/P')$-family ${\mathcal I}'$ of intervals, an $(X/P'P'',H/P'P'')$-family ${\mathcal I}''$ of intervals, and functions $\phi'_{I'}, \phi''_{I''} \in \Phi$ associated to each $I' \in {\mathcal I}', I'' \in {\mathcal I}''$ with the following properties:
\begin{itemize}
\item  One has \eqref{falp} for all $I'' \in {\mathcal I}''$.
\item For each $I' \in {\mathcal I}'$, there are $\gg \pi_0(P')$ pairs $(I,p')$ with $I \in {\mathcal I}$ and $p'$ a prime in $[P',2P']$ such that $I/p'$ lies within $3H/P'$ of $I'$ and
\begin{equation}\label{phi-1}
 (\frac{1}{p'})_{*} \phi_I \sim_{1} \phi'_{I'}.
\end{equation}
\item  For each $I'' \in {\mathcal I}''$, there are $\gg \pi_0(P'')$ pairs $(I',p'')$ with $I' \in {\mathcal I}'$ and $p''$ a prime in $[P''/2, P'']$ such that $I'/p''$ lies within $3\frac{H}{P'P''}$ of $I''$, and
\begin{equation}\label{phi-2}
 (\frac{1}{p''})_{*} \phi_{I'} \sim_{1} \phi''_{I''}.
\end{equation}
\end{itemize}
Note that the property \eqref{falp} only depends on the values of $\phi''_{I''}$ on the integers.  Thus, by Proposition~\ref{chinese}(i), we may without loss of generality upgrade \eqref{phi-2} to
\begin{equation}\label{phi-3}
 (\frac{1}{p''})_{*} \phi_{I'} \sim_{\frac{1}{p''}} \phi''_{I''}
\end{equation}
without impacting \eqref{falp} or any of the other properties listed above.  Henceforth we shall assume that \eqref{phi-3} holds.  Applying Cauchy-Schwarz (as in the continuation of the proof of~\cite[Proposition 3.2]{mrt-fourier} down to~\cite[(43)]{mrt-fourier}), we can now find $\gg \pi_0(P')^2 \pi_0(P'') \frac{X}{H}$ octuplets\footnote{For a visualization of the dependencies between the intervals $I$, $I_1'$, $I_1''$, $I_2'$ and $I_2''$, we refer to~\cite[Figure 8]{mrt-fourier}.} $(I,I'_1,I'_2,I''_1,I''_2,p'_1,p'_2,p'')$ where
\begin{itemize}
\item $I \in {\mathcal I}$, $I'_1,I'_2 \in {\mathcal I}'$, $I''_1, I''_2 \in {\mathcal I}''$;
\item $p'_1,p'_2$ are primes in $[P', 2P']$, and $p''$ is a prime in $[P''/2, P'']$, with $p'_1 \neq p'_2$;
\item For $i=1,2$, $\frac{1}{p'_i} I$ lies within $3\frac{H}{P'}$ of $I'_i$, and $\frac{1}{p''} I'_i$ lies within $3\frac{H}{P'P''}$ of $I''_i$.
\item For each $i=1,2$, we have
\begin{equation}\label{phi-4}
 (\frac{1}{p'_i})_{*} \phi_I \sim_{1} \phi'_{I'_i}
\end{equation}
and
\begin{equation}\label{phi-5}
 (\frac{1}{p''})_{*} \phi'_{I'_i} \sim_{\frac{1}{p''}} \phi''_{I''_i}.
\end{equation}
\end{itemize}

From \eqref{phi-4} and Proposition~\ref{basic}(ii) we have for $i=1,2$ that
$$(\frac{1}{p'_i p''})_{*} \phi_I \sim_{\frac{1}{p''}} (\frac{1}{p''})_{*} \phi'_{I'_i}$$
and hence by \eqref{phi-5} and Proposition~\ref{basic}(i)
$$(\frac{1}{p'_i p''})_{*} \phi_I \sim_{\frac{1}{p''}} \phi''_{I''_i}$$

and thus by Proposition~\ref{basic}(ii), (iii)
$$ (\frac{1}{p'_1 p'_2 p''})_{*} \phi_I \sim_{\frac{1}{p''}} (\frac{1}{p'_{3-i}})_{*} \phi''_{I''_i}$$
and thus by Proposition~\ref{basic}(i) we obtain \eqref{phip}.  The proposition now follows by repeating the remainder of the proof of~\cite[Proposition 3.2]{mrt-fourier} (where one estimates how many quadruples arise from these octuplets).
\end{proof}

One should think of the set ${\mathcal Q}$ of quadruples $e = (I''_1,I''_2, p'_1, p'_2)$ produced by the above proposition as a family of ``edges'' of a certain graph with vertex set ${\mathcal I}''$. Now, we adapt the graph-theoretic arguments in~\cite[\S 4]{mrt-fourier} to locate lots of quadruples $e = (I''_1,I''_2, p'_1, p'_2)$ in ${\mathcal Q}$ for which one has a lot of structural control on the local polynomial phases $\phi''_{I''_1}, \phi''_{I''_2}$, and their relationship to each other.
For the rest of this section we introduce the quantities
\begin{equation}\label{N-def}
 N \coloneqq \# {\mathcal I}'' \asymp \frac{X}{H}\quad \textnormal{and}\quad d \coloneqq \pi_0(P')^2.
\end{equation}
We say that a quantity $a$ is of \emph{polynomial size} if one has $a = O(X^{O(1)})$.  For instance, $P', P'', H, X, N, d$ are all of polynomial size.

\begin{proposition}[Local structure of $\phi''$]\label{prop41-analog}  Let the hypotheses be as in Theorem~\ref{mult-poly}, and let $\eps,X,P',P'',{\mathcal I}'', \phi''_{I''}, {\mathcal Q}$ be as in Proposition~\ref{prop32-analog}.  
Let $\ell_1$, $\ell_2$ be even integers such that
\begin{equation}\label{dop}
d^{\ell_1}, d^{\ell_2} \geq N^2 d^{10}.
\end{equation}
(Note from the lower bound on $P'$ that we can choose $\ell_1,\ell_2 = O_\eps(1)$). We allow implied constants to depend on $\eps, \ell_1, \ell_2$.  Then, for a subset ${\mathcal Q}'$ of the quadruples $e = (I''_1, I''_2, p'_1, p'_2)$ in ${\mathcal Q}$ of cardinality $\gg dN$, one can find a collection ${\mathcal A}_e$ of quadruples $\vec a = (a_1,a_2,b_1,b_2)$ of natural numbers of cardinality $\asymp d^{\ell_1 + \ell_2} / N^2$, and a large collection ${\mathcal P}_{e,\vec a}$ of primes in $[P''/2,P'']$ associated to each $\vec a \in {\mathcal A}_e$, with the following properties:
\begin{itemize}
\item[(i)] One has
\begin{equation}\label{phip-big}
(\frac{1}{p'_2})_{*} \phi''_{I''_1} \sim_{\frac{1}{\prod {\mathcal P}_{e,\vec a}}} (\frac{1}{p'_1})_{*} \phi''_{I''_2}.
\end{equation}
Here the implied constants in the equivalence relation do not depend on $\ell_1$ or $\ell_2$.
\item[(ii)]  For $i=1,2$, $a_i,b_i$ are products of $\ell_i$ primes in $[P',2P']$; in particular
\begin{equation}\label{ab0}
a_i, b_i \asymp (P')^{\ell_i},
\end{equation}
so $a_i,b_i$ are of polynomial size.  Furthermore, we have
\begin{equation}\label{ab3}
 a_i - b_i \asymp \frac{1}{N} a_i.
\end{equation}
\item[(iii)]  For $i=1,2$, we have the approximate dilation invariance
\begin{equation}\label{adi}
(\frac{1}{a_i})_{*} \phi''_{I''_i} \sim_{\frac{1}{\prod {\mathcal P}_{e,\vec a}}} (\frac{1}{b_i})_{*} \phi''_{I''_i}.
\end{equation}
Here the implied constants in the equivalence relation may depend on $\ell_i$, but not on the complementary parameter $\ell_{3-i}$.
\end{itemize}
\end{proposition}

For the arguments in this section, one could take the parameters $\ell_1,\ell_2$ to be equal to each other, but in the next section it will be convenient to allow $\ell_1,\ell_2$ to be distinct (in fact in that section we will take $\ell_1$ to be very large compared to $\ell_2$).  The specified dependence of parameters in \eqref{phip-big}, \eqref{adi} on $\ell_1,\ell_2$ will be of no relevance in the current arguments, but will be crucially exploited in the next section.

\begin{proof}  Running the proof of~\cite[Proposition 4.1]{mrt-fourier} all the way down to~\cite[(53)]{mrt-fourier} (with the role of $k$ replaced by $\ell_1$ and $\ell_2$, noting that the argument works perfectly well when the two cycles in the graph have different length), with Proposition~\ref{prop32-analog} playing the role of~\cite[Proposition 3.2]{mrt-fourier}, we conclude that we can find $\gg d^{\ell_1+\ell_2+1}/N$ $(\ell_1+\ell_2)$-tuples
$$ \vec I'' \coloneqq (I''_{j,i})_{i=1,2; j \in \{0,1,\dots,\ell_i-1\}} \in ({\mathcal I}'')^{\ell_1+\ell_2}$$
which are ``non-degenerate and very good'' in the sense that they obey the following axioms:
\begin{itemize}
\item[(i)]  If $i=1,2$ and $j=0,\dots,\ell_i-1$ then there exist (uniquely determined) distinct primes $p'_{1,j,i}, p'_{2,j,i} \in [P',2P']$ such that $I''_{j+1,i}$ lies within $100 \frac{H}{P'P''}$ of $\frac{p'_{1,j,i}}{p'_{2,j,i}} I''_{j,i}$ (with the cyclic convention $I''_{\ell_i,i} = I''_{0,i}$).  In particular $\frac{1}{p'_{1,j,i}} I''_{j+1,i} \sim \frac{1}{p'_{2,j,i}} I''_{j,i}$.
\item[(ii)]  There also exist distinct primes $p'_1,p'_2 \in [P',2P']$ such that $(I''_{0,1},I''_{0,2}, p'_1, p'_2) \in {\mathcal Q}$.  In particular, $I''_{0,2}$ lies within $100 \frac{H}{P'P''}$ of $\frac{p'_1}{p'_2} I''_{0,1}$ and hence $\frac{1}{p'_1} I''_{0,2} \sim \frac{1}{p'_2} I''_{0,1}$.
\item[(iii)]  For $i=1,2$, the primes $p'_{1,j,i}$, $j=0,\dots,\ell_i-1$ are distinct from the primes $p'_{2,j,i}$, $j=0,\dots,\ell_i-1$.  In particular we have the non-degeneracy condition
\begin{equation}\label{nondeg}
 \prod_{j=0}^{\ell_i-1} p'_{2,j,i} - \prod_{j=0}^{\ell_i-1} p'_{1,j,i} \neq 0
\end{equation}
for $i=1,2$.
\item[(iv)]  There exists a large collection ${\mathcal P}(\vec I'')$ of primes in $[P''/2,P'']$
such that
\begin{equation}\label{pyong}
(\frac{1}{p'_{2,j,i}})_{*} \phi''_{I''_{j,i}} \sim_{\frac{1}{Q}} (\frac{1}{p'_{1,j,i}})_{*} \phi''_{I''_{j+1,i}}
\end{equation}
for all $j=0,\dots,\ell_i-1$ and $i=1,2$, and similarly
\begin{equation}\label{p21}
(\frac{1}{p'_2})_{*} \phi''_{I''_{0,1}} \sim_{\frac{1}{Q}} (\frac{1}{p'_1})_{*} \phi''_{I''_{0,2}},
\end{equation}
where $Q$ is the modulus
\begin{equation}\label{Q-def}
Q \coloneqq \prod {\mathcal P}(\vec I'').
\end{equation}
\end{itemize}

The relationships between the intervals $I''_{j,i}$ can be schematically described by an $\ell_1$-cycle and an $\ell_2$-cycle linked by an edge; see~\cite[Figure 10]{mrt-fourier} for an example of this diagram in the case $\ell_1=\ell_2=4$.

We note that in~\cite{mrt-fourier} the distinctness of the primes $p'_{1,j,i}$ and the primes $p'_{2,j,i}$ in (iii) was not established.  However one can obtain this reduction as follows.  For the sake of notation we eliminate the contribution of the case when one has a collision $p'_{1,0,1}=p'_{2,0,1}$; the other cases are treated similarly.  Firstly observe that from iterating axiom (i) using the equivalence relation and dilation invariance properties of $\sim$, we have
$$ \frac{\prod_{j=0}^{\ell_{i}-1} p'_{1,j,i}}{\prod_{j=0}^{\ell_{i}-1} p'_{2,j,i}} I''_{0,i} \sim I''_{0,i}$$
and hence
\begin{equation}\label{3i0}
\left|  \prod_{j=0}^{\ell_{i}-1} p'_{2,j,i} - \prod_{j=0}^{\ell_{i}-1} p'_{1,j,i} \right| \lesssim \frac{1}{N} (P')^{\ell_{i}}
\end{equation}
for $i=1,2$.  If $p'_{1,0,1} = p'_{2,0,1}$, we can cancel one factor in the $i=1$ case and conclude that
$$
\left|  \prod_{j=1}^{\ell_{1}-1} p'_{2,j,1} - \prod_{j=1}^{\ell_{1}-1} p'_{1,j,1} \right| \lesssim \frac{1}{N} (P')^{\ell_{1}-1}.$$
Using~\cite[Lemma 2.6]{mrt-fourier}, the number of primes $p'_{1,j,i}, p'_{2,j,i}$ that can obey all these constraints is bounded by
$$ \ll \pi_0(P') \frac{d^{\ell_1-1}}{N} \frac{d^{\ell_2}}{N} \ll \frac{d^{\ell_1+\ell_2-1/2} }{N^2}.$$
Since the tuple $\vec I''$ is determined by the quadruple $(I''_{0,1},I''_{0,2}, p'_1, p'_2)  \in {\mathcal Q}$ and the above primes, and since $I''_{0,1}, I''_{0,2}$ uniquely determine $p'_1,p'_2$, we conclude that the number of tuples of this type is bounded by $O( d^{\ell_1+\ell_2+1/2} / N )$, and so these tuples can be removed without significantly affecting the total number of tuples.  Similarly for other collisions.

In a similar spirit, we may improve the non-degeneracy bound property \eqref{nondeg} to
\begin{equation}\label{nondeg-better}
\left|  \prod_{j=0}^{\ell_i-1} p'_{2,j,i} - \prod_{j=0}^{\ell_i-1} p'_{1,j,i} \right| \gg \frac{1}{N} (P')^{\ell_i}
\end{equation}
by the following argument.  Suppose that we had
$$
\left|  \prod_{j=0}^{\ell_i-1} p'_{2,j,i} - \prod_{j=0}^{\ell_i-1} p'_{1,j,i} \right| \leq c \frac{1}{N} (P')^{\ell_i}$$
for some $i=1,2$, and some $c>0$ to be chosen later.  From \eqref{3i0} with $i$ replaced by $3-i$ we also have
\begin{equation}\label{3i}
\left|  \prod_{j=0}^{\ell_{3-i}-1} p'_{2,j,3-i} - \prod_{j=0}^{\ell_{3-i}-1} p'_{1,j,3-i} \right| \lesssim \frac{1}{N} (P')^{\ell_{3-i}}
\end{equation}
By two applications of~\cite[Lemma 2.6]{mrt-fourier}, the number of tuples $(p'_{l,j,i})_{l,i=1,2; j=0,\dots,\ell_i-1}$ of primes in $[P',2P']$ with these properties is $O( c d^{\ell_1+\ell_2} / N^2 )$.  Since the tuple $\vec I''$ is determined by $(I''_{0,1},I''_{0,2}, p'_1, p'_2)$ and the above primes, we conclude that the number of tuples $\vec I''$ arising in this fashion is at most $O( c d^{\ell_1+\ell_2+1}/N )$.  For $c$ small enough, this is less than (say) half of the tuples of $\vec I''$ currently under consideration, so on removing those tuples we obtain the bound \eqref{nondeg-better}.

If we apply Proposition~\ref{basic}(ii) to \eqref{pyong} with the dilation factor
$$ \left( \prod_{0 \leq j' < j} p'_{1,j',i} \right) \left( \prod_{j < j' < \ell_i} p'_{2,j',i} \right) $$
we conclude that
$$ (\frac{1}{a_{j,i}})_{*} \phi''_{I''_{j,i}} \sim_{\frac{1}{Q}} (\frac{1}{a_{j+1,i}})_{*} \phi''_{I''_{j+1,i}}$$
for $j=0,\dots,\ell-1$, where
$$ a_{j,i} \coloneqq \left( \prod_{0 \leq j' < j} p'_{1,j',i} \right) \left( \prod_{j \leq j' < \ell_i} p'_{2,j',i} \right) $$
for $j=0,\dots,\ell_i$.  Observe that the intervals $\frac{1}{a_{j,i}} I''_{j,i}$ all have length $\asymp (P')^{-\ell_i} \frac{H}{P'P''}$ and are comparable to each other in the sense of the relation $\sim$.  Applying Proposition~\ref{basic}(i), (iii) $O(\ell_i)$ times, we conclude that
$$ (\frac{1}{a_{0,i}})_{*} \phi''_{I''_{0,i}} \sim_{\frac{1}{Q}} (\frac{1}{a_{\ell_i,i}})_{*} \phi''_{I''_{0,i}}.$$
Since $a_{0,i}, a_{\ell_i,i}$ are the product of $\ell_i$ distinct primes in $[P',2P']$, we have
\begin{equation}\label{aoli}
 a_{0,i}, a_{\ell_i,i} \asymp (P')^{\ell_i}.
\end{equation}
Also, from the fundamental theorem of arithmetic, once one fixes $I''_{0,1}, I''_{0,2}$, each quadruplet $(a_{0,1},a_{\ell_1,1}, a_{0,2},a_{\ell_2,2})$ is associated to at most $O(1)$ tuples $\vec I''$ (note that from the above axiom (i) that $I''_{j+1,i}$ is uniquely determined by $I''_{j,i}$, $p'_{1,j,i}, p'_{2,j,i}$).  

On the other hand, since $\frac{1}{a_{0,i}} I''_{0,i} \sim \frac{1}{a_{\ell_i,i}} I''_{0,i}$, we have
$$ \left(\frac{1}{a_{\ell_i,i}} - \frac{1}{a_{0,i}}\right) (P')^{-\ell_i} \frac{X}{P'P''} \ll (P')^{-\ell_i} \frac{H}{P'P''} $$
which simplifies using \eqref{aoli}, \eqref{N-def} to
$$ a_{\ell_i,i} - a_{0,i} \ll \frac{(P')^{\ell_i}}{N} .$$
From~\eqref{nondeg-better} we get the corresponding lower bound.  If we set $a_i$ to be the larger of $a_{\ell_i,i}, a_{0,i}$ and $b_i$ to be the smaller, then we have the properties claimed in (ii), (iii) of the proposition, while (i) follows from \eqref{p21}.  

The counting argument at the end of the proof of~\cite[Proposition 4.1]{mrt-fourier} (which is based on the estimate in~\cite[Lemma 2.6]{mrt-fourier}) shows that each quadruple $e$ in ${\mathcal Q}$ is associated to at most $O( d^{\ell_1+\ell_2} / N^2 )$ tuples $\vec I''$ of the above form, and ${\mathcal Q}$ has cardinality $O( dN)$, hence there is a subset ${\mathcal Q}'$ of ${\mathcal Q}$ of cardinality $\gg dN$ such that each $e \in {\mathcal Q}'$ is associated to $\asymp d^{\ell_1+\ell_2}/N^2$ tuples $\vec I''$, which by the previous discussion generates $\asymp d^{\ell_1+\ell_2}/N^2$ quadruples $(a_1,b_1,a_2,b_2)$ obeying the required properties (i), (ii), (iii). The claim follows.
\end{proof}

In this section the precise values of $\ell_1,\ell_2$ are not important; we can select them to be any bounded even integers obeying \eqref{dop}.  In~\cite{mrt-fourier}, $\ell_1,\ell_2$ were essentially chosen to be the minimal even integer obeying \eqref{dop}, so that one could make $a_i-b_i$ as small as possible; however this will convey no significant advantage in our current arguments.

While the above proposition produces a large family ${\mathcal A}_e$ of quadruples $\vec a$ associated to each $e \in {\mathcal Q}'$, in the argument below it will suffice to just use a single such quadruple $\vec a$; this was also the case in the previous paper~\cite{mrt-fourier}.  However, when we work with nilsequences in the next section, it will become necessary to use multiple quadruples $\vec a$ for each $e \in {\mathcal Q}'$.

Thus far we have not exploited the polynomial phase structure of functions in ${\mathcal P}$ beyond the properties in Proposition~\ref{basic} and Proposition~\ref{lsieve}.  Now we make heavier use of this structure in order to ``solve'' the approximate dilation invariance relation \eqref{adi} produced by Proposition~\ref{prop41-analog}, using just a single quadruple from ${\mathcal A}_e$.  The following proposition asserts, roughly speaking, that this equation is only solvable when the local polynomial phases $\phi''_{I''_i}(t)$ ``pretend'' to be like the character $t^{iT}$ on $I''_i$ for some real number $T= T_{I''_1,I''_2} $.  Let us say that a polynomial $\gamma \in \Poly_{\leq k}(\R \to \R)$ is \emph{$Q$-rational} for some $Q$ if it lies in $\Poly_{\leq k}(\frac{q}{Q} \Z \to \Z)$ for some natural number $q$ of polynomial size.

\begin{proposition}[Solving the approximate dilation invariance]\label{sold}  Let the notation and hypotheses be as in Proposition~\ref{prop41-analog}, and write $\phi''_{I''} = (I'', P_{I''})$ for each $I''$.  Then for any of the quadruples $e = (I''_1,I''_2,p'_1,p'_2) \in {\mathcal Q}'$, and any $\vec a = (a_1,b_1,a_2,b_2)$ in ${\mathcal A}_e$, there exists a real number
$$ T = T_{I''_1,I''_2} \ll N^{k+1}$$
and decompositions
$$ P_{I''_i}(t) = \eps_i(t) + \frac{T}{2\pi} \log t + \gamma_i(t) $$
for $i=1,2$ and $t > 0$, where $\eps_i \colon \R^+ \to \R$ is a smooth function obeying the derivative bounds 
$$\eps_i^{(j)}(t) \ll_j |I''_i|^{-j}$$
for all $j \geq 0$ and $t \in I''_i$, and $\gamma_i$ is a $Q$-rational polynomial
with
$$ Q \coloneqq \prod {\mathcal P}_{e,\vec a}.$$
Here $T, \varepsilon_i$ and $\gamma_i$ may depend on $e$ and $\vec a$.

Also, we have
\begin{equation}\label{gap}
 \gamma_1( p'_2 \cdot) = \gamma_2(p'_1 \cdot) \mod \Poly_{\leq k}(\Z \to \Z).
\end{equation}
\end{proposition}

\begin{proof} We abbreviate $P_{I''_i}$ as $P_i$.  From \eqref{adi} we have an identity of the form
$$
P_i( a_i t ) = \eps''_i(t) + P_i(b_i t) + \gamma'_i(t)
$$
where $\eps''_i \in \Poly_{\leq k}(\R \to \R)$ is smooth on $\frac{1}{a_i} I''_1$ and $\gamma'_i \in \Poly_{\leq k}(\frac{1}{Q}\Z \to \Z)$ is $Q$-integral; by a change of variables, we can write this as
\begin{equation}
\label{eq:PiaPibrel} 
P_i( a_i t ) = \eps'_i(a_i t) + P_i(b_i t) + \gamma'_i(t)
\end{equation}
where $\eps'_i \in \Poly_{\leq k}(\R \to \R)$ is now smooth on $I''_1$.  Taking $k^{\mathrm{th}}$ derivatives to make all functions independent of $t$, we conclude in particular that
$$ a_i^k P_i^{(k)} = a_i^k (\eps'_i)^{(k)} + b_i^k P_i^{(k)} + (\gamma'_i)^{(k)}$$
or equivalently
\begin{equation}\label{qpi}
 q_i P_i^{(k)} = a_i^k (\eps'_i)^{(k)} + (\gamma'_i)^{(k)}
\end{equation}
where\footnote{Note that this choice of $q_i$ explains why our bound on $q_i$ in this lemma is a lot weaker that in~\cite[Proposition 4.1]{mrt-fourier}, even if we try to take $\ell_1,\ell_2$ to be as small as possible. Indeed, if $q_i=a_i^k-b_i^k$ with $a_i-b_i$ small, $a_i^k-b_i^k$ may still be relatively large.} $q_i\coloneqq a_i^k - b_i^k$.  As $\gamma_i$ is $Q$-integral, we see on taking $k^{\mathrm{th}}$ divided differences (or using Lemma~\ref{dte}) that $\gamma_i^{(k)}$ is an integer multiple $c_i Q^k$ of $Q^k$.  Thus
$$ q_i P_i^{(k)} = O( a_i^k |I''_i|^{-k} ) + c_i Q^k$$
From \eqref{ab0} we also know that $q_i$ is a natural number of polynomial size; and from the mean value theorem and \eqref{ab3} we have
$$ q_i \asymp \frac{a_i - b_i}{a_i} a_i^k \asymp \frac{1}{N} a_i^k.$$
We thus have
$$ P_i^{(k)} = \frac{c_i}{q_i} Q^k + O( N |I''_i|^{-k} ).$$
Recalling that $x_{I''_i}$ is the midpoint of $I''_i$, we can write the above estimate as
\begin{equation}\label{qpi-2}
 P_i^{(k)} = \frac{c_i}{q_i} Q^k + \frac{(-1)^{k-1} (k-1)!}{2\pi} \frac{T_i}{x_{I''_i}^k} 
\end{equation}
for some real number $T_i$ with the bounds
$$ T_i \ll N \left(\frac{X}{P'P''}\right)^k |I''_i|^{-k} \ll N^{k+1}.$$
Motivated by the Taylor expansion around $x_{I''_i}$, we write
\begin{equation}
\label{eq:PiPitilderel} 
P_i(t) = \tilde \eps_i(t) + \frac{T_i}{2\pi} \log t + \tilde P_i(t) + \tilde \gamma_i(t)
\end{equation}
for $t \in \R^+$, where $\tilde \eps_i: \R^+ \to \R$ is the Taylor remainder
$$ \tilde \eps_i(t) = - \frac{T_i}{2\pi} \log t + \frac{T_i}{2\pi} \log x_{I''_i} + \sum_{j=1}^k \frac{(-1)^{j-1} T_i}{2\pi j} \frac{(t - x_{I''_i})^j}{x_{I''_i}^j} $$ 
which is a smooth function obeying the bounds
$$ \tilde \eps_i^{(j)}(t) \ll_j (H/P'P'')^{-j}$$ 
for $j \geq 0$ and $t \in I''_i$, and $\tilde \gamma_i \in \Poly_{\leq k}(\R \to \R)$ is the function
$$ \tilde \gamma_i(t) \coloneqq \frac{c_i}{q_i} \binom{Qt}{k},$$
and
$$ \tilde P_i(t) \coloneqq P_i(t) - \tilde \gamma_i(t) - \sum_{j=1}^k \frac{(-1)^{j-1} T_i}{2\pi j} \frac{(t - x_{I''_i})^j}{x_{I''_i}^j} - \frac{T_i}{2\pi} \log x_{I_i''}$$
is an element of $\Poly_{\leq k-1}(\R \to \R)$.  
We can then write by~\eqref{eq:PiaPibrel} and~\eqref{eq:PiPitilderel}
\begin{equation}\label{pio}
 \tilde P_i( a_i t ) = \eps^*_i(a_i t) + \tilde P_i(b_i t) + \gamma^*_i(t)
\end{equation}
for $t \in \R^+$, where
$$ \gamma^*_i(t) \coloneqq \gamma'_i(t) + \tilde \gamma_i(b_i t) - \tilde \gamma_i(a_i t) + \left\lfloor \frac{T_i}{2\pi} \log \frac{b_i}{a_i} \right\rfloor$$
and
$$ \eps^*_i(t) \coloneqq \eps'_i(t) + \tilde \eps_i\left(\frac{b_i}{a_i} t\right) - \tilde \eps_i(t) + \left\{ \frac{T_i}{2\pi} \log \frac{b_i}{a_i} \right\}.$$
By construction, $\gamma^*_i$ is an element of $\Poly_{\leq k}(\frac{q_i}{Q}\Z \to \Z)$ that has vanishing $k^{\mathrm{th}}$ derivative, so $\gamma^*_i$ in fact lies in $\Poly_{\leq k-1}(\R \to \R)$. From \eqref{pio} we conclude that $\eps^*_i(a_i t)$ also lies in $\Poly_{\leq k-1}(\R \to \R)$, and from the triangle inequality we have
$$ (\eps_i^*)^{(j)}(t) \ll (H/P'P'')^{-j}$$ 
for all $j \geq 0$ and $t \in I''_i$.  In conclusion, $\tilde P_i$ obeys similar properties to $P_i$ except that all polynomials involved have degree at most $k-1$ instead of at most $k$, and the polynomial $\gamma^*_i$ lies in $\Poly_{\leq k-1}(\frac{q_i}{Q}\Z \to \Z)$  rather than $\Poly_{\leq k-1}(\frac{1}{Q}\Z \to \Z)$.  One can iterate this procedure $k$ times and after collecting terms in the telescoping series, one ends up with a decomposition of the form
$$ P_i(t) = \eps^{**}_i(t) + \frac{T^{**}_i}{2\pi} \log t + P^{**}_i + \gamma^{**}_i(t)$$
for $t \in \R^+$, where $T^{**}_i$ is a real number with
$$ T^{**}_i \ll N^{k+1},$$
$\eps^{**}_i: \R^+ \to \R$ is a smooth function obeying the derivative estimates
$$ (\eps_i^{**})^{(j)}(t) \ll_j (H/P'P'')^{-j}$$ 
for all $j \geq 0$ and $t \in I''_i$, $P^{**}_i \in \R$ is a constant, and $\gamma^{**}_i$ is $Q$-rational.  By splitting $P^{**}_i$ into integer and fractional parts and redistributing these parts to $\gamma^{**}_i$ and $\eps^{**}_i$ respectively, we may assume that $P^{**}_i=0$, thus
\begin{equation}\label{pit}
 P_i(t) = \eps^{**}_i(t) + \frac{T^{**}_i}{2\pi} \log t + \gamma^{**}_i(t)
\end{equation}
for $t \in \R^+$.

This is almost what we need for the claims of the proposition (excluding \eqref{gap}), except that the two real numbers $T^{**}_1, T^{**}_2$ are allowed to be unequal. From \eqref{p21} and Definition~\ref{poly-comp} we have
$$ P_1(p'_2 t) = \eps^\dagger(p'_2 t) + P_2(p'_1 t) + \gamma^\dagger(t)$$
for $t \in \R^+$, where $\eps^\dagger \in \Poly_{\leq k}(\R \to \R)$ is smooth on $I''_{0,1}$, and $\gamma^\dagger$ is $Q$-integral.  Inserting \eqref{pit}, we conclude that
$$ \frac{T^{**}_1}{2\pi} \log(p'_2 t) = \eps^{\dagger \dagger}(p'_2 t) + \frac{T^{**}_2}{2\pi} \log(p'_1 t) + \gamma^{\dagger \dagger}(t)$$
where 
$\eps^{\dagger\dagger} \colon \R^+ \to \R$ is given by the formula
$$ \eps^{\dagger \dagger}(p'_2 t) \coloneqq \eps^\dagger(p'_2 t) + \eps^{**}_2(p'_1 t) - \eps^{**}_1(p'_2 t)$$
and obeys the derivative estimates 
$$(\eps^{\dagger \dagger})^{(j)}(t) \ll_j (H/P'P'')^{-j}$$
for all $j \geq 0$ and $t \in I''_i$, and $\gamma^{\dagger\dagger}$ is given by the formula
\begin{equation}\label{gap-0}
 \gamma^{\dagger\dagger}(t) \coloneqq \gamma^\dagger(t) + \gamma_2^{**}(p'_1 t) - \gamma_1^{**}(p'_2 t).
\end{equation}
and in particular is $Q$-rational.  Let $n_{I''_{0,1}}$ be an integer point of $I''_{0,1}$.  From Lemma~\ref{dte} we see that the first derivative $(\gamma^{\dagger\dagger})'(n_{I''_{0,1}})$ takes values in $\frac{Q}{q^k k!} \Z$ for some $q$ of polynomial size.  We conclude that
\begin{equation}\label{tt}
 \frac{T^{**}_1}{2\pi n_{I''_{0,1}}} = O\left( \frac{P'P''}{H} \right) + \frac{T^{**}_2}{2\pi n_{I''_{0,1}}} \mod \frac{Q}{q^k k!}\Z.
\end{equation}
Since ${\mathcal P}_{I''_1,I''_2}$ is a large set of primes in $[P''/2,P'']$, we see from \eqref{exp-lower} that $Q \gg \exp( c P'' )$ for some $c \gg 1$, so in particular\footnote{\label{foot1}We remark that it is this need for $Q$ to be bigger than $X$ that puts a limit on the range of $H$ where one could possibly prove Theorem~\ref{mult-poly} using the strategy of this paper. Since $P''\leq H^{\varepsilon}$, we must have $H\geq (\log x)^{A}$ for any fixed $A$. It turns out that there are further restrictions on the size of $H$ in our proof, coming from the graph theory part of the proof, where factors of $\ell!$ appear, and also from the Vinogradov--Korobov zero-free region. For these reasons, $H$ actually needs to be at least $\exp((\log X)^{c})$ for some $c \geq 1/2$.} $Q$ exceeds $X^C$ for any fixed $C$ if $X$ is large enough.  But both sides of \eqref{tt} are of polynomial size, and thus have magnitude less than $\frac{Q}{2q^k k!}$ for $X$ large enough.  Hence we may remove the modulus restriction and conclude that
$$ \frac{T^{**}_1}{2\pi n_{I''_{0,1}}} = O\left( \frac{P'P''}{H} \right) + \frac{T^{**}_2}{2\pi n_{I''_{0,1}}} $$
which we can rearrange using \eqref{N-def} as
$$ T^{**}_1 = T^{**}_2 + O( N ).$$
If we set $T \coloneqq T^{**}_1$, 
\[
\gamma_i(t) \coloneqq \gamma^{**}_i(t) + \left\lfloor \frac{T^{**}_i-T}{2\pi} \log n_{I''_{0,1}} \right\rfloor, \text{  and  } \eps_i(t) \coloneqq \eps^{**}_i(t) + \frac{T^{**}_i-T}{2\pi} \log t - \left\lfloor \frac{T^{**}_i-T}{2\pi} \log n_{I''_{0,1}} \right\rfloor,
\]

we obtain all the required claims except for \eqref{gap}.  
But observe that the previous argument in fact showed that the first derivative of $\gamma^{\dagger \dagger}$ vanished at all integer points of $I''_{0,1}$, and thus vanished identically thanks to Lagrange interpolation; hence $\gamma^{\dagger \dagger}$ is in fact an integer constant. The claim \eqref{gap} now follows from \eqref{gap-0} since $\gamma^\dagger$ is already $1$-integral.
\end{proof}

We now follow the arguments in~\cite[\S 5]{mrt-fourier} (starting after the proof of~\cite[Corollary 5.2]{mrt-fourier}.  Let $\delta>0$ be a sufficiently small quantity (depending on $k,\eps,\eta,\theta$) to be chosen later.  We assume $X$ (and hence $H$) to be sufficiently large depending on $\delta$, and allow implied constants to depend on $\delta$.  Define a \emph{good quadruple} to be a tuple $(I'',T,q,\gamma)$ with $I'' \in {\mathcal I}''$, $T$ a real number with
\begin{equation}\label{teo}
 |T| \leq \frac{1}{\delta} N^{k+1},
\end{equation}
and $q$ a natural number with
\begin{equation}\label{qhk}
 1 \leq q \leq X^{1/\delta}
\end{equation}
and $\gamma$ an element of $\Poly_{\leq k}(\frac{q}{\prod {\mathcal P}}\Z \to \Z)$ for some collection ${\mathcal P}$ of primes in $[P''/2,P'']$ of cardinality $\geq \delta \pi_0(P'')$ that do not divide $q$, such that we have a decomposition
\begin{equation}\label{que}
P_{I''}(t) = \eps(t) + \frac{T}{2\pi} \log t + \gamma(t)
\end{equation}
for all $t > 0$, where $\phi''_{I''} = (I'', P_{I''})$, and $\eps \colon \R^+ \to \R$ is a smooth function obeying the estimates
\begin{equation}\label{epso}
 |\eps^{(j)}(t)| \leq \frac{1}{\delta} (H/P'P'')^{-j}
\end{equation}
for $t \in I''$ and $0 \leq j \leq k$.  We also require that $q$ is the least natural number for which $\gamma \in \Poly_{\leq k}(q\Z \to \Z)$.

We will shortly show that Proposition~\ref{prop41-analog} yields a lot of pairs of ``compatible'' good quadruples.

Each interval $I''$ is only associated with a small number of essentially distinct good quadruples.  Indeed, we have

\begin{proposition}\label{prop53-analog}  Let $I'' \in {\mathcal I''}$, let $K$ be a sufficiently large natural number depending on $\delta$, and let $(I'',T_j,q_j,\gamma_j), j=1,\dots,K$ be a collection of good quadruples associated to the interval $I''$.  Then there exist $1 \leq j < j' \leq K$ with the following properties:
\begin{itemize}
\item[(i)] $q_j = q_{j'}$.
\item[(ii)] $\gamma_j = \gamma_{j'} \mod \Z$.  (Here we view $\Z \subset \Poly_{\leq k}(\R \to \R)$ as the group of constant integer functions). 
\item[(iii)] $T_j = T_{j'} + O( N )$.
\end{itemize}
(Recall that we allow implied constants to depend on $\delta$.)
\end{proposition}

\begin{proof}
We modify the proof of~\cite[Proposition 5.3]{mrt-fourier}.  For $j = 1,\dots,K$, let ${\mathcal P}_j$ denote the set of primes in $[P''/2,P'']$ associated to the good quadruple $(I'', T_j, q_j, \gamma_j)$.  Then
$$ \sum_{p'' \in [P''/2,P'']} \sum_{j=1}^K 1_{p'' \in {\mathcal P}_j} \gg K \delta \pi_0(P'')$$
and hence by the prime number theorem we have that
$$ \sum_{j=1}^K 1_{p'' \in {\mathcal P}_j} \gg K \delta $$
for $\gg \delta \pi_0(P'')$ primes $p'' \in [P''/2,P'']$. For $K$ large in terms of $\delta$, we can then find $j, j' \in \{1,\dotsc,K\}$ such that ${\mathcal P} := {\mathcal P}_j \cap {\mathcal P}_{j'}$ contains $\gg_\delta \pi_0(P'')$ primes $p'' \in [P''/2,P'']$. 

From \eqref{que}, we have for all $j = 1,\dots,K$ that
$$P_{I''}(t) = \eps_j(t) + \frac{T_j}{2\pi} \log t + \gamma_j(t)$$
for all $t>0$, where $\phi''_{I''} = (I'', P'')$ and $\eps_j \colon \R^+ \to \R$ is smooth with $\eps_j^{(l)}(t) \ll (H/P'P'')^{-l}$ for all $t \in I''$ and $0 \leq l \leq k$.  Taking first derivatives, we see that the function
\begin{equation}\label{cho}
 \eps'_j(t) + \frac{T_j}{2\pi t} + \gamma'_j(t)
\end{equation}
is independent of $j$.  We now specialize $t$ to an integer point $n_{I''}$ of $I''$.  From Lemma~\ref{dte}, we have $\gamma'_j(n_I) \in \frac{\prod {\mathcal P}}{q^k_j k!} \Z$.  Thus we have
$$ \frac{T_j}{2\pi n_{I''}} =  \frac{T_{j'}}{2\pi n_{I''}} + O\left( \frac{P' P''}{H}\right) \mod \frac{\prod {\mathcal P}}{q^k_jq^k_{j'}k!} \Z$$
for all $j,j' \in \{1,\dots,K\}$.   Both sides of this equation are of polynomial size, while the modulus $\frac{\prod {\mathcal P}}{q^k_jq^k_{j'}k!}$ is far larger than this thanks to \eqref{exp-lower}.  We may thus remove the modulus and conclude that
$$ \frac{T_j}{2\pi n_{I''}} =  \frac{T_{j'}}{2\pi n_{I''}} + O\left( \frac{P' P''}{H}\right) $$
and hence by \eqref{N-def}
$$ T_{j'} = T_j + O( N ),$$
giving the conclusion (iii).  If we now return to the independence of \eqref{cho} in $j$, we conclude that
$$ \gamma'_j(t) - \gamma'_{j'}(t) = O\left( \frac{P' P''}{H}\right) $$
for all $t \in I''$.  By the Bernstein inequality \eqref{ber-2}, we can thus obtain the bound
$$ \gamma^{(l)}_j(n_I) - \gamma^{(l)}_{j'}(n_I) = O\left( \left(\frac{P' P''}{H}\right)^l\right) $$
for all $1 \leq l \leq k$.  On the other hand, from Lemma~\ref{dte} the left-hand side lies in $\frac{\prod {\mathcal P}}{q^k_jq^k_{j'}k!} \Z$.  Using \eqref{exp-lower} as before, we conclude that
$$ \gamma^{(l)}_j(n_I) - \gamma^{(l)}_{j'}(n_I) = 0$$
for $1 \leq l \leq k$, hence by Taylor expansion $\gamma_j$ and $\gamma_{j'}$ differ by a constant, which must lie in $\Z$ since $\gamma_j,\gamma_{j'} \in \Poly_{\leq k}( q_j q_{j'} \Z \to \Z )$.  This gives the conclusion (ii).  Finally, since $q_j$ is the minimal natural number for which $\gamma_j \in \Poly_{\leq k}(q_j\Z \to \Z)$, and $\gamma_j, \gamma_{j'}$ differ by an integer shift, we conclude (i).
\end{proof}

From this and the greedy algorithm, we conclude the following analogue of~\cite[Corollary 5.4]{mrt-fourier}:

\begin{corollary}\label{cor54-analog} For each $I'' \in {\mathcal I}''$ there exists a set ${\mathcal F}(I'')$ of triples $(T',q,\gamma')$ of cardinality 
$$ \# {\mathcal F}(I'') \ll 1$$
such that for any good quadruple $(I'',T,q,\gamma)$ there exists a real number $T'$ and a $\gamma' = \gamma \mod \Z$ such that $(T',q,\gamma') \in {\mathcal F}(I'')$ and
$$ T = T' + O(N).$$
\end{corollary}

Henceforth we fix the finite sets ${\mathcal F}(I'')$. Now we can obtain many pairs of compatible good quadruples:

\begin{proposition}\label{prop55-analog}  For $\gg N \pi_0(P')^2$ pairs $(I''_1,I''_2) \in ({\mathcal I''})^2$, there exist $T_1,T_2,q,\gamma_1,\gamma_2$ with $(T_i,q,\gamma_i) \in {\mathcal F}(I''_i)$ for $i=1,2$ and
\begin{equation}\label{t1}
 T_2 = T_1 + O(N)
\end{equation}
Furthermore, for each such pair, there exist primes $p'_1,p'_2 \in [P',2P']$ coprime to $q$ such that $I''_1$ lies within $100 \frac{H}{P'P''}$ of $\frac{p'_2}{p'_1} I''_2$ with
\begin{equation}\label{p1}
\gamma_1(p'_2 \cdot) = \gamma_2(p'_1 \cdot) \mod \Poly_{\leq k}(\Z \to \Z).
\end{equation}
\end{proposition}

\begin{proof}
This will be a modification of the arguments used to establish~\cite[Proposition 5.5]{mrt-fourier}. From Propositions~\ref{prop41-analog} and~\ref{sold}, we can find a collection ${\mathcal Q}'$ of quadruples $e = (I''_1, I''_2, p'_1, p'_2)$ in ${\mathcal Q}$ of cardinality $\gg N \pi_0(P')^2$, such that to each such quadruple $e$ there exists $T, \eps_1, \eps_2, \gamma_1, \gamma_2, Q$ obeying the conclusions of Proposition~\ref{sold} (for some quadruple $\vec a$, which will play no further role in the arguments).  In particular, each $e \in {\mathcal Q}'$ generates a pair of good quadruples $(I''_1, T_1, q_1, \gamma_1)$, $(I''_2, T_2, q_2, \gamma_2)$ for some $\gamma_1 \in \Poly_{\leq k}(q_1 \Z \to \Z)$, $\gamma_2 \in \Poly_{\leq k}(q_2 \Z \to \Z)$ obeying \eqref{t1}, \eqref{p1}.
By Corollary~\ref{cor54-analog} we may adjust these good quadruples so that $(T_i, q_i, \gamma_i) \in {\mathcal F}(I''_i)$ for $i=1,2$.

At present it is possible that $p'_i$ divides $q_j$ for some $i,j=1,2$.  But, as noted in~\cite[Proposition 5.5]{mrt-fourier}, for each $q_j$ there are only at most $O(1)$ such $p'_i$ that can do this, and by the bounded cardinality of the ${\mathcal F}(I''_i)$, the total number of quadruples $e = (I''_1, I''_2, p'_1, p'_2)$ that generate such a situation is $O( N \pi_0(P') )$, which is negligible compared to the cardinality of ${\mathcal Q}'$. Thus by refining ${\mathcal Q}'$ we may assume that $p'_1,p'_2$ do not divide $q_1$ or $q_2$.

We now claim that $q_1$ and $q_2$ are equal.  By the definition of a good quadruple, $\gamma_1$ lies in $\Poly_{\leq k}(q_1\Z \to \Z)$; by \eqref{p1} this implies that $\gamma_2$ lies in $\Poly_{\leq k}(p'_2 q_1\Z \to \Z)$.  On the other hand, $q_2$ is the minimal natural number for which $\gamma_2$ lies in $\Poly_{\leq k}(q_2\Z \to \Z)$; by Lemma~\ref{bezout}, this implies that $q_2$ divides $p'_2 q_1$, and similarly $q_1$ divides $p'_1 q_2$.  Since $p'_1, p'_2$ do not divide $q_1,q_2$, we obtain $q_1=q_2$, and the claim follows.
\end{proof}

As in~\cite[\S 5]{mrt-fourier}, on the space $Z$ of triples $(T,q,\gamma)$ with $T \in \R$, $q \geq 1$, $\gamma \in \Poly(q\Z \to \Z)$ we define the metric
$$ d((T_1,q_1,\gamma_1), (T_2,q_2,\gamma_2)) \coloneqq c(\delta) \frac{1}{N} |T_1-T_2| + 1_{q_1 \neq q_2} + \frac{1}{100} 1_{\gamma_1 \neq \gamma_2}$$
with some sufficiently small constant $c(\delta)>0$.  Proposition~\ref{prop55-analog} provides one with a collection ${\mathcal S}$ of sextuples $(I''_1, I''_2, (T_1,q_1,\gamma_1), (T_2,q_2,\gamma_2), p'_1, p'_2)$ of cardinality $\gg N \pi_0(P')^2$ such that
$$ d((T_1,q_1,\gamma_1), (T_2,q_2,\gamma_2)) \leq \frac{1}{10}.$$
Applying the mixing lemma in~\cite[Corollary 5.2]{mrt-fourier}, we conclude that there exists a triple $(T_0,q_0,\gamma_0) \in Z$ and a collection ${\mathcal T}$ of quadruples $(I'',T,q,\gamma)$ with $I'' \in {\mathcal I}''$, $(T,q,\gamma) \in {\mathcal F}(I'')$, and $d((T,q,\gamma), (T_0,q_0,\gamma_0)) \leq \frac{1}{5}$ such that
$$ \# {\mathcal T} \gg N$$
and such that there are $\gg Nd$ sextuples $(I''_1,I''_2, (T_1,q',\gamma_1), (T_2,q',\gamma_2), p'_1,p'_2)$ such that $(I''_i,T_i,q',\gamma_i) \in {\mathcal T}$ and $p'_1,p'_2$ distinct primes in $[P',2P']$ with $I''_1$ lying within $100 \frac{H}{P'P''}$ of $\frac{p'_2}{p'_1} I''_2$ (so in particular $I''_1 \sim \frac{p'_2}{p'_1} I''_2$), with $p'_1, p'_2$ coprime to $q'$, and obeying the properties \eqref{t1}, \eqref{p1}.

In particular, if $(I'',T,q,\gamma)\in {\mathcal T}$, then $q=q_0$ and 
\begin{equation}\label{tan}
 T = T_0 + O(N).
\end{equation}
From this and \eqref{teo}, we conclude in particular that
\begin{equation}\label{teo-2}
T_0 \ll N^{k+1}.
\end{equation}
At present our upper bound \eqref{qhk} on $q=q_0$ is quite large (and significantly worse than in~\cite{mrt-fourier}).  Nevertheless, we can improve the bound on $q_0$ after first establishing the following variant of~\cite[Lemma 2.6]{mrt-fourier}:

\begin{lemma}\label{lem26-analog}  Let $m,\ell \in \N$ and $P',N \geq 3$ be such that $(P')^{\ell-1} \gg N$.  Let $q \geq 1$. Then the number of $2\ell$-tuples $(p'_{1,1},\dots,p'_{1,\ell},p'_{2,1},\dots,p'_{2,\ell})$ of primes in $[P',2P']$ not dividing $q$ obeying the condition 
$$ \left| \prod_{j=1}^\ell p'_{2,j} - \prod_{j=1}^\ell p'_{1,j} \right| \leq C \frac{(P')^\ell}{N}$$
and
$$ \prod_{j=1}^\ell (p'_{2,j})^m = \prod_{j=1}^\ell (p'_{1,j})^m \mod q$$
for some $C \geq 1$ is bounded by
$$ \ll_{\ell,C,m} \frac{d^\ell}{N} \left( \frac{m^{\omega(q)} }{\phi(q)} + \frac{1}{\log N} \right),$$
where $\omega(q)$ denotes the number of prime factors of $q$.
\end{lemma}

\begin{proof}  This follows the same Dirichlet character argument used to prove~\cite[Lemma 2.6]{mrt-fourier}, with the one main difference being that the indicator $1_{\chi = \chi_0}$ is replaced by $1_{\chi^m = \chi_0}$.  This latter condition is attained for at most $m^{\omega(q)}$ characters $\chi$ with period $q$, explaining the additional factor of $m^{\omega(q)}$ here compared with~\cite[Lemma 2.6]{mrt-fourier}.
\end{proof}

We now have

\begin{proposition}\label{prop56-analog} $q_0 \ll 1$.
\end{proposition}

\begin{proof}  This will be a modification of the proof of~\cite[Proposition 5.6]{mrt-fourier}, using Lemma~\ref{lem26-analog} in place of~\cite[Lemma 2.6]{mrt-fourier}.  Let $\ell$ be the first even natural number such that $d^\ell \geq N^{2+\eps}$. Arguing as in the proof of~\cite[Proposition 5.6]{mrt-fourier}, we can find $\gg d^\ell$ tuples
$$ (Q_0,\dots,Q_{\ell-1}) \in {\mathcal T}^\ell$$
such that if we write $Q_j = (I''_j, T_j, q_0, \gamma_j)$ for $j=0,\dots,\ell$ (with the convention $Q_{\ell} = Q_0$) then for each $j=0,\dots,\ell-1$, there exist primes $p'_{j,1}, p'_{j,2} \in [P',2P']$ such that
$$
\gamma_j(p'_{j,2} \cdot) = \gamma_{j+1}(p'_{j,1} \cdot) \mod \Poly_{\leq k}(\Z \to \Z)$$
and such that $I''_j \sim \frac{p'_{j,2}}{p'_{j,1}} I''_{j+1}$.  From the first claim we have
$$ \gamma_j\left( (\prod_{i=0}^{j-1} p'_{i,1}) (\prod_{i=j}^{\ell-1} p'_{i,2}) \cdot \right) = 
\gamma_{j+1}\left( (\prod_{i=0}^{j} p'_{i,1}) (\prod_{i=j+1}^{\ell-1} p'_{i,2}) \cdot \right) \mod \Poly_{\leq k}(\Z \to \Z)$$
for $j=0,\dots,\ell-1$, which by transitivity implies that
\begin{equation}\label{gap2}
 \gamma_0\left( (\prod_{i=0}^{\ell-1} p'_{i,2}) \cdot \right) = 
\gamma_0\left( (\prod_{i=0}^{\ell-1} p'_{i,1}) \cdot \right) \mod \Poly_{\leq k}(\Z \to \Z).
\end{equation}
Similarly, we have that $I''_0 \sim \frac{\prod_{i=0}^{\ell-1} p'_{i,2}}{\prod_{i=0}^{\ell-1} p'_{i,1}} I''_0$, which implies that
$$\prod_{i=0}^{\ell-1} p'_{i,2} - \prod_{i=0}^{\ell-1} p'_{i,1} \ll \frac{(P')^\ell}{N}.$$

Now we analyze the condition \eqref{gap2}.   We write the polynomial $\gamma_0$ as
$$ \gamma_0(t) = \sum_{m=0}^k \frac{a_m}{b_m} t^m$$
where $b_m$ are natural numbers and each $a_m$ is an integer coprime to $b_m$.  Clearly $\gamma_0 \in \Poly_{\leq k}(b_1 \dots b_k \Z \to \Z)$, and hence $q_0 \leq b_1 \dots b_k$.  In particular, there exists $1 \leq m \leq k$ such that $b_m \geq q_0^{1/k}$.  From \eqref{gap2} and Lemma~\ref{dte}, and extracting the $t^m$ coefficient, we see that
$$ \left(\prod_{i=0}^{\ell-1} p'_{i,2}\right)^m \frac{a_m}{b_m} = \left(\prod_{i=0}^{\ell-1} p'_{i,1}\right)^m \frac{a_m}{b_m} \mod \frac{1}{k!} \Z$$
and hence
$$ \left(\prod_{i=0}^{\ell-1} p'_{i,2}\right)^m = \left(\prod_{i=0}^{\ell-1} p'_{i,1}\right)^m \mod \frac{b_m}{(b_m,k!)}.$$
By Lemma~\ref{lem26-analog} (and bounding $\frac{m^{\omega(q)} }{\phi(q)}  \ll q^{-1/2}$, say), we conclude that the total number of tuples of primes $(p'_{i,1}, p'_{i,2})_{0 \leq i < \ell}$ is at most 
$$ \ll \frac{d^\ell}{N} \left( q_0^{-1/2k} + \frac{1}{\log X} \right).$$
Since there are $\ll N$ choices for the interval $I_1''$, and $I_1''$ and $(p'_{i,1}, p'_{i,2})_{0 \leq i < \ell}$ determine the other $I_j''$, and we have $\#{F}(I_j'')\ll 1$, we deduce that the number of tuples $(Q_0,\ldots, Q_{\ell-1})\in \mathcal{T}^{\ell}$ is in fact $\ll d^{\ell}(q_0^{-1/2k} + (\log X)^{-1})$.
Comparing with the lower bound we had for the number of these tuples, we must have
$$ q_0^{-1/2k} + \frac{1}{\log X} \gg 1,$$
giving the claim.
\end{proof}

Let $(I'',T,q_0,\gamma) \in {\mathcal T}$, then from \eqref{falp} one has
$$
\left| \sum_{n \in I''} f(n) e(-P_{I''}(n)) \right| \gg |I''|.
$$
Let $H^* \coloneqq c \frac{H}{P'P''}$ for a sufficiently small $c>0$.  Then one has
$$ 
\sum_{n \in I''} f(n) e(-P_{I''}(n))  
= \frac{1}{H^*} \int_{I''} \sum_{n \in [x,x+H^*]} f(n) e(-P_{I''}(n)) \ dx + O( H^* )$$
and thus by the triangle inequality we have (for $c$ small enough)
\begin{equation}\label{abs}
\int_{I''} \left| \sum_{n \in [x,x+H^*]} f(n) e(-P_{I''}(n)) \right| \ dx \gg |I''| H^*.
\end{equation}
For $n \in [x,x+H^*] \cap \Z$, we have from \eqref{que} that
$$
P_{I''}(n) = \eps(n) + \frac{T}{2\pi} \log n + \gamma(n);$$
from \eqref{epso} we have
$$ \eps(n) = \eps(x) + O( c )$$
while from \eqref{tan} one has
$$ \frac{T}{2\pi} \log n  = \frac{T_0}{2\pi} \log n + \frac{T-T_0}{2\pi} \log x + O( c ).$$
The effect of the $O( c)$ error to \eqref{abs} is negligible if $c$ is small enough, and the constant terms $\eps(x), \frac{T-T_0}{2\pi} \log x $ disappear once the absolute value signs in \eqref{abs} are applied.  We conclude that
$$
\int_{I''} \left| \sum_{n \in [x,x+H^*]} f(n) n^{-iT_0} e(-\gamma(n))\right| \ dx \gg |I''| H^*.
$$
The function $e(-\gamma(n))$ is periodic modulo $q_0$.  Since $q_0 = O(1)$, we can expand $e(\gamma(n))$ as a linear combination of $O(1)$ functions of the form $1_{q_1|n} \chi(n/q_1)$, where $q_1$ divides $q_0$ and $\chi$ is a Dirichlet character of period $q_0/q_1$.  We conclude that there exists $q_1, \chi$ of this form such that
$$
\int_{I''} \left| \sum_{n \in [x,x+H^*]} f(n) n^{-iT_0} 1_{q_1|n} \overline{\chi}(n/q_1) \right| \ dx \gg |I''| H^*.
$$
Since each $I''$ is associated to $O(1)$ quadruples in ${\mathcal T}$, there are $\gg X/H$ intervals $I'' \in {\mathcal I}''$ for which we have an estimate of this form.  At present $q_1, \chi$ can depend on $I''$, but there are only $O(1)$ choices for these quantities, so by the pigeonhole principle we may make $q_1,\chi$ independent of $I''$, while still retaining $\gg X/H$ intervals.  Summing in these intervals, we conclude that
$$ \int_{X/4P'P''}^{4X/P'P''} \left| \sum_{n \in [x,x+H^*]} f(n) n^{-iT_0} 1_{q_1|n} \overline{\chi}(n/q_1) \right| \ dx \gg \frac{X}{P'P''} H^*.$$
Arguing exactly as in the final part of~\cite[\S 5]{mrt-fourier} (namely, applying the complex-valued version~\cite{MRT} of the main result from~\cite{mr}), we conclude that
$$ M( f; T, Q ) \ll 1$$
for some $T \ll \frac{X^{k+1}}{H^{k+1}}$ and $Q \ll 1$, and Theorem~\ref{mult-poly} follows.

\section{Local correlation with nilsequences}\label{nilseq}
\subsection{The set-up} 
In this section we prove Theorem~\ref{mult-pret}.  Our argument shall closely follow in large parts the proof of Theorem~\ref{mult-poly}, except that the space $\Phi$ of local polynomial phases will be replaced by a different family $\Psi$ of local nilsequences, and significantly more effort needs to be expended to ``solve'' the approximate dilation invariance ``equations''.

Recall that a \emph{degree $k$ filtered nilmanifold} $G/\Gamma$ is a quotient space $G/\Gamma$, where
\begin{itemize}
\item $G$ is a connected, simply connected Lie group equipped with a filtration $G_\bullet = (G_i)_{i \geq 0}$ of closed connected subgroups $G_i$, with $G_0=G_1=G$, $G_i \supset G_{i+1}$ for all $i$, $G_i = \{1\}$ for $i > k$, and $[G_i,G_j] \subset G_{i+j}$ for $i,j \geq 0$ (note in particular that this implies that $G$ is nilpotent);
\item $\Gamma$ is a discrete subgroup of $G$ such that the subgroups $\Gamma_i \coloneqq G_i \cap \Gamma$ are cocompact subgroups of $G_i$ for each $i$, so that the quotient spaces $G_i/\Gamma_i$ are all compact.
\end{itemize}

Let $G$ be a connected, simply connected nilpotent Lie group.  Then $G$ is isomorphic to a matrix Lie group (a Lie group consisting of invertible $n \times n$ complex matrices for some $n$); see e.g.,~\cite[Proposition 16.2.6]{hilgert}, and so for the following discussion we may assume without loss of generality that $G$ is a matrix Lie group.  The Lie algebra of $G$, defined as the tangent space of $G$ at the identity, will be denoted $\log G$.  The matrix exponential map $\exp \colon \log G \to G$ is then a diffeomorphism (see e.g.,~\cite[Corollary 11.2.7]{hilgert}), and hence we have a well-defined logarithm map $\log\colon G \to \log G$ inverting this map; similarly we have the diffeomorphism $\log \colon G_i \to \log G_i$ where $\log G_i$ is the Lie algebra of $G_i$.  We define exponentiation $g^t$ for any $g \in G$ and $t \in \R$ by the familiar formula
\begin{equation}\label{realexp}
 g^t \coloneqq \exp( t \log g ),
\end{equation}
so in particular $\log(g^t) = t \log g$.
We place an arbitrary Euclidean metric on the vector space $\log G$, and allow implied constants to depend on $G$ and this metric.  
 If $g \in G$ and $X > 0$, we then write $g = O(X)$ as shorthand
 for $|\log g| =O(X)$.  We also place an arbitrary smooth metric $d$ on $G/\Gamma$ (for instance, one could take the Carnot--Carath\'eodory metric associated to the metric on $\log G$, although it is not essential here that we do so), and define the Lipschitz norm of a function $F \colon G/\Gamma \to \C$ to be
$$ \|F\|_{\mathrm{Lip}} \coloneqq \sup_{x \in G/\Gamma} |F(x)| + \sup_{x,y \in G/\Gamma: x \neq y} \frac{|F(x)-F(y)|}{d(x,y)}$$
and call a function $F$ \emph{Lipschitz continuous} if its Lipschitz norm is finite.

The presence of the logarithm here may seem strange to those accustomed to more ``abelian'' analysis, but for nilpotent groups (written multiplicatively) one should view $\log$, $\exp$, and $(g,t) \mapsto g^t$ as polynomial maps, as the following example illustrates:

\begin{example}[Heisenberg group]\label{heisen}  Take $G$ to be the Heisenberg group $G = \begin{pmatrix} 1 & \R & \R \\ 0 & 1 & \R \\ 0 & 0 & 1 \end{pmatrix}$, with filtration $G_0=G_1=G$, $G_2 = \begin{pmatrix} 1 & 0 & \R \\ 0 & 1 & 0 \\ 0 & 0 & 1 \end{pmatrix}$, and $G_i=\left\{\begin{pmatrix} 1 & 0 & 0 \\ 0 & 1 & 0 \\ 0 & 0 & 1 \end{pmatrix}\right\}$ for all $i>2$.  Then $\log G = \begin{pmatrix} 0 & \R & \R \\ 0 & 0 & \R \\ 0 & 0 & 0 \end{pmatrix}$ and 
$$\exp \begin{pmatrix} 0 & x & z \\ 0 & 0 & y \\ 0 & 0 & 0 \end{pmatrix} = \begin{pmatrix} 1 & x & z + \frac{xy}{2} \\ 0 & 1 & y \\ 0 & 0 & 1 \end{pmatrix},$$
and hence 
$$\log \begin{pmatrix} 1 & x & z \\ 0 & 1 & y \\ 0 & 0 & 1 \end{pmatrix} = \begin{pmatrix} 0 & x & z - \frac{xy}{2} \\ 0 & 0 & y \\ 0 & 0 & 0 \end{pmatrix}$$
for any $x,y,z \in\R$.
In particular we have
$$ \begin{pmatrix} 1 & x & z \\ 0 & 1 & y \\ 0 & 0 & 1 \end{pmatrix}^t = \begin{pmatrix} 1 & tx & tz + \frac{t(t-1)}{2} xy \\ 0 & 1 & ty \\ 0 & 0 & 1 \end{pmatrix}$$
for any $x,y,z,t \in \R$, and $\begin{pmatrix} 1 & x & z \\ 0 & 1 & y \\ 0 & 0 & 1 \end{pmatrix}  = O(X)$ if and only if $x,y,z-\frac{xy}{2} = O(X)$.
\end{example}

From the identity $\log g^{-1} = -\log g$ we see that if $g=O(X)$ then $g^{-1}=O(X)$.  Similarly, from the Baker--Campbell--Hausdorff formula \eqref{ast-def}, \eqref{poly} we see that that $\log(gh)$ is a polynomial function of $\log g, \log h$ (with degree and coefficients $O(1)$), and hence if $g,h = O(X)$ then $gh = O(X^{O(1)})$.

We define $\Poly(\R \to G)$ to be the space of all maps $g\colon \R \to G$ of the form
$$ g(t) \coloneqq \exp( \sum_{i=0}^k X_i t^i )$$
where $X_i \in \log G_i$ for $i=0,\dots,k$.  From the Baker--Campbell--Hausdorff formula \eqref{ast-def}, \eqref{poly}, \eqref{gij-inc} we see that $\Poly(\R \to G)$ is a group with respect to multiplication.  For any $\delta>0$, we define $\Poly(\delta \Z \to G)$ to be the set of all maps $g \colon \delta \Z \to G$ such that
$$ \partial_{h_1} \dots \partial_{h_i} g(t) \in G_i$$
for all $i \geq 0$ and $h_1,\dots,h_i,t \in\delta \Z$, where $\partial_h g(t) \coloneqq g(t+h) g(t)^{-1}$.  We similarly define $\Poly(\delta \Z \to \Gamma)$ by replacing $G_i$ with $\Gamma_i$ in the above definition; equivalently, $\Poly(\delta \Z \to \Gamma)$ consists of those elements of $\Poly(\delta \Z \to G)$ that take values in $\Gamma$. We refer to elements of $\Poly(\R \to G)$ and $\Poly(\delta \Z \to G)$ as \emph{polynomial maps}.    We have the following basic fact:

\begin{lemma}\label{uniq}  Let $\delta > 0$.  Then every element $\tilde g$ of $\Poly(\R \to G)$ restricts to an element $g$ of $\Poly(\delta \Z \to G)$; conversely, every element $g$ of $\Poly(\delta \Z \to G)$ has a unique extension to an element $\tilde g$ of $\Poly(\R \to G)$.  Finally, $\Poly(\delta \Z \to \Gamma)$ forms a group.
\end{lemma}

\begin{proof}  See Appendix~\ref{bch}.
\end{proof}

In view of this lemma we shall abuse notation by identifying $\Poly(\delta \Z \to G)$ with $\Poly(\R \to G)$, and viewing each of the $\Poly(\delta \Z \to \Gamma)$ as subgroups of $\Poly(\R \to G)$.  We will refer to polynomial maps in $\Poly(\delta \Z \to \Gamma)$ as being \emph{$\frac{1}{\delta}$-integral}.

Applying the inverse conjecture for the Gowers norms as in~\cite[\S 4]{TaoEq}, \cite[\S C]{gtz} we see that Theorem~\ref{mult-pret} follows from (and is in fact equivalent to) the following claim:

\begin{theorem}[Non-pretentious multiplicative functions do not correlate with nilsequences on short intervals on average]\label{mult-nil}  Let $k \geq 0$ be a non-negative integer, and let $0 < \theta < 1$.  Let $G/\Gamma$ be a degree $k$ filtered nilmanifold, and let $F \colon G/\Gamma \to \C$ be a Lipschitz function.  Suppose that $f\colon \N \to \C$ is a multiplicative $1$-bounded function, and suppose that $X \geq 1$, $X^\theta \leq H \leq X^{1-\theta}$, and $\eta > 0$ are such that
$$ \int_X^{2X} \sup_{g \in \Poly( \R \to G )} \left| \sum_{n \in [x,x+H]} f(n) \overline{F}(g(n) \Gamma)\right| \ dx \geq \eta H X.$$
Then one has
\begin{equation}
\label{eq:PretCond}
M(f; CX^{k+1} / H^{k+1}, Q ) \ll_{k,\eta,\theta,F,G/\Gamma} 1
\end{equation}
for some $C, Q \ll_{k,\eta,\theta,F,G/\Gamma} 1$.
\end{theorem}

We note that in order to prove Theorem~\ref{mult-pret} it suffices to prove Theorem~\ref{mult-nil} with $F$ fixed since by Arzel\`a--Ascoli the family of Lipschitz functions $F$ on $G / \Gamma$ of bounded norm is precompact in the uniform topology, and moreover we can modify $F$ in the uniform norm by anything less than $\eta / 10$, say, without significantly affecting the assumption of Theorem~\ref{mult-nil} (i.e changing $\geq \eta H X$ to $\geq \eta H X / 2$, say). As a result we can restrict to a finite set of $F$'s and thus to a fixed $F$ by pigeonholing. 

As in the previous section, at present it is only the values of $g$ on $\Z$ that are relevant, but once one begins exploiting the dilation structure of $\R$ it becomes convenient to view $g$ as a polynomial map on all of $\R$ and not just on $\Z$. As remarked in the introduction, in~\cite{HeWang}  a variant of this estimate was established in which the supremum in $g$ was placed outside the integral, and in which $H$ was allowed to grow in $X$ arbitrarily slowly rather than at a polynomial rate; see also~\cite{flam} for an earlier partial result in this direction.

We prove Theorem~\ref{mult-nil} by induction on the dimension $\mathrm{dim}(G/\Gamma) = \mathrm{dim}(G)$ of the nilmanifold $G/\Gamma$ (keeping $k$ fixed).  When $\mathrm{dim}(G/\Gamma)=0$, the function $F(g(n)\Gamma)$ is constant, and the claim corresponds to $k=0$ case of Theorem~\ref{mult-poly} which in turn essentially followed from the result in~\cite{MRT}. Hence we assume inductively that $\mathrm{dim}(G/\Gamma) \geq 1$, and that the claim has already been proven for all $G$ of smaller dimension.  We now fix $k,\eta,\theta,F,G/\Gamma$, and allow implied constants to depend on these quantities.  Thus we have
\begin{equation}\label{xi}
 \int_X^{2X} \sup_{g \in \Poly( \R \to G )} \left| \sum_{n \in [x,x+H]} f(n) \overline{F}(g(n) \Gamma)\right| \ dx \gg H X,
\end{equation}
and our objective is to show that
$$ M(f; C X^{k+1} / H^{k+1}, Q ) \ll 1$$
for some $C, Q = O(1)$.  We may normalize $F$ to be bounded in magnitude by $1$, so that the sequences $n \mapsto \overline{F}(g(n)\Gamma)$ are $1$-bounded.  As in the previous section, we also introduce a small parameter $\eps>0$ that can depend on $k,\eta,\theta,F,G/\Gamma$, and allow implied constants to also depend on $\eps$ unless otherwise specified.

\subsection{Initial reductions}
We first make a minor but convenient reduction, namely that we restrict to the case when $f$ is completely multiplicative rather than merely multiplicative (cf.~\cite[Proposition 10]{Tao}).  If we let $f_1$ be the completely multiplicative function that equals $f$ at each prime $p$, then we can write $f$ as a Dirichlet convolution $f(n) = \sum_{d=1}^\infty 1_{d|n} f_1(\frac{n}{d}) h(d)$ for some multiplicative function $h$ with $h(p)=0$ and $|h(p^j)| \leq 2$ for all $j \geq 2$ (in fact $h(p^j) = f(p^j) - f(p) f(p^{j-1})$).  From \eqref{xi} and the triangle inequality, we thus have
$$
\sum_{d=1}^\infty |h(d)| \int_X^{2X} \sup_{g \in \Poly( \R \to G )} \left| \sum_{n \in [x,x+H]} 1_{d|n} f_1(\frac{n}{d}) \overline{F}(g(n) \Gamma)\right| \ dx \gg H X.$$
From Euler products we see that $\sum_{d=1}^\infty \frac{|h(d)|}{d^{2/3}} \ll 1$ (say), so by the pigeonhole principle there exists $d \geq 1$ such that
$$
\int_X^{2X} \sup_{g \in \Poly( \R \to G )} \left| \sum_{n \in [x,x+H]} 1_{d|n} f_1(\frac{n}{d}) \overline{F}(g(n) \Gamma)\right| \ dx \gg d^{-2/3} H X.$$
The left-hand side can be trivially bounded by $O( d^{-1} HX)$, hence $d = O(1)$.  Making the change of variables $n = dn'$ and $x = dx'$, we then have
$$
\int_{X/d}^{2X/d} \sup_{g \in \Poly( \R \to G )} \left| \sum_{n' \in [x',x'+H/d]} f_1(n') \overline{F}(g(dn') \Gamma)\right| \ dx \gg (H/d) (X/d).$$
Note that if $g$ lies in $\Poly(\R \to G)$ then the dilation $g(d\cdot)$ does also.
Applying Theorem~\ref{mult-nil} for the completely multiplicative function $f_1$ (adjusting $\theta$ slightly to retain the hypothesis $X^\theta \leq H \leq X^{1-\theta}$), we conclude that
$$ M(f_1; C (X/d)^{k+1} / (H/d)^{k+1}, Q ) \ll 1$$
and the claim follows.

It remains to establish the claim for completely multiplicative $f$.  Assume for contradiction that this claim is false.  Then we can find a sequence $X = X_{\n} \geq 1$ of real numbers and a sequence $f = f_\n$ of $1$-bounded completely multiplicative functions, such that \eqref{xi} holds uniformly in $\n$, but such that
\begin{equation}\label{dq}
 M(f; C X^{k+1} / H^{k+1}, Q ) \to \infty
\end{equation}
as $\n \to \infty$ for any fixed $Q, C$, where $H = H_\n$ lies in the interval $[X_\n^\theta, X_\n^{1-\theta}]$.  Among other things, this implies that $X \to \infty$ as $\n \to \infty$.  We now restrict attention to $\n$ sufficiently large, so that $X$ can be made larger than any fixed constant.
Henceforth we suppress the dependence of $X, H, f$ on $\n$.  We refer to a quantity as \emph{fixed} if it is independent of $\n$, and use the asymptotic notation $Y = o(Z)$ to denote the claim $|Y| \leq c(\n) Z$ for some quantity $c(\n)$ that may depend on fixed quantities, but goes to zero as $\n \to \infty$.  From the induction hypothesis, we conclude that
$$ \int_X^{2X} \sup_{\tilde g \in \Poly( \R \to \tilde G )} \left| \sum_{n \in [x,x+H]} f(n) \overline{\tilde F}(\tilde g(n) \tilde \Gamma)\right| \ dx = o(HX)$$
whenever $\tilde G/\tilde \Gamma$ is a fixed degree $k$ filtered nilmanifold of dimension strictly less than that of $G/\Gamma$, and $\tilde F: \tilde G/\tilde \Gamma\to \mathbb{C}$ is a fixed Lipschitz function.  More generally, for any fixed Dirichlet character $\chi$, we see from \eqref{dq} and enlarging $Q$ that
$$ M(f\chi; C X^{k+1} / H^{k+1}, Q ) \to \infty$$
for any fixed $C$, and hence
$$ \int_X^{2X} \sup_{\tilde g \in \Poly( \R \to \tilde G )} \left| \sum_{n \in [x,x+H]} f(n) \chi(n) \overline{\tilde F}(\tilde g(n) \tilde \Gamma)\right| \ dx = o(HX).$$
By multiplicative Fourier expansion we thus have
\begin{equation}\label{mul}
\int_X^{2X} \sup_{\tilde g \in \Poly( \R \to \tilde G )} \left| \sum_{n \in [x,x+H]} f(n) 1_{n = a\mod q} \overline{\tilde F}(\tilde g(n) \tilde \Gamma)\right| \ dx = o(HX)
\end{equation}
for any fixed natural number $q$ and any fixed $a$ coprime to $q$.  Because $f$ is completely multiplicative, we also see that the same claim is true when $a$ shares a common factor $d$ with $q$, after rescaling $X,H,x,n$ by $d$ as before (and expressing sum over the shrunken interval $[x/d, x/d+H/d]$ as an average of sums over intervals of length $(X/H)^{\theta/2}$, plus negligible error).

Among other things, this allows us to eliminate ``major arc'' cases of \eqref{xi}.  Define a \emph{rational} subgroup of $G$ to be a closed subgroup $\tilde G$ of $G$ for which $\tilde G \cap \Gamma$ is cocompact in $\tilde G$.

\begin{proposition}[Major arc case]\label{major}  Assume that $f$ satisfies \eqref{dq}. Let $\tilde G$ be a fixed connected rational subgroup of $G$, and suppose that $\tilde G$ is a proper subgroup in the sense that $\mathrm{dim}(\tilde G) < \mathrm{dim}(G)$ (or equivalently\footnote{This is because $\tilde G \neq G$ is equivalent to $\log \tilde G$ being a proper subspace of $\log G$.}, $\tilde G \neq G$).  We endow $\tilde G$ with the filtration $\tilde G_i \coloneqq G_i \cap \tilde G$ induced from $G$.  Let $q$ be a fixed natural number, and let $E$ be a fixed compact subset of $\tilde G$.  Then
$$
\int_X^{2X} \sup_{\substack{\eps \in E\\ \tilde g \in \Poly( \R \to \tilde G )\\ \gamma \in \Poly(q\Z \to \Gamma)}} \left| \sum_{n \in [x,x+H]} f(n) \overline{F}(\eps \tilde g(n) \gamma(n) \Gamma)\right| \ dx = o(HX).$$
\end{proposition}

\begin{proof} Since $F$ is a Lipschitz function, and $E$ is compact it suffices to verify the Theorem for a single choice of $\eps$.
  Next, we claim that the quotient space $\Poly(q\Z \to \Gamma) / \Poly(\Z \to \Gamma)$ is finite.  Indeed, from Taylor expansion we see that if $\gamma \in \Poly(q\Z \to \Gamma)$, then $\gamma(\Z)$ takes values in the group $\Gamma'$ generated by the roots $\{ \gamma^{1/q^k}: \gamma \in \Gamma \}$ of $\Gamma$.  As noted at the end of Appendix~\ref{bch}, $\Gamma$ has finite index in $\Gamma'$, so there are only finitely many possibilities for the tuple $(\gamma(0),\dots,\gamma(k))$ modulo right multiplication by elements of $\Gamma^{k+1}$.  As this tuple uniquely determines the polynomial map $\gamma$, we conclude that there are only finitely many possibilities for $\gamma$ modulo right multiplication by elements of $\Poly(\Z \to \Gamma)$, giving the claim.

Since the quantity $F(\tilde g(n) \gamma(n) \Gamma)$ is unaffected if one multiplies $\gamma$ on the right by an element of $\Poly(\Z \to \Gamma)$, we see that we may restrict $\gamma$ without loss of generality to a set of coset representatives of the finite quotient space $\Poly(q\Z \to \Gamma) / \Poly(\Z \to \Gamma)$.  Thus, by the triangle inequality, it suffices to prove the claim for a single fixed choice of $\gamma$.  

Fix $\gamma$.  As $\Gamma$ has finite index in $\Gamma'$, there is a finite index subgroup $\Gamma_*$ of $\Gamma$ which is normal in $\Gamma'$ (for instance, one can take $\Gamma_*$ to be the kernel of the left-action of $\Gamma$ on the finite space $\Gamma/\Gamma'$). 

The sequence $n \mapsto \gamma(n) \Gamma_*$ is then a polynomial map from $\Z$ to the finite group $\Gamma'/\Gamma_*$ (it is the composition of $\gamma \in \Poly(\Z \to \Gamma')$ with the quotient homomorphism $\pi$ from $\Gamma'$ to $\Gamma'/\Gamma_*$, where we equip $\Gamma'/\Gamma_*$ with the filtration $\pi(\Gamma'_i)$) and is hence periodic of some fixed period $Q$; this implies that $n \mapsto \gamma(n) \Gamma$ depends only on the residue class $n \mod Q$.  By the triangle inequality, it now suffices to show that
\begin{equation}\label{hax}
\int_X^{2X} \sup_{\tilde g \in \Poly( \R \to \tilde G )} \left| \sum_{n \in [x,x+H]} f(n) 1_{n=a \mod Q} \overline{F}(\tilde g(n) \gamma_0 \Gamma)\right| \ dx = o(HX)
\end{equation}
for any fixed $a$ and any fixed $\gamma_0 \in \Gamma'$.  

Since $\tilde G \cap \Gamma$ is cocompact in $\tilde G$, so is $\tilde G \cap \Gamma_*$.  As $\Gamma_*$ is normalized by $\gamma_0$, this implies that $\gamma_0^{-1} \tilde G \gamma_0 \cap \Gamma_*$ is cocompact in $\gamma_0^{-1} \tilde G \gamma_0$, so in particular the group $\gamma_0^{-1} \tilde G \gamma_0$ is rational.  If we let $\tilde F: \gamma_0^{-1} \tilde G \gamma_0 / (\gamma_0^{-1} \tilde G \gamma_0 \cap \Gamma_*) \to \C$ be the function 
$$\tilde F( \gamma_0^{-1} \tilde g \gamma_0 \Gamma_* ) \coloneqq F( \tilde g \gamma_0 \Gamma )$$
then $\tilde F$ is Lipschitz, and the left-hand side of \eqref{hax} can be rewritten (after conjugating $\tilde g$ by $\gamma_0$) as
$$\int_X^{2X} \sup_{\tilde g \in \Poly( \R \to \gamma_0^{-1} \tilde G \gamma_0 )} \left| \sum_{n \in [x,x+H]} f(n) 1_{n=a \mod Q} \overline{\tilde F}(\tilde g(n) \Gamma^*)\right| \ dx.$$
Here of course we give $\gamma_0^{-1} \tilde G \gamma_0$ the filtration $(\gamma_0^{-1} \tilde G \gamma_0)_i = \gamma_0^{-1} \tilde G_i \gamma_0$, and note that composition with the Lie group isomorphism $g \mapsto \gamma_0^{-1} g \gamma_0$ gives an isomorphism between $\Poly(\R \to \tilde G)$ and $\Poly( \R \to \gamma_0^{-1} \tilde G \gamma_0 )$.  Since the dimension of the nilmanifold $\gamma_0^{-1} \tilde G \gamma_0 / (\gamma_0^{-1} \tilde G \gamma_0 \cap \Gamma_*)$ is strictly less than that of $G/\Gamma$, the claim now follows from \eqref{mul}.
\end{proof}

We now eliminate some components of $F$ that arise from lower dimensional nilmanifolds\footnote{The need for dealing with these arises from the large sieve for nilsequences that we present as Proposition~\ref{lsieve-nil}.}.  Suppose that there is a non-trivial normal rational connected closed subgroup $N$ of $G$.  Then inside the Hilbert space $L^2(G/\Gamma)$ of square-integrable functions on $G/\Gamma$ (with respect to the Haar probability measure $\mu_{G/\Gamma}$) there is the closed subspace $L^2(G/\Gamma)^N$ of functions that are invariant with respect to the left-action of $N$; from normality this space is also preserved by the left-action of $G$.  

\begin{proposition}[Invariant case]\label{inv}  Assume that $f$ satisfies \eqref{dq}. If $N$ is a fixed non-trivial normal connected rational subgroup of $G$, and $F_N \in L^2(G/\Gamma)^N$ is a fixed Lipschitz continuous function, then
$$
\int_X^{2X} \sup_{g \in \Poly( \R \to G )} \left| \sum_{n \in [x,x+H]} f(n) \overline{F_N}(g(n) \Gamma)\right| \ dx  = o(HX).
$$
\end{proposition}

\begin{proof}  Let $\pi \colon G \to G/N$ be the quotient map from $G$ to $G/N$.  As $N$ is normal, closed, and connected, $G/N$ is also a nilpotent connected, simply connected\footnote{Indeed, from the Baker--Campbell--Hausdorff formula the space $G/N$ is homeomorphic to the vector space $\log G/\log N$.} Lie group, with a degree $k$ filtration $(G/N)_j \coloneqq \pi( G_j)$.  Because $\Gamma$ is discrete and cocompact in $G$ and $N \cap \Gamma$ is discrete and cocompact in $N$, we see that $\pi(\Gamma) \equiv \Gamma / (N \cap \Gamma)$ is discrete and cocompact in $\pi(G) = G/N$.  Thus $\pi(G)/\pi(\Gamma)$ is a degree $k$ filtered nilmanifold, whose dimension $\mathrm{dim}(G)-\mathrm{dim}(N)$ is strictly less than that of $G/\Gamma$. Then we can write $F_N = \tilde F \circ \tilde \pi$ for some $\tilde F \colon \pi(G)/\pi(\Gamma) \to \C$ with $\tilde \pi \colon G/\Gamma \to \pi(G)/\pi(\Gamma)$ is the obvious projection; this function $F_N$ can be seen to also be Lipschitz continuous by working in local coordinates.  Since $\pi \circ g \in \Poly(\R \to \pi(G))$ whenever $g \in \Poly(\R \to G)$, the claim now follows from \eqref{mul}.
\end{proof}

We let $F \mapsto {\mathbf E}(F|N)$ denote the orthogonal projection from $L^2(G/\Gamma)$ to $L^2(G/\Gamma)^N$; it can be described explicitly as
$$ {\mathbf E}(F|N)(g\Gamma) = \int_{N/(N \cap \Gamma)} F(gx)\ d\mu_{N/(N \cap \Gamma)}(x)$$
for almost every $g \in G$, where we view $N/(N \cap \Gamma)$ as a subset of $G/\Gamma$ in the natural fashion.  One can check (using the normality of $N$ and the uniqueness of the Haar probability measure $\mu_{N/(N \cap \Gamma)}$) that this gives a well-defined self-adjoint projection from $L^2(G/\Gamma)$ to $L^2(G/\Gamma)^N$, and so must indeed agree with the orthogonal projection to the latter space.  It is also clear from this definition that if $F$ is Lipschitz continuous then so is ${\mathbf E}(F|N)$.  In particular, from Proposition~\ref{inv} one can remove the component ${\mathbf E}(F|N)$ from $F$ while making a negligible impact to \eqref{xi}.  In our arguments we would like to perform this maneuver not for a single $N$, but for a large (but fixed) finite collection of such $N$.  To do this we need the following observation:

\begin{lemma}[Composition of projections]  Let $N_1, N_2$ be two normal connected rational subgroups of $G$.  Then $N_1N_2$ is also a normal connected rational subgroup, and
$$ {\mathbf E}( {\mathbf E}(F|N_1) | N_2 ) = {\mathbf E}(F|N_1 N_2)$$
for all $F \in L^2(G/\Gamma)$.  In particular (since $N_1 N_2 = N_2 N_1$), the projections $F \mapsto {\mathbf E}(F|N_1)$ and $F \mapsto {\mathbf E}(F|N_2)$ commute with each other.
\end{lemma}

\begin{proof} It is clear that $N_1 N_2$ is a normal connected subgroup of $G$.  Because $N_1 \cap \Gamma$ is cocompact in $N_1$ and $N_2 \cap \Gamma$ is cocompact in $N_2$, and $N_1$ is normal, $(N_1 \cap \Gamma)(N_2 \cap \Gamma)$ is cocompact\footnote{Indeed, we have $N_1 = K_1 (N_1 \cap \Gamma)$ and $N_2 = K_2 (N_2 \cap \Gamma)$ for some compact $K_1,K_2$, hence $N_1 N_2 = N_1 K_2 (N_2 \cap \Gamma) = K_2 N_1 (N_2 \cap \Gamma) = K_2 K_1 (N_1 \cap \Gamma) (N_2 \cap \Gamma)$, giving the cocompactness.} in $N_1 N_2$, so $N_1 N_2$ is rational.
The function
$$ F - {\mathbf E}( {\mathbf E}(F|N_1) | N_2 ) =(F - {\mathbf E}(F|N_1) ) + ({\mathbf E}(F|N_1) - {\mathbf E}( {\mathbf E}(F|N_1) | N_2 ) )$$
is orthogonal to $L^2(G/\Gamma)^{N_1} \cap L^2(G/\Gamma)^{N_2} = L^2(G/\Gamma)^{N_1 N_2}$.  The function
$$ {\mathbf E}( {\mathbf E}(F|N_1) | N_2 ) $$
is clearly $N_2$-invariant, and can also be seen to be $N_1$-invariant using the normality of $N_2$.  Thus ${\mathbf E}( {\mathbf E}(F|N_1) | N_2 ) $ lies in $L^2(G/\Gamma)^{N_1 N_2}$, and is thus the orthogonal projection of $F$ to this space.  The claim follows.
\end{proof}

Given any fixed finite collection $N_1,\dots,N_{\ell}$ of non-trivial normal connected rational subgroups $N_1,\dots,N_{\ell}$ of $G$, let $\Pi_{N_j} \colon L^2(G/\Gamma) \to (L^2(G/\Gamma)^{N_j})^\perp$ denote the complementary orthogonal projection to $L^2(G/\Gamma)^{N_j}$, thus
$$ \Pi_{N_j} F \coloneqq F - {\mathbf E}(F|N_j).$$
From the above lemma, the $\Pi_{N_j}$ all commute with each other.  Let $\Pi_{N_1,\dots,N_{\ell}} \coloneqq \Pi_{N_1} \dots \Pi_{N_{\ell}}$ denote the composition of these projections.  Then one can express $F - \Pi_{N_1,\dots,N_{\ell}} F$ as a finite sum of Lipschitz functions, each of which lies in one of the $L^2(G/\Gamma)^{N_j}$.  From Proposition~\ref{inv} and the triangle inequality, we thus have
\begin{equation}\label{leo}
\int_X^{2X} \sup_{g \in \Poly( \R \to G )} \left| \sum_{n \in [x,x+H]} f(n) \overline{(F-\Pi_{N_1,\dots,N_{\ell}}F)}(g(n) \Gamma)\right| \ dx = o(HX)
\end{equation}
as $\n \to \infty$.

We can also use Theorem~\ref{mult-poly}, proven in the previous section, to obtain

\begin{proposition}\label{nonab} Let the hypotheses be as in Theorem~\ref{mult-nil}, but assume that  $f$ satisfies \eqref{dq}. Then $G$ is not abelian.
\end{proposition}

\begin{proof}  Suppose for contradiction that $G$ was abelian, then $G/\Gamma$ is a connected abelian Lie group and is therefore a torus (this follows for instance from Pontryagin duality).  One can approximate $F$ uniformly by finite linear combinations of characters $e(\xi)$, where $\xi \colon G/\Gamma \to \R/\Z$ are continuous homomorphisms. By the triangle inequality (and passing to a subsequence of $X$ if necessary), we may thus find $\xi$ such that
$$ \int_X^{2X} \sup_{g \in \Poly( \R \to G )} \left| \sum_{n \in [x,x+H]} f(n) e(-\xi(g(n) \Gamma))\right| \ dx \gg H X.$$
But from Taylor expansion we see that $t \mapsto \xi(g(t) \Gamma)$ is of the form $t \mapsto P(t) \mod \Z$ for some $P \in \Poly_{\leq k}(\R \to \R)$, and Theorem~\ref{mult-poly} supplies the required contradiction.
\end{proof}

\subsection{Studying the structure of local nilsequences}
Now we start following the arguments of the previous section.  Define a \emph{local nilsequence} to be a pair $\phi = (I,g)$, where $I$ is an interval and $g \in \Poly(\R \to G)$.  We let $\Psi$ be the collection of all local nilsequences $\phi = (I,g)$, and $\Psi_I$ to be the collection of local nilsequences $(I,g)$ with a fixed choice of $I$.  One should view $(I,g)$ as an abstraction of the function $t \mapsto F(g(t) \Gamma)$ on $I$.  For any $\phi = (I,g) \in \Psi$ and $f \colon \R \to \C$, we define the correlation
$$ \langle f, \phi \rangle \coloneqq \frac{1}{|I|} \sum_{n \in I} f(n) \overline{F}(g(n) \Gamma),$$
where $F:G/\Gamma\to \mathbb{C}$ is understood to be a fixed Lipschitz function, with $G/\Gamma$ a fixed filtered nilmanifold.
As before we have the dilation action
$$ \lambda_{*} (I,g) \coloneqq \left(\lambda I, g\left(\frac{1}{\lambda} \cdot\right) \right)$$
for any $(I,g) \in \Psi$ and $\lambda >0$.  The family $\Psi$ will play the role of the family $\Phi$ from the preceding section (which can be viewed as the special case when $G/\Gamma = \R/\Z$ with the filtration $G_j = \R$ for $j \leq k$ and $G_j = \{0\}$ for $j>k$, and $F(x) \coloneqq e(x)$).  From \eqref{xi} we have
$$
 \int_X^{2X} \sup_{\phi \in \Psi_{[x,x+H]}} \left| \langle f, \phi \rangle\right| \ dx \gg X
$$
and hence by repeating the proof of~\cite[Lemma 2.1]{mrt-fourier} as in the previous section, we can find a large $(X,H)$-family of intervals ${\mathcal I}$, such that for each $I \in {\mathcal I}$ one can find $\phi_I \in \Psi_I$ such that $|\langle f, \phi_I \rangle| \gg 1$.

For subsequent analysis we will need to somehow import the decay estimates in Proposition~\ref{major} and \eqref{leo} into this context.  This is achieved via the following application of Markov's inequality.  Call a $(X,H)$-family of intervals \emph{small} if it has cardinality $o(X/H)$.

\begin{proposition}[Local decay outside of exceptional set]\label{loc-decay}   Assume that $f$ satisfies \eqref{dq}. Let $1 \leq P \leq X^{2\eps}$, and let ${\mathcal I}'$ be a $(X/P,H/P)$-family of intervals.  Then there exists a small exceptional subset ${\mathcal E}$ of ${\mathcal I}'$ such that the following properties hold uniformly for all $I \in {\mathcal I}' \backslash {\mathcal E}$:
\begin{itemize}
\item[(i)]  (Major arc estimate) If $\tilde G$ is a fixed connected closed proper rational subgroup of $G$, $E$ is a fixed compact subset of $\tilde G$, and $q$ is a fixed natural number, then
$$
\sup_{\substack{\eps \in E \\ \tilde g \in \Poly( \R \to \tilde G ) \\ \gamma \in \Poly(q\Z \to \Gamma)}} \sup_{I' \subset 500 I} \left| \sum_{n \in I'} f(n) \overline{F}(\eps \tilde g(n) \gamma(n) \Gamma)\right| \ dx = o(H/P)$$
where $I'$ ranges over all intervals contained in $500I$.
\item[(ii)]  (Invariant estimate)  For any fixed finite collection $N_1,\dots,N_{\ell}$ of non-trivial normal connected rational subgroups $N_1,\dots,N_{\ell}$ of $G$, one has
$$ \sup_{g \in \Poly( \R \to G )} \sup_{I' \subset 500 I} \left| \sum_{n \in [x,x+H]} f(n) \overline{(F-\Pi_{N_1,\dots,N_{\ell}}F)}(g(n) \Gamma)\right| = o(H/P).$$
\end{itemize}
\end{proposition}

\begin{proof}  We begin with (i).  We will shortly establish that
\begin{equation}\label{toast}
 \sum_{I \in {\mathcal I}'} \sup_{\substack{\eps \in B_{\tilde G}(1,r)\\ \tilde g \in \Poly( \R \to \tilde G )\\ \gamma \in \Poly(q\Z \to \Gamma)}} \sup_{I' \subset 500 I} \left| \sum_{n \in I'} f(n) \overline{F}(\eps \tilde g(n) \gamma(n) \Gamma)\right| = o(X/P)
\end{equation}
for each fixed $\tilde G, q,r$, where $B_{\tilde G}(1,r)$ denotes the ball of radius $r$ centred at the identity in $\tilde G$, and the decay rate in the $o(X)$ right-hand side may depend on $\tilde G,q$.  Assuming this bound for the moment, we can perform the following ``diagonalization'' argument.  There are only countably many rational subgroups $\tilde G$ of $G$ (because $\log \tilde G$ can be described as a subspace of $\log G$ cut out by equations with rational coefficients).  Enumerate the countable set of triples $(\tilde G, q,r)$ with $r$ a natural number as $(\tilde G_i,q_i,r_i)$.  For each $i$, we see from \eqref{toast}, the triangle inequality, and Markov's inequality that we can find an exceptional set ${\mathcal E}_i \subset {\mathcal I}$ of cardinality at most $\frac{1}{i} \frac{X}{H}$, and a threshold $\x_i$, such that
$$ 
\sum_{j \leq i} \sup_{\substack{\eps \in B_{\tilde G}(1,r_j)\\ \tilde g \in \Poly( \R \to \tilde G_j )\\ \gamma \in \Poly(q_j\Z \to \Gamma)}} \sup_{I' \subset 500 I} \left| \sum_{n \in I'} f(n) \overline{F}(\varepsilon \tilde g(n) \gamma(n) \Gamma)\right| \leq \frac{1}{i} H/P
$$
whenever $x \geq \x_i$ and $I \in {\mathcal I} \backslash {\mathcal E}_i$.  By increasing the $\x_i$ as necessary we may assume that $\x_{i+1}> \x_i$ for all $i$.  If we now set ${\mathcal E} \coloneqq {\mathcal E}_{i_*}$, where $i_*$ is the largest natural number for which $x \geq \x_{i_*}$, then ${\mathcal E}$ is well-defined for sufficiently large $x$, and the claim (i) follows (since any compact set $E$ is a subset of some ball $B_{\tilde{G}}(1,r)$). 

It remains to verify \eqref{toast}.  Set $H^* \coloneqq (X/P)^{\theta/2}$.  Then we can use the triangle inequality to write
$$
\sum_{n \in I'} f(n) \overline{F}(\varepsilon \tilde g(n) \gamma(n) \Gamma) \ll \frac{1}{H^*} \int_{I'} \left| \sum_{n \in [x,x+H^*]} f(n) \overline{F}(\varepsilon \tilde g(n) \gamma(n) \Gamma)\right|\ dx + O( H^* )$$
and thus (since the intervals $500 I$ in ${\mathcal I}'$ have bounded overlap in $[X/2P,4X/P]$) we can bound the left-hand side of \eqref{toast} by
$$ \frac{1}{H^*} \int_{X/2P}^{4X/P}  \sup_{\substack{\varepsilon \in E \\\tilde g \in \Poly( \R \to \tilde G )\\ \gamma \in \Poly(q\Z \to \Gamma)}} \left| \sum_{n \in [x,x+H^*]} f(n) \overline{F}(\varepsilon \tilde g(n) \gamma(n) \Gamma)\right| \ dx + o(X/P)$$
and the claim now follows from Proposition~\ref{major} (which is also valid if one replaces $X$ by a quantity comparable to $X/P$).

The claim (ii) is proven similarly (using \eqref{leo} in place of Proposition~\ref{major}), noting that there are only countably many rational closed connected subgroups $N$ of $G$ (since such groups are determined by their intersection $N \cap \Gamma$ with $\Gamma$, which is a finitely generated subgroup of the countable group $\Gamma$), and hence only countably many finite tuples $(N_1,\dots,N_{\ell})$. 
\end{proof}

Thus, for instance, using this proposition (with $P=1$ and ${\mathcal I}'={\mathcal I}$), we could now delete a small set of intervals from ${\mathcal I}$ and assume without loss of generality that the conclusions of this proposition hold for all $I \in {\mathcal I}$. As it turns out, however, it will be more useful to apply this proposition to a different family ${\mathcal I}'$ of intervals than ${\mathcal I}$, as we shall shortly see.

As in the preceding section, the next step is to relate the various $\phi_I$ to each other.  We need a variant of Definition~\ref{poly-comp}.  If $I$ is an interval, we say that a polynomial map $\eps \in \Poly(\R \to G)$ is \emph{smooth} on $I$ if $\eps(t)=O(1)$ for all $t \in I$.  Taking logarithms and applying \eqref{ber-3} to the polynomial map $\log \eps \colon \R \to \log G$, this implies in particular that
$|\frac{d^j}{dt^j} \log \eps(t)| \ll |I|^{-j} \langle t \rangle_I^{O(1)}$ for all $j \geq 0$ and $t \in \R$. In particular $\eps$ is also smooth on any interval $I' \sim I$ that is comparable to $I$. Also observe from the Baker--Campbell--Hausdorff formula \eqref{poly} that if $\eps_1,\eps_2$ are both smooth on $I$, then so are $\eps_1^{-1}$ and $\eps_1 \eps_2$ (with slightly different implied constants).

\begin{definition}[Comparability of nilsequences]\label{nil-comp} Given two local nilsequences $\phi = (I,g), \phi' = (I',g') \in \Psi$ and a scaling factor $\delta>0$, we define the relation
$$ \phi \sim_{\delta} \phi'$$
to hold if $I \sim I'$, and we have the relation
$$ g(t) = \eps(t) g'(t) \gamma(t)$$
for all $t \in \R$, where $\eps, \gamma \in \Poly(\R \to G)$ are polynomials obeying the following axioms:
\begin{itemize}
\item[(i)]  ($\eps$ smooth) $\eps$ is smooth on $I$.
\item[(ii)]  ($\gamma$ is $\frac{1}{\delta}$-integral)  $\gamma \in \Poly(\delta \Z \to \Gamma)$.
\end{itemize}
\end{definition}

We have the following analogue of Proposition~\ref{basic}:

\begin{proposition}[Basic properties of $\sim_{\delta}$]\label{basic-nil}  Let $\delta > 0$, and let $\phi, \phi', \phi'' \in \Psi$.
\begin{itemize}
\item[(i)] (Equivalence relation)  We have $\phi \sim_{\delta} \phi$, and if $\phi \sim_{\delta} \phi'$ then $\phi' \sim_{\delta} \phi$.  Finally, if $\phi \sim_{\delta} \phi'$ and $\phi' \sim_{\delta} \phi''$ then $\phi \sim_{\delta} \phi''$, where we allow the implied constants in the latter relations to depend on the implied constants in the former relations.
\item[(ii)]  (Dilation invariance)  If  $\phi \sim_{\delta} \phi'$ and $\lambda > 0$, then $\lambda_{*} \phi \sim_{\lambda \delta} \lambda_{*} \phi'$.
\item[(iv)]  (Sparsification)  If $\phi \sim_{\delta} \phi$, then $\phi \sim_{l\delta} \phi$ for any natural number $l$.
\end{itemize}
\end{proposition}

\begin{proof} These are immediate from Definition~\ref{nil-comp}, together with the previous observation that a polynomial map that is smooth on $I$ is also smooth on $I'$ for any $I' \sim I$, and the observation that the product of two polynomial maps smooth on $I$ is also smooth on $I$.
\end{proof}

Now we have the analogue of Proposition~\ref{lsieve}:

\begin{proposition1}[Large sieve]\label{lsieve-nil}  Let $I$ be an interval of some length $|I| \geq 1$, and let $f \colon \Z \to \C$ be a function bounded in magnitude by $1$.  Suppose that for each $i=1,\dots,K$ there is an interval $I_i \sim I$ and a local nilsequence $\phi_i \in \Psi_{I_i}$ such that
\begin{equation}\label{eo}
 |\langle f, \phi_i \rangle| \gg 1.
\end{equation}
Then at least one of the following claims hold:
\begin{itemize}
\item[(i)] $K \ll 1$.
\item[(ii)]  There exist $1 \leq i < j \leq K$ such that
$ \phi_i \sim_{1} \phi_j.$
\item[(iii)] (Correlation with major arc nilsequence) There is a connected closed proper rational subgroup $\tilde G$ of $G$ (drawn from a fixed finite collection of such subgroups) and a natural number $q$ (drawn from a fixed finite collection of such numbers) and a compact subset $E$ of $\tilde G$ (again drawn from a fixed finite collection) such that
$$
\sup_{\substack{\eps \in E\\ \tilde g \in \Poly( \R \to \tilde G )\\ \gamma \in \Poly(q\Z \to \Gamma)}} \sup_{I' \subset 500 I} \left| \sum_{n \in I'} f(n) \overline{F}(\eps \tilde g(n) \gamma(n) \Gamma)\right| \ dx \gg |I|.$$
\item[(iv)]  (Correlation with invariant nilsequence)  There is a tuple $(N_1,\dots,N_{\ell})$ of non-trivial normal connected rational subgroups $N_1,\dots,N_{\ell}$ of $G$ (drawn from a fixed finite collection of such subgroups) such that
$$ \sup_{g \in \Poly( \R \to G )} \sup_{I' \subset 500 I} \left| \sum_{n \in I'} f(n) \overline{(F-\Pi_{N_1,\dots,N_{\ell}}F)}(g(n) \Gamma)\right| \gg |I|.$$
\end{itemize}
\end{proposition1}

As one might expect, we will be able to use Proposition~\ref{loc-decay} to eliminate the options (iii), (iv) from this proposition, after removing a small set of exceptional intervals.

\begin{proof}  We let $K_0$ be a sufficiently large fixed natural number (depending on $F,G/\Gamma$), to be chosen later, and write $\phi_j = (I_j, g_j)$.  We can assume that $K \geq K_0$, since otherwise we are in case (i).  We will initially just analyze the first $K_0$ local nilsequences $\phi_j$, and return to the remaining $\phi_j$ later.

Let $S \colon \R^+ \to \R^+$ be a sufficiently rapidly growing but fixed function depending on $F,G/\Gamma,K_0$ to be chosen later.    The tuple $\vec g \coloneqq (g_1,\dots,g_{K_0})$ can be viewed as a polynomial map in the product group $G^{K_0}$ (endowed with the obvious filtration $(G^{K_0})_j \coloneqq G_j^{K_0}$).  The subgroup $\Gamma^{K_0}$ is a discrete cocompact lattice in $G^{K_0}$.  We may thus apply the quantitative factorization theorem in~\cite[Theorem 1.19]{gt-nil}, using the function $S$ in place of the function $M \mapsto M^A$, to obtain a factorization
\begin{equation}\label{voc}
 \vec g = \vec \eps\, \vec g'\, \vec \gamma
\end{equation}
where $\vec \eps, \vec g', \vec \gamma \in \Poly(\R \to G^{K_0})$ obey the following properties for some quantity $1 \leq M \ll_{K_0,S} 1$:
\begin{itemize}
\item[(i)]  (Smoothness) One has $|\log \vec \eps(t)| \leq M$ for all $t \in I$ (and hence by \eqref{ber-2}, $|\frac{d^j}{dt^j} \log \vec \eps(t)| \ll M |I|^{-j}$ for all $t \in I$ and $j \geq 0$).
\item[(ii)]  (Equidistribution) $\vec g'$ takes values in some rational connected closed subgroup $\vec G'$ of $G^{K_0}$ which is $M$-rational (in the sense of~\cite[Definition 2.5]{gt-nil}, using some arbitrarily chosen Mal'cev basis on $G^{K_0}$), and is totally $1/S(M)$-equidistributed in the sense that
$$ \left| \frac{1}{\# P} \sum_{n \in P} \vec F(\vec g'(n) \vec \Gamma') - \int_{\vec G'/\vec \Gamma'} \vec F\ d\mu_{\vec G'/\vec \Gamma'} \right|
\leq \frac{1}{S(M)} \|\vec F\|_{\mathrm{Lip}}$$
for any Lipschitz function $\vec F \colon \vec G'/\vec \Gamma' \to \C$, and any arithmetic progression $P$ in $I \cap \Z$ of length at least $\frac{1}{S(M)} |I|$, where $\vec \Gamma' \coloneqq \vec G' \cap \Gamma^{K_0}$ (and we endow $\vec G'/\vec \Gamma'$ with the metric induced from $(G/\Gamma)^{K_0}$).  
\item[(iii)]  (Rationality) $\vec \gamma(\Z) \Gamma^{K_0}$ takes values in the set $\{ \vec \gamma \Gamma^{K_0}: \vec \gamma^q \in \Gamma^{K_0} \}$
for some $1 \leq q \leq M$, and the sequence $n \mapsto \vec \gamma(n) \Gamma^{K_0}$ is periodic on $\Z$ with period at most $M$.
\end{itemize}
Arguing as in the proof of~\cite[Corollary 1.20]{gt-nil}, this gives a summation formula of the form
\begin{equation}\label{geo}
 \sum_{n \in I'} \vec F( \vec g(n) \vec \Gamma ) = \sum_{i=1}^s A_i \int_{x_i \vec G' y_i \Gamma^{K_0} / \Gamma^{K_0}} \vec F\ d\mu_{x_i \vec G' y_i \Gamma^{K_0} / \Gamma^{K_0}} + O_{M,K_0}\left( \frac{\|\vec F\|_{\mathrm{Lip}}}{S(M)^{1/2}} |I| \right)
\end{equation}
for any interval $I' \subset I$, where the $A_i$ are positive quantities summing to $O(|I|)$, the $x_i$ are elements of $\vec G^{K_0}$ with $\log x_i = O_M(1)$, of magnitude $O_M(1)$, and the $y_i$ are elements of $\vec G^{K_0}$ with $y_i^q \in \Gamma$ (the argument proceeds by 
splitting $I'$ into $O_{M,K_0}(S(M)^{1/2})$ arithmetic progressions of diameter $O_{M,K_0}( \frac{|I|}{S(M)^{1/2}})$ and spacing equal to the period of $\vec \gamma \Gamma^{K_0}$).  One could be more precise about the values of $A_i,x_i,y_i$ here, as well as provide upper bounds on the quantity $s$ but it will not be necessary for our argument to do so.

We write $\vec \eps = (\eps_1,\dots,\eps_{K_0})$, $\vec g' = (g'_1,\dots,g'_{K_0})$, and $\vec \gamma = (\gamma_1,\dots,\gamma_{K_0})$.  We now divide into several cases, depending on the nature of $\vec G'$.  For each $1 \leq j \leq K_0$, let $\pi_j \colon G^{K_0} \to G$ be the projection to the $j^{\mathrm{th}}$ factor of $G$.  Then $\pi_j(\vec G')$ is a closed connected rational subgroup of $G$.  Suppose that there exists $j$ for which $\pi_j$ is not surjective, so that $\pi_j(\vec G')$ is a proper subgroup of $G$.  Because $\vec G'$ is $M$-rational, it belongs to a fixed finite family of subgroups of $G^{K_0}$, and hence $\pi_j(\vec G')$ also belongs to a fixed finite family of subgroups.  From \eqref{voc}, \eqref{eo} we have
$$
 \left| \sum_{n \in I_j} f(n) \overline{F}(\eps_j(n) g'_j(n) \gamma_j(n) \Gamma) \right| \gg |I|,
$$ 
so in particular $|I_j| \gg |I|$.  Let $\sigma>0$ be a small quantity to be chosen later. Then by covering $I_j$ by intervals $I'_j$ of length $\sigma |I|$ and using the pigeonhole principle, we can find one such interval $I'_j$ for which
$$
 \left| \sum_{n \in I'_j} f(n) \overline{F}(\eps_j(n) g'_j(n) \gamma_j(n) \Gamma) \right| \gg \sigma |I|.
$$ 
From property (i) and \eqref{ber} we see that
$$ F(\eps_j(n) g'_j(n) \gamma_j(n) \Gamma) = F(\eps_j(x_{I'_j}) g'_j(n) \gamma_j(n) \Gamma) + O_M(\sigma)$$
for $n \in I'_j$.  For $\sigma$ sufficiently small depending on $M$ (but with $\sigma \asymp_M 1$) we can then neglect the error term and conclude that
$$
 \left| \sum_{n \in I'_j} f(n) \overline{F}(\eps_j(x_{I'_j}) g'_j(n) \gamma_j(n) \Gamma) \right| \gg \sigma |I|,
$$ 
and now we have conclusion (iii) of the proposition.

Henceforth we now assume that $\pi_j$ is surjective for all $1 \leq j \leq K_0$.  For distinct $i,j \in \{1,\dots,K_0\}$, consider the group 
$$ N_{i,j} \coloneqq \{ \pi_i(\vec h): \vec h \in \vec G'; \pi_j(\vec h) = 1 \}.$$
This is a normal connected closed rational subgroup of $G$; indeed one can check that
$$ \log N_{i,j} = \{ \tilde \pi_i(\vec h): \vec h \in \log \vec G'; \tilde \pi_j(\vec h) = 0\}$$
where $\tilde \pi_i \colon \log G^{K_0} \to \log G$ are the coordinate projections, and then the claims are easily verified.  Let $\Pi$ be the projection on $L^2(G/\Gamma)$ formed by composing together the $\Pi_{N_{i,j}}$ for all distinct $i,j \in \{1,\dots,K_0\}$ for which $N_{i,j}$ is not trivial.  Note that because $\vec G'$ belongs to a fixed finite family of subgroups of $G^{K_0}$, $N_{i,j}$ belongs to a fixed finite family of subgroups of $G$ (depending on $M, K_0$).  Thus, if
$$ \left| \sum_{n \in I_i} f(n) \overline{(F-\Pi F)(g_i(n) \Gamma)} \right| \gg |I|$$
for some $i=1,\dots,K_0$, then we have conclusion (iv) of the proposition.  Otherwise, by \eqref{eo} and the triangle inequality, we may assume that
$$ \left| \sum_{n \in I_i} f(n) \overline{\Pi F(g_i(n) \Gamma)} \right| \gg |I|$$
for all $i=1,\dots,K_0$.  We may now apply Cauchy--Schwarz as in the proof of Proposition~\ref{lsieve} and conclude that 
\begin{equation}\label{neo}
 \sum_{i=1}^{K_0} \sum_{j=1}^{K_0} \left| \sum_{n \in I_i \cap I_j} \Pi F(g_i(n) \Gamma) \overline{\Pi F(g_j(n) \Gamma)}\right| \gg K_0^2 |I|.
\end{equation}
We now dispose of the diagonal terms by claiming that
\begin{equation}\label{nii}
 \sum_{n \in I_i} |\Pi F(g_i(n) \Gamma)|^2 \ll |I|
\end{equation}
for each $i$.  A key point here is that the implied constant does not depend on $K_0,M$.  Here we have a technical difficulty because $\Pi F$ is not well controlled in $L^\infty(G/\Gamma)$ norm (one has a $L^\infty$ bound of $O_{K_0,M}(1)$ rather than $O(1)$); however it is still bounded in $L^2(G/\Gamma)$ by $1$ since $\Pi$ is an orthogonal projection, and it also has a Lipschitz norm of $O_{K_0,M}(1)$.  Nevertheless, by applying the formula
\eqref{geo}, one can write the left-hand side of \eqref{nii} as 
$$\sum_{j=1}^s A_j \int_{x_j \vec G' y_j \Gamma^{K_0} / \Gamma^{K_0}} |\Pi F \circ \pi_i|^2\ d\mu_{x_j \vec G' y_j \Gamma^{K_0} / \Gamma^{K_0}} + O_{M,K_0}\left( \frac{1}{S(M)^{1/2}} |I| \right)$$
for some $A_j,x_j,y_j$ (which can depend on $i$) with the properties listed after \eqref{geo}.  As $\pi_i$ is surjective, it pushes forward Haar measure to Haar measure by the uniqueness properties of Haar measure, so the above estimate simplifies to
$$\sum_{j=1}^s A_j \int_{G / \Gamma} |\Pi F|^2\ d\mu_{G / \Gamma} + O_{M,K_0}\left( \frac{1}{S(M)^{1/2}} |I| \right).$$
Since the $L^2$ norm of $\Pi F$ is bounded by $1$, and $\sum_{j=1}^s A_j = O(|I|)$, we obtain the claim \eqref{nii} if $S$ is chosen to be sufficiently rapidly growing.

Using \eqref{nii} to remove the diagonal terms from \eqref{neo}, we conclude (for $K_0$ large enough) that there exist distinct $i,j \in \{1,\dots,K_0\}$ such that
$$ \left| \sum_{n \in I_i \cap I_j} \Pi F(g_i(n) \Gamma) \overline{\Pi F(g_j(n) \Gamma)}\right| \gg |I|.$$
Applying \eqref{geo}, we can bound the left-hand side by
$$ \sum_{l=1}^s A_l \int_{x_l \vec G' y_l \Gamma^{K_0} / \Gamma^{K_0}} (\Pi F \circ \pi_i) \overline{(\Pi F \circ \pi_j)}\ d\mu_{x_l \vec G' y_l \Gamma^{K_0} / \Gamma^{K_0}} + O_{M,K_0}\left( \frac{|I|}{S(M)} \right)$$
for some $A_l,x_l,y_l$ obeying the properties after \eqref{geo}; in particular, for $S$ sufficiently rapidly growing, there exists $l$ such that
$$ \int_{x_l \vec G' y_l \Gamma^{K_0} / \Gamma^{K_0}} (\Pi F \circ \pi_i) \overline{(\Pi F \circ \pi_j)}\ d\mu_{x_l \vec G' y_l \Gamma^{K_0} / \Gamma^{K_0}}  \neq 0.$$
We can project the nilmanifold $x_l \vec G' y_l \Gamma^{K_0} / \Gamma^{K_0}$ down to $(G/\Gamma)^2$ using the projection map $(\pi_i,\pi_j)$ to the $i,j$ coordinates.  The image of this nilmanifold is then invariant under the left action of the normal group $N_{i,j} \times \{1\}$.  If $N_{i,j}$ is non-trivial, then $\Pi F$ has mean zero along all orbits of $N_{i,j}$ by construction, and the above integral will vanish.  Thus $N_{i,j}$ must be trivial.  A similar argument shows that $N_{j,i}$ is trivial.

Now consider the subgroup
$$ G_{i,j} \coloneqq \{ (\pi_i(\vec g), \pi_j(\vec g)): \vec g \in \vec G' \}$$
of $G^2$; this is a closed connected rational subgroup of $G^2$.  By the preceding discussion, the projections $\pi_1 \colon G_{i,j} \to G$, $\pi_2 \colon G_{i,j} \to G$ are both surjective and injective.  By the Goursat lemma, $G_{i,j}$ then takes the form
\begin{equation}\label{gij-graph}
 G_{i,j} = \{ (g, \phi_{i,j}(g)): g \in G \}
\end{equation}
for some group isomorphism $\phi_{i,j} \colon G \to G$.  As there are $O_{K_0,M}(1) = O_{K_0,S}(1)$ possible choices for $\vec G'$, there are $O_{K_0,S}(1)$ choices of $\phi_{i,j}$.  As $G_{i,j}$ is rational, the map $\phi_{i,j}$ (when expressed in the standard basis for $\log G$) is a polynomial map with rational coefficients, hence (by Baker--Campbell--Hausdorff) $\phi_{i,j}(\Gamma)$ is covered by finitely many translates of $\Gamma$, and conversely; thus $\phi_{i,j}(\Gamma)$ must be commensurate with $\Gamma$, in the sense that $\phi_{i,j}(\Gamma) \cap \Gamma$ has finite index in $\phi_{i,j}(\Gamma)$ or $\Gamma$.  Since $\vec g'$ takes values in $\vec G'$, we see from \eqref{gij-graph} that
$$ g'_j = \phi_{i,j}(g'_i)$$
and thus by \eqref{voc}
\begin{equation}\label{gji}
 g_j = \eps_{i,j} \phi_{i,j}(g_i) \gamma_{i,j}
\end{equation}
where
$$ \eps_{i,j} \coloneqq \eps_j \phi_{ij}(\eps_i)^{-1}$$
and
$$ \gamma_{i,j} \coloneqq \phi_{i,j}(\gamma_i)^{-1} \gamma_j.$$
From the smoothness properties of $\eps_i,\eps_j$ we see that
$$ \log \eps_{i,j}(t) \ll_{K_0,M} 1 $$
for $t \in I$, and hence (since $M = O_{K_0,S}(1)$)
\begin{equation}\label{log}
 \log \eps_{i,j}(t) \ll_{K_0,S} 1.
\end{equation}
By rationality, the functions $\phi_{i,j}(\gamma_i) \Gamma$, $\gamma_j \Gamma$ each map $\Z$ to $\{ \gamma \Gamma: \gamma^q \in \Gamma \}$ for some $q = O_{K_0,M}(1) = O_{K_0,S}(1)$ and are also periodic with period $O_{K_0,S}(1)$, which (as discussed at the end of Appendix~\ref{bch}) implies that
\begin{equation}\label{gamp}
 \gamma_{i,j}(\Z) \Gamma \subset \{ \gamma \Gamma: \gamma^q \in \Gamma \}
\end{equation}
for some $q = O_{K_0,S}(1)$, and $\gamma_{i,j}$ is periodic with period $O_{K_0,S}(1)$.

Call a pair $(i,j)$ of distinct elements of $\{1,\dots,K\}$ \emph{good} if there is an identity of the form \eqref{gji}, where $\phi_{i,j}$ ranges over one of $O_{K_0,S}(1)$ isomorphisms of $G$, $\eps_{i,j}$ obeys \eqref{log} on $I$, and $\gamma_{i,j}$ obeys \eqref{gamp} for some $q = O_{K_0,S}(1)$ and is periodic with period $O_{K_0,S}(1)$.  By relabeling, we have shown that every $K_0$-element subset of $\{1,\dots,K\}$ contains a good pair $(i,j)$.  Averaging over all such subsets, we conclude that there are $\gg_{K_0} K^2$ good pairs.  In particular, by the pigeonhole principle, there exists $i \in \{1,\dots,K\}$ such that $(i,j)$ is good for $\gg_{K_0} K$ values of $j$.  Note that there are only $O_{K_0,S}(1)$ possible values of $\phi_{i,j}$ and of the coset $\gamma_{i,j} \Poly(\Z \to \Gamma)$.  Thus, if $K$ is large enough, we see from the pigeonhole principle that there exist distinct $j,j'$ such that $\phi_{i,j} = \phi_{i,j'}$ and $\gamma_{i,j} \Poly(\Z \to \Gamma) = \gamma_{i,j'} \Poly(\Z \to \Gamma)$.  From \eqref{gji} we conclude that
$$ g_j = \eps_{i,j} \eps_{i,j'}^{-1} g_{j'} \gamma_{i,j'}^{-1} \gamma_{i,j}$$
and hence by Definition~\ref{nil-comp}
$$ \phi_j \sim_{1} \phi_{j'}$$
and the claim follows.
\end{proof}

Using this proposition as in the previous section (but now also using Proposition~\ref{loc-decay} to eliminate the unwanted options (iii), (iv) from Proposition~\ref{lsieve-nil}), we obtain the following variant of Proposition~\ref{Scaling-down}.

\begin{proposition}[Scaling down]\label{Scaling-down-nil} Let $2 \leq P \leq Q \leq H \leq X$ and let $f \colon \N \to \C$ be a $1$-bounded completely multiplicative function.  Assume that $P, \frac{\log Q}{\log P}$ are sufficiently large (depending on the parameters $k,\theta,\eta$).  Suppose there exists a large $(X,H)$-family ${\mathcal I}$ and a local nilsequence $\phi_I \in \Psi_I$ associated to each interval $I \in {\mathcal I}$ such that
$$ |\langle f, \phi_I \rangle| \gg 1$$
holds for all $I \in {\mathcal I}$.  Then there exist $P' \in [P,Q/2]$, a large $(\frac{X}{P'}, \frac{H}{P'})$-family ${\mathcal I}'$, and a local nilsequence $\phi'_{I'} \in \Psi_{I'}$ associated to each $I' \in {\mathcal I}'$, such that
$$ |\langle f, \phi'_{I'} \rangle| \gg 1$$
for all $I' \in {\mathcal I}'$.  Furthermore, for each $I' \in {\mathcal I}'$, one can find $\gg \pi_0(P')$ pairs $(I, p')$, where $I \in {\mathcal I}$ and $p'$ is a prime in $[P', 2P']$, such that the rescaled interval $\frac{1}{p'} I$ lies within $3\frac{H}{P'}$ of $I'$, and such that
\begin{equation}\label{peo-nil}
(\frac{1}{p'})_{*} \phi_I \sim_{1} \phi'_{I'}.
\end{equation}
\end{proposition}

\begin{proof}  Repeat the proof of~\cite[Proposition 3.1]{mrt-fourier} down to the paragraph after (36).  Then one can find $P' \in [P,Q/2]$, and a collection ${\mathcal I}_2$ of intervals in $[0, 10 X/P']$ that are separated by distance at least $2H/P'$, with the property that for $\gg \frac{X}{H} \pi_0(P')$ pairs $(I,p')$ with $I \in {\mathcal I}$ and $p'$ a prime in $[P',2P']$, $\frac{1}{p'} I$ lies within $3\frac{H}{P'}$ of some interval $I' \in {\mathcal I}_2$, and furthermore
$$ |\langle f, (\frac{1}{p'})_{*} \phi_I \rangle| \gg 1.$$
Note that each $I'$ is associated to at most $O(\pi_0(P'))$ such pairs.  In particular we have the freedom to remove a small set of intervals from ${\mathcal I}'$ without significantly diminishing the set of pairs $(I,p')$ in the above claims.

From Proposition~\ref{lsieve-nil} and the greedy algorithm, we see that for each $I' \in {\mathcal I}_2$, at least one of the following claims hold:
\begin{itemize}
\item[(i)] There is a family $\phi_{I',1},\dots,\phi_{I',K_{I'}} \in \Psi$ of functions with $K_{I'} = O(1)$ such that whenever $(I,p')$ is one of the above pairs with $\frac{1}{p'} I$ within $3\frac{H}{P'}$ of $I'$, one has
$$ (\frac{1}{p'})_{*} \phi_I \sim_{1} \phi_{I',K_i}$$ 
for some $i=1,\dots,K_{I'}$.
\item[(ii)] There is a connected closed proper rational subgroup $\tilde G$ of $G$ (drawn from a fixed finite collection of such subgroups) and a natural number $q$ (drawn from a fixed finite collection of such numbers) and a compact subset $E$ of $\tilde G$ (again drawn from a fixed finite collection) such that
$$
\sup_{\substack{\eps \in E \\ \tilde g \in \Poly( \R \to \tilde G ) \\ \gamma \in \Poly(q\Z \to \Gamma)}} \sup_{J \subset 500 I'} \left| \sum_{n \in J} f(n) \overline{F}(\eps \tilde g(n) \gamma(n) \Gamma)\right| \ dx \gg \frac{H}{P'}.$$
\item[(iii)]  There is a tuple $(N_1,\dots,N_{\ell})$ of non-trivial normal connected rational subgroups $N_1,\dots,N_{\ell}$ of $G$ (drawn from a fixed finite collection of such subgroups) such that
$$ \sup_{g \in \Poly( \R \to G )} \sup_{J \subset 500 I'} \left| \sum_{n \in J} f(n) \overline{(F-\Pi_{N_1,\dots,N_{\ell}}F)}(g(n) \Gamma)\right| \gg \frac{H}{P'}.$$
\end{itemize}
By Proposition~\ref{loc-decay}, we can eliminate the options (ii), (iii) by removing a small set of intervals from ${\mathcal I}_2$, leaving only option (i).  One can now continue the proof of~\cite[Proposition 3.1]{mrt-fourier} (making only the obvious changes) to conclude the proposition.
\end{proof}

We continue to follow the line of argument from the previous section.  We will need an analogue of Lemma~\ref{bezout} for nilsequences:

\begin{lemma}[Bezout identity]\label{bezout-nil}  Let $a,b$ be coprime natural numbers, and let $\lambda>0$.  Then
$$ \Poly(\frac{\lambda}{a} \Z \to \Gamma) \cdot \Poly(\frac{\lambda}{b} \Z \to \Gamma) = \Poly(\lambda\Z \to \Gamma)$$
and
$$ \Poly(\frac{\lambda}{a} \Z \to \Gamma) \cap \Poly(\frac{\lambda}{b} \Z \to \Gamma) = \Poly(\frac{\lambda}{ab} \Z \to \Gamma).$$
\end{lemma}

\begin{proof}  
See Appendix~\ref{app:b}.
\end{proof}

As a consequence, we can now establish the analogue of Proposition~\ref{chinese} for nilsequences (though with a slightly weaker version of part (ii)):

\begin{proposition}[Chinese remainder theorem]\label{chinese-nil} Let $I$ be an interval of some length $|I| \geq 1$, and let ${\mathcal P}$ be a finite collection of primes.
\begin{itemize}
\item[(i)]  Suppose that $\phi \in \Psi_I$, and for each $p \in {\mathcal P}$ there exists $\phi_p \in \Psi$ such that
$$ \phi_p \sim_{1} \phi.$$
Then there exists $\phi' \in \Psi_I$ such that
$$ \phi_p \sim_{\frac{1}{p}} \phi'$$
for all $p \in {\mathcal P}$, and furthermore $\langle f, \phi \rangle = \langle f, \phi' \rangle$ for all $f \colon \Z \to \C$.
\item[(ii)]  Suppose that $\phi \in \Psi_I$ and $\phi' \in \Psi$ are such that
$$ \phi \sim_{\frac{1}{p}} \phi'$$
for all $p \in {\mathcal P}$, and suppose $|I|$ is sufficiently large (depending on the implied constants in the $\sim_{\frac{1}{p}}$ notation). Then there is a subset ${\mathcal P}'$ of ${\mathcal P}$ with $\# {\mathcal P}' \gg \# {\mathcal P}$ such that
$$ \phi \sim_{\frac{1}{\prod {\mathcal P}'}} \phi'.$$
\end{itemize}
\end{proposition}

\begin{proof}  See Appendix~\ref{app:b}
\end{proof}

One can now conclude an analog of Proposition~\ref{prop32-analog}:

\begin{proposition}[Building a family of related local nilsequences]\label{prop32-analog-nil}  Let the hypotheses be as in Theorem~\ref{mult-nil}.  Let $\eps > 0$ be sufficiently small depending on $k,\theta,\eta$, and suppose that $X$ is sufficiently large depending on $\theta,\eta,\eps,k$.  Then there exist $P', P'' \in [X^{\eps^2/2},X^\eps]$, a large $(\frac{X}{P'P''}, \frac{H}{P'P''})$-family ${\mathcal I}''$, and a local nilsequence $\phi''_{I''} \in \Psi_{I''}$ for each $I'' \in {\mathcal I}''$ such that \eqref{falp} holds for all $I'' \in {\mathcal I}''$; also, each $I'' \in {\mathcal I}''$ obeys the conclusions (i), (ii) of Proposition~\ref{loc-decay} (with $P=P'P''$).  Furthermore, there exist a collection ${\mathcal Q}$ of $\gg \pi_0(P')^2 \frac{X}{H}$ quadruples $(I''_1,I''_2, p'_1,p'_2)$ with $I''_1, I''_2$ distinct intervals in ${\mathcal I}''$ and $p'_1,p'_2$ distinct primes in $[P',2P']$, such that $I''_1$ lies within $50 \frac{H}{P'P''}$ of $\frac{p'_2}{p'_1} I''_2$ (so in particular $I''_1 \sim \frac{p'_2}{p'_1} I''_2$), and such that \eqref{phip} holds for a large set of primes $p''$ in $[P''/2, P'']$.
\end{proposition}

\begin{proof}  One repeats the proof of Proposition~\ref{prop32-analog} verbatim, using Propositions~\ref{basic-nil}, \ref{Scaling-down-nil}, \ref{chinese-nil} in place of Propositions~\ref{basic}, \ref{Scaling-down}, \ref{chinese}.  To ensure the conclusions (i), (ii) of Proposition~\ref{loc-decay}, one simply removes the exceptional set produced by that proposition, which has only a negligible impact on the cardinality of ${\mathcal Q}$.
\end{proof}

For the rest of this section we introduce the quantities
$$ N \coloneqq \# {\mathcal I}'' \asymp \frac{X}{H}$$
and
$$ d \coloneqq \pi_0(P')^2$$
as in the previous section.  We now establish an analog of Proposition~\ref{prop41-analog}:

\begin{proposition}[Local structure of $\phi''$]\label{prop41-analog-nil}  Let the hypotheses be as in Theorem~\ref{mult-nil}, and let $\eps,X,P',P'',{\mathcal I}'', \phi''_{I''}$ be as in Proposition~\ref{prop32-analog-nil}.   
Let $\ell_1, \ell_2$ be bounded even integers obeying \eqref{dop}.  We allow implied constants to depend on $\eps,\ell_1, \ell_2$.  Then, for a subset ${\mathcal Q}'$ of the quadruples $e = (I''_1, I''_2, p'_1, p'_2)$ in ${\mathcal Q}$ of cardinality $\gg dN$, one can find a collection ${\mathcal A}_e$ of quadruples $\vec a = (a_1,a_2,b_1,b_2)$ of natural numbers of cardinality $\asymp d^{\ell_1+\ell_2} / N^2$, and a large collection ${\mathcal P}_{e,\vec a}$ of primes in $[P''/2,P'']$ associated to each $\vec a \in {\mathcal A}_e$, obeying the properties (i), (ii), (iii) of Proposition~\ref{prop41-analog}.  In particular, the implied constants in \eqref{phip-big} do not depend on $\ell_1,\ell_2$, and the implied constants in \eqref{adi} may depend on $\ell_i$ but do not depend on $\ell_{3-i}$.
\end{proposition}

\begin{proof} One repeats the proof of Proposition~\ref{prop41-analog}, using Propositions~\ref{basic-nil}, \ref{prop32-analog-nil}, \ref{chinese-nil} in place of Propositions~\ref{basic}, \ref{prop32-analog}, \ref{chinese}.  Note that Proposition~\ref{chinese-nil}(ii) will force us to refine the set of primes ${\mathcal P}_{I''_1,I''_2}$ somewhat, but it will still remain large. 
\end{proof}

In the previous section, the values of $\ell_1,\ell_2$ were not of particular significance.  In this section it will be convenient to choose $\ell_1$ to be significantly larger than $\ell_2$, because we will need to work with many quadruples simultaneously.

\subsection{Solving the approximate dilation invariance}
The next step is to solve the approximate dilation invariance equation \eqref{adi} for a given quadruple $e$.  In the previous section, we were able to obtain a satisfactory description of the solutions just by using a single choice of $\vec a=(a_1,b_1,a_2,b_2) \in {\mathcal A}_e$; see Proposition~\ref{sold}.  Here, however, the situation will be more complicated, because for each $\vec a$ there can be some unwanted ``exotic'' solutions $\phi''_{I''_i}$ to \eqref{adi} that do not pretend to behave like a character $t^{iT}$, and which therefore cannot be treated using the results from~\cite{mr}, \cite{MRT}.  For instance, consider the situation in which
\begin{equation}\label{far}
 \phi''_{I''_1} = ( I''_1, t \mapsto \gamma^{P(t)} )
\end{equation}
for some $\gamma = \gamma_{a_1,b_1} \in \Gamma$ of polynomial size $\gamma = X^{O(1)}$ and some polynomial $P(t)$ which is a partial Taylor expansion of the analytic function $t \mapsto \frac{\log t}{\log(a_1/b_1)}$ around the midpoint $x_{I''_1}$ of $I''_1$.  If the filtration $G_i$ is defined suitably, $t \mapsto \gamma^{P(t)}$ will be a polynomial map.  On the other hand, since
$$ \gamma^{\frac{\log(a_1 t)}{\log(a_1/b_1)}} = \gamma^{\frac{\log(b_1 t)}{\log(a_1/b_1)}} \gamma$$
one can verify that the approximate dilation invariance \eqref{adi} will be obeyed for $i=1$ if $P(t)$ is a sufficiently long partial Taylor expansion of
$t \mapsto \frac{\log t}{\log(a_1/b_1)}$.  If $\gamma$ is a central element of $G$, the local nilsequence \eqref{far} will then ``pretend'' to be like $t^{iT}$ for some $T$ depending on $\gamma$ and $\log(a_1/b_1)$, but if $\gamma$ is not central then one would not expect this to be the case in general.  As a consequence, merely having \eqref{adi} for a single tuple $\vec a$ will be insufficient for our arguments.  However, as we shall see, if have the approximate dilation invariance \eqref{adi} holds for a very ``dense'' collection of ratios $a_1/b_1$, then one cannot have a representation such as \eqref{far} for all of these $a_1/b_1$ simultaneously unless the bases $\gamma$ involved are essentially central, or if $\phi''_{I''_1}$ can be modeled by a lower dimensional nilsequence.  Actually the first case is contained in the second thanks to Proposition~\ref{nonab}, so we will be able to proceed via Proposition~\ref{major}.

We now begin the formal arguments.  The first step is to decouple the ``continuous'' (or ``Archimedean'') aspects of the equation \eqref{adi} (associated to the smooth polynomial maps $\eps$ in Definition~\ref{nil-comp} and the dilation structure in \eqref{adi}) from the ``rational'' (or ``non-Archimedean'') aspects (associated to the rational maps $\gamma$ in Definition~\ref{nil-comp}).  It will be possible to do this thanks to the exponentially large size of the modulus $\prod {\mathcal P}_{e,\vec a}$ occurring in \eqref{adi}, which enable a sort of ``Lefschetz principle'' to pass to the continuous setting.  To describe this more precisely we need some more notation.  As in the previous section, a quantity $a$ (which could be a number or an element of $G$ or $\log G$) is said to be of \emph{polynomial size} if $a = O(X^{O(1)})$.  We similarly say that a map $g \in \Poly(\R \to G)$ is of \emph{polynomial size} if the coefficients $g_0,\dots,g_k$ of the Taylor expansion
$$ g(t) = g_0 g_1^{\binom{t}{1}} \dots g_k^{\binom{t}{k}}$$
of $g$ around the origin are all of polynomial size.  Observe from many applications of the Baker--Campbell--Hausdorff formula (Appendix~\ref{bch}) that a polynomial map $g \in \Poly(\R \to G)$ is of polynomial size if and only if the polynomial map $\log g \colon \R \to \log G$ has all coefficients of its Taylor expansion around the origin of polynomial size.  In particular (from a further application of Baker--Campbell--Hausdorff) if $g,h \in \Poly(\R \to G)$ are of polynomial size then so are $g^{-1}$ and $gh$ (though with different implied constants in the $O()$ notation); also one has $g(t)$ of polynomial size whenever $g,t$ are.  Next, for any modulus $Q > 0$, we say that a map $\gamma \in \Poly(\R \to G)$ is \emph{$Q$-rational} if $\gamma \in \Poly(\frac{q}{Q} \Z \to \Gamma)$ for some natural number $q$ of polynomial size.  From Lemma~\ref{bezout-nil} (and a rescaling by $q$) we see that if $\gamma, \gamma' \in \Poly(\R \to G)$ are $Q$-rational, then so are $\gamma^{-1}$ and $\gamma \gamma'$, again with different implied constants in the $O()$ notation.  The key fact that allows us to decouple is the following ``transversality'' between the collection of maps of polynomial size and the collection of maps that are extremely rational.

\begin{lemma}[Transversality]\label{disj}  Let ${\mathcal P}$ be a large set of primes in $[P''/2,P'']$.  Suppose that $g \in \Poly(\R \to G)$ is both of polynomial size and $Q$-rational, where $Q=\prod \mathcal{P}$.  Then $g$ is equal to a constant $g(t) = \gamma$ for some $\gamma \in \Gamma$ of polynomial size.
\end{lemma}

\begin{proof}  The group element $g(0)$ lies in $\Gamma$ and is of polynomial size.  By dividing this out we may assume $g(0)=1$. We first prove the claim for abelian groups $G$. Since $g$ in a map in $ \Poly(\frac{q}{Q}\mathbb Z \to \Gamma)$ we have 
\[
g(\frac{q}{Q}t)  =\sum_{i=0}^k a_i \binom{t}{i}
\]
with $a_i \in \Z$.  So that 
\[
g(t)  =\sum_{i=0}^k a_i \binom{\frac{Q}{q}t}{i}.
\]
Since $g$ is polynomial size we conclude that $[\frac{Q}{q}]^k\frac{1}{k!}a_k$ must be of polynomial size; as $q$ is also of polynomial size, we therefore have
$$ a_{k} = O( X^{O(1)} Q^{-k} ).$$
On the other hand, as ${\mathcal P}$ is a large set of primes in $[P''/2,P'']$, we have from \eqref{exp-lower} that
$$ Q \gg \exp( X^{\eps^2/3} )$$
(say).  Since $a_k \in \Z$  we conclude that $a_k = 0$. Proceeding by induction we obtain that $a_i=0$ for all $i>0$. 

Now that if $g \in \Poly( \R\to G)$ is of polynomial size then $\bar g = g[G,G] \in \Poly( \R\to G/[G,G])$  is also of polynomial size,
 since if  $g(t)=g_0 g_1^{\binom{t}{1}} \dots g_k^{\binom{t}{k}}$ then  $\bar g(t)=   \bar g_0 \bar g_1^{\binom{t}{1}} \dots \bar g_k^{\binom{t}{k}}$,  where $\bar g_i =  g_i[G,G]$. 
Consider  $\bar g$ now as a polynomial map in $ \Poly(\frac{q}{Q}\mathbb Z \to \Gamma/[\Gamma, \Gamma])$, then by Lemma~\ref{na-dte} we have the Taylor expansion
$$ \bar g\left(\frac{q}{Q} t\right) =   \bar \gamma_0 \bar \gamma_1^{\binom{t}{1}} \dots \bar \gamma_k^{\binom{t}{k}}$$
where $\bar \gamma_i =  \gamma_i[\Gamma, \Gamma]$.
By the claim for abelian groups we have for each $i \geq 1$ that $\bar \gamma_i =1$, so that $\gamma_i \in [\Gamma, \Gamma]$.  The claim now follows by induction on the derived sequence. 
\end{proof}

Using this lemma, we obtain the following.

\begin{proposition}[Splitting]\label{split}  Let the notation and hypotheses be as in Proposition~\ref{prop41-analog-nil}, and assume $\ell_1 \geq \ell_2$.  Let  $e = (I''_1,I''_2,p'_1,p'_2) \in {\mathcal Q}'$ and $\vec a = (a_1,b_1,a_2,b_2)$ in ${\mathcal A}_e$, and write $\phi_{I''_i} = (I''_i, g_{I''_i})$ for $i=1,2$ and some $g_{I''_i} \in \Poly(\R \to G)$.  Then we may factor
\begin{equation}\label{factors}
 g_{I''_i} = \tilde g_{e,\vec a, i} \gamma_{e,\vec a,i}
\end{equation}
for $i=1,2$, where $\tilde g_{e,\vec a, i} \in \Poly(\R \to G)$ is of polynomial size (with exponents that can depend on $\ell_2$, but are independent of $\ell_1$) and $\gamma_{e,\vec a, i}$ is $\prod {\mathcal P}_{e,\vec a}$-rational.  Furthermore, we have the approximate dilation invariance
\begin{equation}\label{gip}
 \tilde g_{e,\vec a,i}\left(\frac{a_i}{b_i} \cdot\right) = \eps_i \tilde g_{e,\vec a,i} \gamma_i
\end{equation}
for some $\gamma_i = \gamma_{i, e,\vec a} \in \Gamma$ of polynomial size, and some $\eps_i = \eps_{i, e, \vec a} \in \Poly(\R \to G)$ that is smooth on $I''_i$.  In a similar vein we have
\begin{equation}\label{gip-2}
 \tilde g_{e,\vec a,1}(p'_2 \cdot) = \eps^\dagger \tilde g_{e,\vec a,2}(p_1' \cdot) 
\end{equation}
for some $\eps^\dagger = \eps^\dagger_{e,\vec a} \in \Poly(\R \to G)$ that is smooth on $\frac{1}{p'_2} I''_1$, and
\begin{equation}\label{egad}
\gamma_{e,\vec a,1}(p'_2 \cdot) = \gamma_{e,\vec a,2}(p'_1 \cdot) \gamma^{\dagger} 
\end{equation}
for some $\gamma^\dagger = \gamma^\dagger_{e,\vec a} \in \Poly( \frac{1}{\prod {\mathcal P}_{e,\vec a}} \Z \to \Gamma)$.
\end{proposition}

The fact that the polynomial size bounds for $\tilde g_{e,\vec a, 1}$ depend only on the smaller exponent $\ell_2$ rather than the larger one $\ell_1$ will be crucial in our subsequent analysis.

\begin{proof}  Let $i=1,2$, and set $Q \coloneqq \prod {\mathcal P}_{e,\vec a}$. From \eqref{adi} and Definition~\ref{nil-comp} one has
\begin{equation}\label{ep}
 g_{I''_i}(a_i \cdot) = \eps^*_i g_{I''_i}(b_i \cdot) \gamma^*_i
\end{equation}
where $\gamma^*_i$ is $Q$-rational and $\eps^*_i$ is smooth on $\frac{1}{a_i} I''_i$.  Applying \eqref{ber-3} to the polynomial $\log \eps^*_i$ we conclude that $\eps^*_i$ is of polynomial size (with exponents that do not depend on $\ell_1,\ell_2$).  We now claim inductively for every $j=1,\dots,k+1$ that we can factor
\begin{equation}\label{gjsplit}
 g_{I''_i} = \tilde g_{e,\vec a,i,j} g_{e,\vec a,i,j} \gamma_{e,\vec a,i,j}
\end{equation}
where $\tilde g_{e,\vec a,i,j} \in \Poly(\R \to G)$ is of polynomial size (with exponents that may depend on $\ell_i$ but not on $\ell_{3-i}$), $\gamma_{e,\vec a,i,j} \in \Poly(\R \to G)$ is $Q$-rational, and $g_{e,\vec a,i,j} \in \Poly(\R \to G_j)$ takes values in $G_j$; setting $j=k+1$ then gives the desired claim \eqref{factors} for $i=2$ at least; for $i=1$ we will have the issue that the exponents depend on $\ell_1$ rather than $\ell_2$, but we will return to fix this issue later.  

The inductive claim is trivial for $j=1$ (set $g_{e,\vec a,i,1} = g_{I''_i}$ with $\tilde g_{e,\vec a,i,1}, \gamma_{e,\vec a,i,1}$ trivial); now suppose that the claim has been established for some $1 \leq j \leq k$.  In this argument all exponents are allowed to depend on $\ell_i$ but not on $\ell_{3-i}$. Then from \eqref{ep} we see that
$$ g_{e,\vec a,i,j}(a_i \cdot) = \eps_j g_{e,\vec a,i,j}(b_i \cdot) \gamma_j$$
for some $\eps_j$ of polynomial size and $Q$-rational $\gamma_j$ (we suppress the dependence of these maps on $e,\vec a,i$ for brevity).  Quotienting by $G_j$ we see that $\eps_j^{-1}$ and $\gamma_j$ agree modulo $G_j$, and hence by Lemma~\ref{disj} applied to $G/G_j$ are both equal modulo $G_j$ to a constant $\gamma \in \Gamma$ of polynomial size.  Thus we have
$$ g_{e,\vec a,i,j}(a_i \cdot) = \tilde \eps_j \gamma^{-1} g_{e,\vec a,i,j}(b_i \cdot) \gamma \tilde \gamma_j$$
for some $\tilde \eps_j$ of polynomial size taking values in $G_j$, and $Q$-rational $\tilde \gamma_j$ taking values in $G_j$.  In the abelian group $G_j/ G_{j+1}$, we thus have the identity
$$ g_{e,\vec a,i,j}(a_i \cdot) = \tilde \eps_j g_{e,\vec a,i,j}(b_i \cdot) \tilde \gamma_j \mod G_{j+1}$$
and thus on taking logarithms and working in the abelian Lie algebra $\log G_j/\log G_{j+1}$ (noting from Appendix~\ref{bch} that the logarithm map is a homomorphism from $G_j/G_{j+1}$ to $\log G_j/\log G_{j+1}$), we have from \eqref{gh} that
$$ \log g_{e,\vec a,i,j}(a_i \cdot) = \log \tilde \eps_j + \log g_{e,\vec a,i,j}(b_i \cdot) + \log \tilde \gamma_j \mod \log G_{j+1}.$$
For $d=0,\dots,j$, we may differentiate $d$ times at $0$ and rearrange to conclude that
$$ (a_i^d - b_i^d) (\log g_{e,\vec a,i,j})^{(d)}(0) = (\log \tilde \eps_j)^{(d)}(0) + (\log \tilde \gamma_j)^{(d)}(0) \mod \log G_{j+1}.$$
As $\tilde \eps_j$ is of polynomial size, we have
$$  (\log \tilde \eps_j)^{(d)}(0) = O( X^{O(1)} ).$$
Similarly, as $\tilde \gamma_j$ is $Q$-rational, $(\log \tilde \gamma_j)^{(d)}(0) \mod \log G_{j+1}$ takes values in $\frac{Q}{q} \log \Gamma_j \mod \log G_{j+1}$ for some positive integer $q$ of polynomial size.  Since $a_i^d-b_i^d$ is also a positive integer of polynomial size, we conclude that
$$ (\log g_{e,\vec a,i,j})^{(d)}(0) = O(X^{O(1)}) + \frac{Q}{q_d} \gamma_d \mod \log G_{j+1}$$
for some $\gamma_d \in \Gamma_j$ and positive integer $q_d$ of polynomial size.  By Taylor expansion (and clearing denominators with the $q_d$), we may then write
$$ \log g_{e,\vec a,i,j} = \log g^*_j + \log \gamma^*_j \mod \log G_{j+1}$$
where $g^*_j \in \Poly(\R \to G_j)$ is of polynomial size and $\gamma^*_j \in \Poly(\R \to G_j)$ is $Q$-rational.  Exponentiating (noting that $G_j/G_{j+1}$ is abelian), we conclude that
$$ g_{e,\vec a,i,j} = g^*_j g_{e,\vec a,i,j+1} \gamma^*_j$$
for some $g_{e,\vec a,i,j+1} \in \Poly(\R \to G_{j+1})$.  Inserting this into \eqref{gjsplit} we close the induction and establish \eqref{factors} (with the above caveat regarding the exponents depending on $\ell_i$ rather than $\ell_2$). 

From \eqref{phip-big} we have
$$ g_{I''_1}(p'_2 \cdot) = \eps^{\dagger} g_{I''_2}(p_1 \cdot) \gamma^{\dagger}$$
for some $\eps^{\dagger} \in \Poly(\R \to G)$ smooth on $\frac{1}{p'_2} I''_1$, and some $\gamma^{\dagger} \in \Poly(\frac{1}{Q}\Z \to \Gamma)$; in particular, $\gamma^\dagger$ is $Q$-rational with exponents that do not depend on $\ell_1$ or $\ell_2$.  Combining this with \eqref{factors} and rearranging, we see that
$$ \tilde g_{e,\vec a,2}(p'_1 \cdot)^{-1} (\eps^{\dagger})^{-1} \tilde g_{e,\vec a,1}(p'_2 \cdot) =  \gamma_{e,\vec a,2}(p'_1 \cdot) \gamma^{\dagger} \gamma_{e,\vec a,1}(p'_2 \cdot)^{-1}.$$
The left-hand side is of polynomial size and the right-hand side is $Q$-rational.  Here the exponents depend on both $\ell_1,\ell_2$; since $\ell_1 \geq \ell_2$, we can view these exponents as depending on $\ell_1$ only.  Applying Lemma~\ref{disj}, both sides are equal to a constant $\gamma \in \Gamma$ of polynomial size (with exponents depending on $\ell_1,\ell_2$).  By multiplying $\tilde g_{e,\vec a,1}$ on the right by $\gamma^{-1}$ (and $\gamma_{e,\vec a,1}$ on the left by $\gamma$), we can assume that $\gamma=1$, without significantly worsening any of the claimed properties of these objects, thus we may assume without loss of generality that $\gamma=1$.  Once one makes this normalization, one obtains the factorizations \eqref{gip-2}, \eqref{egad}.  Furthermore, since the right-hand side of \eqref{gip-2} is of polynomial size with exponents depending only on $\ell_2$, the left-hand side is also.  Hence we have now resolved the previously mentioned caveat in \eqref{factors} in that the exponents for the polynomial size nature of $\tilde g_{e,\vec a,1}$ were depending on $\ell_1$ rather than $\ell_2$.

Inserting \eqref{factors} back into \eqref{ep} and rearranging, we conclude that
$$
\tilde g_{e,\vec a,i}(b_i \cdot)^{-1} (\eps^*_i(t))^{-1} \tilde g_{e,\vec a,i}(a_i t) = \gamma_{e,\vec a,i}(b_i t) \gamma^*_i(t) \gamma_{e,\vec a,i}^{-1}(a_i t).$$
As the left-hand side is of polynomial size and the right-hand side is $Q$-rational, we conclude from Lemma~\ref{disj} that both sides are equal to a constant $\gamma_i \in \Gamma$ of polynomial size.  This rearranges to give 
$$ \tilde g_{e,\vec a,i}(a_i t) = \eps^*_i(t) \tilde g_{e,\vec a,i}(b_i t) \gamma_i$$
and therefore the claim \eqref{gip}  follows from reparameterizing $t$ and defining $\eps_i(t) \coloneqq \eps^*_i(a_i t)$. 
\end{proof}

At this point we encounter a minor technical complication due to the fact that the factors $\tilde g_{e,\vec a,i}, \gamma_{e,\vec a,i}$ generated by the above proposition depend on $\vec a$, so in particular as one varies $a_i,b_i$ the polynomial map $\tilde g_{e,\vec a,i}$ appearing in relations such as \eqref{gip} also varies.  Fortunately, using some arguments of a graph theoretic nature, and taking advantage of the ability to make the two parameters $\ell_1,\ell_2$ differ significantly from each other, we can eliminate this dependence:

\begin{proposition}[Approximate dilation invariance for a dense set of dilations]\label{adi-dense}  Let  $e = (I''_1,I''_2,p'_1,p'_2) \in {\mathcal Q}'$, and let $g_{I''_1}, g_{I''_2} \in \Poly(\R \to G)$ be the maps associated to $\phi_{I''_1}, \phi_{I''_2}$.  Assume that $\ell_1$ is sufficiently large depending on $\ell_2$.  Then there is a large set ${\mathcal P}_e$ of primes in $[P''/2,P'']$ and a factorization
\begin{equation}\label{factors-alt}
 g_{I''_i} = \tilde g_{e,i} \gamma_{e,i}
\end{equation}
for each $i=1,2$, where $\tilde g_{e,i} \in \Poly(\R \to G)$ is of polynomial size and $\gamma_{e,i}$ is $\prod {\mathcal P}_e$-rational, one has the relation 
\begin{equation}\label{gip-2-alt}
 \tilde g_{e,1}(p'_2 \cdot) = \eps^\dagger \tilde g_{e,2}(p_1 \cdot) 
\end{equation}
for some $\eps^\dagger \in \Poly(\R \to G)$ that is smooth on $\frac{1}{p'_2} I''_1$, and one has the relation 
\begin{equation}\label{egad-alt}
\gamma_{e,1}(p'_2 \cdot) = \gamma_{e,2}(p'_1 \cdot) \gamma^{\dagger} 
\end{equation}
for some $\gamma^\dagger \in \Poly( \frac{1}{\prod {\mathcal P}_{e}} \Z \to \Gamma)$.  (In all these cases we permit the exponents to depend on both $\ell_1$ and $\ell_2$.) Furthermore: 
\begin{itemize}
\item[(i)]  There exists a measurable subset $\Omega_e$ of the interval $[1 + \frac{1}{CN}, 1 + \frac{C}{N}]$ for some fixed constant $C>0$ of measure $\gg 1 / N$, such that for each $\alpha \in \Omega_e$ one has the approximate dilation invariance
\begin{equation}\label{gia}
  \tilde g_{e,1}(\alpha \cdot) = \eps_{\alpha} \tilde g_{e,1} \gamma_{\alpha}
	\end{equation}
for some $\gamma_{\alpha} \in \Gamma$ of polynomial size, and some $\eps_{\alpha} \in \Poly(\R \to G)$ that is smooth on $I''_1$. 
\item[(ii)]  We have $\tilde g_{e,1}(x_{I''_1}) = O(1)$.
\end{itemize}
\end{proposition}

\begin{proof} We first observe that we may drop the conclusion (ii) as follows.  Suppose we have already obtained all the conclusions of the proposition other than (ii).  Then $\tilde g_{e,1}(x_{I''_1})$ is already of polynomial size.  Since $G/\Gamma$ is compact, we may write
$$\tilde g_{e,1}(x_{I''_1}) = O(1) \gamma$$
for some $\gamma \in \Gamma$ of polynomial size.  If we then multiply $\tilde g_{e,1}$ on the right by $\gamma^{-1}$, multiply $\gamma_{e,1}$ and $\gamma^\dagger$ on the left by $\gamma$, and replace the lattice element $\gamma_{\alpha}$ appearing in \eqref{gia} by $\gamma \gamma_{\alpha} \gamma^{-1}$, we thus see that we may recover the claimed property (ii), without significantly impacting any of the other claims.

Henceforth we focus on establishing the remaining conclusions of the proposition. For $i=1,2$, let $V_i$ denote the set of ratios $\frac{a_i}{b_i}$ of coprime positive integers $a_i,b_i$ that are products of $\ell_i$ primes in $[P',2P']$ with 
$$ \frac{a_i}{b_i} - 1 \asymp \frac{1}{N} \asymp \frac{H}{X}.$$
By~\cite[Lemma 2.6]{mrt-fourier}, $V_i$ has cardinality $O( d^{\ell_i} / N )$.  From Proposition~\ref{prop41-analog-nil}, we see that for any $e \in {\mathcal Q}'$, the set
$$ E_e \coloneqq \left\{ (\frac{a_1}{b_1}, \frac{a_2}{b_2}): (a_1,b_1,a_2,b_2) \in {\mathcal A}_e \right\}$$
is a subset of $V_1 \times V_2$ of cardinality $\gg d^{\ell_1+\ell_2} / N^2$, thus $\# V_i \asymp d^{\ell_i} / N$ and $\# E_e \asymp (\# V_1) (\# V_2)$.  We view $E_e$ as a dense bipartite graph on $V_1, V_2$.  Each edge $\vec a = (\frac{a_1}{b_1}, \frac{a_2}{b_2})$ in $E_e$ is associated to a large set of primes ${\mathcal P}_{e,\vec a} \coloneqq {\mathcal P}_{e,(a_1,b_1,a_2,b_2)}$ in $[P''/2,P'']$.  In particular
$$ \sum_{\vec a \in E_e} \# {\mathcal P}_{e,\vec a} \gg (\# V_1) (\# V_2) \pi_0(P'')$$
which we rearrange as
$$ \sum_{p'' \in [P''/2,P'']} \sum_{v_2 \in V_2} \# \{ v_1 \in V_1: (v_1,v_2) \in E_e; p'' \in {\mathcal P}_{e,(v_1,v_2)} \} \gg \pi_0(P'') (\# V_1) (\# V_2).$$
By Cauchy--Schwarz, this implies that
$$ \sum_{p'' \in [P''/2,P'']} \sum_{v_2 \in V_2} \# \{ v_1 \in V_1: (v_1,v_2) \in E_e; p'' \in {\mathcal P}_{e,(v_1,v_2)} \}^2 \gg \pi_0(P'') (\# V_1)^2 (\# V_2),$$
which we rearrange as
$$ \sum_{(v_1,v_2) \in E_e} \sum_{v'_1 \in V_1: (v'_1,v_2) \in E_e} \# ({\mathcal P}_{e,(v_1,v_2)} \cap {\mathcal P}_{e,(v'_1,v_2)}) \gg \pi_0(P'') (\# V_1)^2 (\# V_2).$$
Hence by the pigeonhole principle there exists $(v_1,v_2) \in E_e$ for which
$$ \sum_{v'_1 \in V_1: (v'_1,v_2) \in E_e} \# ({\mathcal P}_{e,(v_1,v_2)} \cap {\mathcal P}_{e,(v'_1,v_2)}) \gg \pi_0(P'') \# V_1$$
which implies that 
$$ \# ({\mathcal P}_{e,(v_1,v_2)} \cap {\mathcal P}_{e,(v'_1,v_2)}) \gg \pi_0(P'') $$
and $(v'_1,v_2) \in E_e$ for all $v'_1$ in a subset $V_e$ of $V_1$ of cardinality $\gg \# V_1 \gg d^{\ell_1}/N$.  

Set ${\mathcal P}_e \coloneqq {\mathcal P}_{e,(v_1,v_2)}$.  From Proposition~\ref{split} applied to the quadruple $(v_1,v_2)$, we obtain factorizations 
\begin{equation}\label{factors-prime}
 g_{I''_i} = \tilde g_{e,i} \gamma_{e,i}
\end{equation}
for $i=1,2$, where $\tilde g_{e,i} = \tilde g_{e,(v_1,v_2),i} \in \Poly(\R \to G)$ is of polynomial size and $\gamma_{e,i} = \gamma_{e,(v_1,v_2),i}$ is $\prod {\mathcal P}_{e}$-rational, obeying 
\begin{equation}\label{gip-prime}
 \tilde g_{e,1}(p'_2 \cdot) = \eps^\dagger \tilde g_{e,2}(p_1 \cdot) 
\end{equation}
for some $\eps^\dagger = \eps^\dagger_{e,(v_1,v_2)} \in \Poly(\R \to G)$ that is smooth on $\frac{1}{p'_2} I''_1$.  For any $v'_1 \in V_e$, we also have a factorization
\begin{equation}\label{g2}
 g_{I''_1} = \tilde g_{e,(v'_1,v_2),1} \gamma_{e,(v'_1,v_2),1},
\end{equation}
where $\tilde g_{e,(v'_1,v_2),1}$ is of polynomial size and $\gamma_{e,(v'_1,v_2),1}$ is $\prod {\mathcal P}_{e,(v'_1,v_2)}$-rational, and 
\begin{equation}\label{ham}
 \tilde g_{e,(v'_1,v_2),1}(v'_1 \cdot) = \eps_{v'_1} \tilde g_{e,(v'_1,v_2),1} \gamma_{v'_1}
\end{equation}
for some $\eps_{v'_1}$ smooth on $I''_1$ and $\gamma_{v'_1} \in \Gamma$ of polynomial size.  From \eqref{factors-prime}, \eqref{g2} we have
$$ \tilde g_{e,1}^{-1} \tilde g_{e,(v'_1,v_2),1} = \gamma_{e,1} \gamma_{e,(v'_1,v_2),1}^{-1}.$$
The left-hand side is of polynomial size and the right-hand side is $\prod ({\mathcal P}_{e,(v_1,v_2)} \cap {\mathcal P}_{e,(v'_1,v_2)})$-rational.  By Lemma~\ref{disj}, both sides are then equal to a constant $\gamma^*_{v'_1} \in \Gamma$ of polynomial size, thus
$$ \tilde g_{e,(v'_1,v_2),1} = \tilde g_{e,1} \gamma^*_{v'_1}.$$
We conclude from \eqref{ham} that
$$
 \tilde g_{e,1}(v'_1 \cdot) = \eps_{v'_1} \tilde g_{e,1} \tilde \gamma_{v'_1}
$$
for all $t \in \R$, where $\tilde \gamma_{v'_1} \coloneqq \gamma^*_{v'_1} \gamma_{v'_1} (\gamma^*_{v'_1})^{-1}$ is an element of $\Gamma$ of polynomial size.  This gives the bound \eqref{gia} for all $\alpha$ in the discrete set $V_e$.  This is not yet what we need because $V_e$ has measure zero.  However we can use the hypothesis that $\ell_1$ is large compared to $\ell_2$ to remove the discretization as follows.  Recall from Proposition~\ref{split} that $\tilde g_{e,1}$ is of polynomial size, with exponents depending only on the smaller parameter $\ell_2$ and not on the larger parameter $\ell_1$.  As a consequence, if \eqref{gia} holds for some real number $\alpha = 1 + O( \frac{1}{N} )$, then one can perturb $\alpha$ by at most $d^{-\ell_1/10}$ (say) and still retain \eqref{gia} with only a negligible change in all the implied constants.  Hence we have \eqref{gia} for all $\alpha \in \Omega_e$, where $\Omega_e$ is the $d^{-\ell_1/10}$-neighborhood of $V_e$.  We have
$$
\int_{\Omega_e} \sum_{\alpha \in V_e} 1_{[\alpha - d^{-\ell_1/10}, \alpha+d^{-\ell_1/10}]}(\beta)\ d\beta = 2 d^{-\ell_1/10} \# V_e \gg
d^{9\ell_1/10}/N.$$
To obtain the desired lower bound of $\gg 1/N$ on the measure of $\Omega_e$, it suffices to establish the pointwise bound
$$ \sum_{\alpha \in V_e} 1_{[\alpha - d^{-\ell_1/10}, \alpha+d^{-\ell_1/10}]}(\beta) \ll d^{9\ell_1/10}$$
for any $\beta = 1 + O(1/N)$.  The left-hand side can be written as
$$ \# \{ \alpha \in V_e: |\alpha - \beta| \leq d^{-\ell_1/10} \}.$$
This in turn can be bounded by the number of pairs $(a,b) \in S^2$ with $\frac{a}{b} = \beta + O( d^{-\ell_1/10} )$, where $S$ is the collection of
 products of $\ell_1$ primes in $[P',2P']$.  This can then be bounded by
$$ d^{\ell_1/10} \int_0^\infty f(t) f(\beta t)\frac{dt}{t}$$
where
$$ f(t) \coloneqq  \# (S \cap [(1-C_1 d^{-\ell_1/10})t, (1+C_1d^{-\ell_1/10})t])$$
for some absolute constant $C_1>0$.  By Cauchy--Schwarz, the previous expression may be bounded by
$$ d^{\ell_1/10} \int_0^\infty f(t)^2 \frac{dt}{t}$$
which is in turn bounded by the number of pairs $(a,b) \in S^2$ with $\frac{a}{b} = 1 + O(d^{-\ell_1/10})$.  Applying~\cite[Lemma 2.6]{mrt-fourier}, this quantity is $O( d^{9\ell_1/10} )$, and the claim follows.
\end{proof}

Now that we have established an approximate dilation invariance \eqref{gia} for a large set of dilation parameters $\alpha$, we can begin solving this equation effectively.  The first step is as follows.

\begin{proposition}\label{prop1}  Let $e, I''_1, \tilde g_{e,1}, \Omega_e, \gamma_\alpha$ be as in Proposition~\ref{adi-dense}.  Then for any $\alpha \in \Omega_e$, we have the estimate
\begin{equation}\label{gas}
 \tilde g_{e,1}(t) = O(1) \gamma_{\alpha}^{\frac{\log(t / x_{I''_1})}{\log \alpha}} 
\end{equation}
for real $t$ with $\langle t \rangle_{I''_1} \ll 1$.	 As a consequence, for any $\alpha, \alpha' \in \Omega_e$, we have
\begin{equation}\label{gas-2}
\gamma_{\alpha'}^s = O(1) \gamma_\alpha^{s \frac{\log \alpha'}{\log \alpha}}.
\end{equation}
for all $s=O(1)$.
\end{proposition}

\begin{proof}  From iterating \eqref{gia} we see that for any fixed natural number $n$ and any $\alpha \in \Omega_e$ we have
$$
  \tilde g_{e,1}(\alpha^n x_{I''_1}) = O(1) \tilde g_{e,1}(x_{I''_1}) \gamma_{\alpha}^n
$$
which we rearrange as
\begin{equation}\label{gbang}
 \tilde g_{e,1}(\exp( n \log \alpha ) x_{I''_1}) \gamma_{\alpha}^{-n} = O(1).
\end{equation}
The left-hand side is a (matrix-valued) exponential polynomial in $n$, with the exponents in the exponentials being bounded multiples of $\log \alpha$ and thus of size $O(1/N)$.  Applying Lemma~\ref{ber-exp} to each component of this matrix-valued function, we conclude that \eqref{gbang} holds for all \emph{real} $n=O(1)$.  Rearranging using the fact that $\log \alpha \asymp \frac{1}{N}$, we conclude the estimate \eqref{gas}.  Applying this estimate twice we conclude that
$$ \tilde g_{e,1}( e^{s\log \alpha'} x_{I''_1} ) = O(1) \gamma_{\alpha'}^s$$
and
$$ \tilde g_{e,1}( e^{s\log \alpha'} x_{I''_1} ) = O(1) \gamma_{\alpha}^{s \frac{\log \alpha'}{\log \alpha}}$$
for $\alpha,\alpha' \in \Omega_e$ and $s = O(1)$, giving \eqref{gas-2}.
\end{proof}

Now we give some satisfactory control on $\tilde g_{e,1}$, which roughly speaking asserts that $\tilde g_{e,1}$ ``pretends to be like'' 
$t\mapsto T^{\log(t / x_{I''_1})}$ for some $T$ which is either nearly central, or nearly contained in a proper subgroup of $G$. Following~\cite{gt-nil}, we define a \emph{horizontal character} to be a continuous additive homomorphism $\eta: G \to \R/\Z$ that annihilates $\Gamma$; its derivative $d\eta: \log G \to \R$ at the identity is then a linear functional on $\log G$, and is related to $\eta$ by the formula
\begin{equation}\label{etag}
 \eta(g) = d\eta(\log g) \mod \Z,
\end{equation}
as can be seen by starting with the formula $\eta(g) = n \eta( \exp( \frac{1}{n} \log g ) )$ and taking limits as $n \to \infty$.  In particular, $\eta$ is the descent of the homomorphism $d\eta \circ \log: G \to \R$ to $\R/\Z$.

\begin{example} Let $G$ be the Heisenberg group from Example~\ref{heisen}, and let $\Gamma$ be the lattice
$$ \Gamma \coloneqq \begin{pmatrix} 1 & \Z & \Z \\ 0 & 1 & \Z \\ 0 & 0 & 1 \end{pmatrix}.$$
Then every horizontal character $\eta \colon G \to \R/\Z$ takes the form
$$ \eta \left( \begin{pmatrix} 1 & x & z \\ 0 & 1 & y \\ 0 & 0 & 1 \end{pmatrix} \right) = ax + by \hbox{ mod } 1$$
for some integers $a,b$, and the corresponding map $d\eta \colon \log G \to \R$ is given by
$$ d\eta \left( \begin{pmatrix} 0 & x & z \\ 0 & 1 & y \\ 0 & 0 & 0 \end{pmatrix} \right) = ax + by.$$
\end{example}

\begin{proposition}[Description of $\tilde g_{e,1}$]\label{desc} 
Let $e, \tilde g_{e,1}, I''_1$ be as in Proposition~\ref{adi-dense}.   
\begin{enumerate}
\item Then there exists $T = T_e \in G$ of polynomial size such that the map
\begin{equation}\label{sio}
 t \mapsto \log\left( \tilde g_{e,1}(t) T^{-\log(t / x_{I''_1})} \right)
\end{equation}
is bounded by $O(1)$ and has a Lipschitz norm of $O(|I''_1|^{-1})$ whenever $\langle t \rangle_{I''_1} \ll 1$.  
\item There is a non-trivial horizontal character $\eta = \eta_e: G \to \R/\Z$ such that $d\eta: \log G \to \R$ has operator norm $O(1)$, and such that 
\begin{equation}\label{deat}
d\eta( \log T ) = O( N ).
\end{equation}
\end{enumerate}
\end{proposition}

\begin{proof} Let $\Omega_e$ and $\gamma_\alpha$ be as in Proposition~\ref{adi-dense}.  Let $\alpha_0$ be an arbitrary element of $\Omega_e$, and let $T \in G$ be the quantity
$$ T \coloneqq \gamma_{\alpha_0}^{\frac{1}{\log \alpha_0}}.$$
Since $\gamma_{\alpha}$ is of polynomial size and $\alpha_0 - 1 \asymp \frac{1}{N}$, we see that $T$ is also of polynomial size. 
From \eqref{gas} one has
\begin{equation}\label{gast}
 \tilde g_{e,1}(t) = O(1) T^{\log(t / x_{I''_1})}
\end{equation}
whenever $\langle t \rangle_{I''_1} \ll 1$.  In particular, after making the substitution $u \coloneqq N \log(t/x_{I''_1})$, the function
$$u \mapsto \log\left( \tilde g_{e,1}(e^{u/N} x_{I''_1}) T^{-u/N} \right)$$ 
is bounded for $u=O(1)$.  By the Baker--Campbell--Hausdorff formula (see Appendix~\ref{bch}), this map is an exponential polynomial involving $O(1)$ terms with exponents of order $O(1/N)$.  (Note that the quantity $T^{-u/N} = \exp(-u \log T/N)$ is actually a polynomial in $u$, rather than an exponential polynomial, due to the nilpotent nature of $G$.) Applying Lemma~\ref{ber-exp}, we conclude that this map has a Lipschitz constant of $O(1)$.  Undoing the substitution, we obtain the claims regarding \eqref{sio}.

Applying \eqref{gas} again and combining with \eqref{gast}, we see that 
$$ \gamma_\alpha^s = O(1) T^{s \log \alpha}$$
for all $\alpha \in \Omega_e$ and $s = O(1)$.  If we then write
$$ g_\alpha \coloneqq T^{\log \alpha} \gamma_\alpha^{-1}$$
then $g_\alpha = T^{\log \alpha} \mod \Gamma$ and $g_\alpha = O(1)$.  Furthermore, for any $s=O(1)$ we have
\begin{align*}
 T^{s \log \alpha} g_\alpha T^{-s\log \alpha} &= T^{(s+1) \log \alpha} \gamma_\alpha^{-1} T^{-s \log \alpha} \\
&= O(1) \gamma_\alpha^{s+1} \gamma_\alpha^{-1} ( O(1) \gamma_\alpha^s )^{-1} \\
&= O(1),
\end{align*}
and thus
$$ T^t g_\alpha T^{-t} = O(1)$$
for all $\alpha \in \Omega_e$ and $t = O( \frac{1}{N} )$.  Taking logarithms and applying the Lie algebra identity \eqref{hgh}, we may rewrite this as
\begin{equation}\label{relation}
 e^{t \mathrm{ad}_{\log T}} \log g_\alpha = O(1)
\end{equation}
for all $t = O(\frac{1}{N} )$ and $\alpha \in \Omega_e$.

Let $C_0 > 0$ be a sufficiently large fixed quantity to be chosen later.  Suppose first that $\frac{1}{N} \mathrm{ad}_{\log T}$ has operator norm less than $C_0$.  The map $\mathrm{ad}: X \mapsto \mathrm{ad}_X$ is a fixed linear map from $\log G$ to the space $\mathrm{End}(\log G)$ of linear endomorphisms of $\log G$, and its kernel is $\log Z(G)$ where $Z(G)$ is the center of $G$.  The image of $\frac{1}{N} \log T$ under this map has size $O(C_0)$, hence $\frac{1}{N} \log T$ lies at a distance $O(C_0)$ from $\log Z(G)$.  On the other hand, from Proposition~\ref{nonab}, $\log Z(G)$ is a proper normal subalgebra of the Lie algebra $\log G$; using Mal'cev bases (for definition, see Appendix~\ref{bch}) it can also be seen to be rational.  By lifting a non-trivial horizontal character of $G/Z(G)$ (which can be in turn obtained by lifting a non-trivial character from the horizontal torus formed by quotienting out $G/Z(G)$ by both $\Gamma Z(G)/Z(G)$ and the commutator group $[G/Z(G),G/Z(G)]$), we may thus find a fixed non-trivial horizontal character $\eta$ that annihilates $\log Z(G)$ and such that  $d\eta$ has operator norm $O(1)$, so that \eqref{deat} holds, in which case we are done.

Henceforth we may assume that $\frac{1}{N} \mathrm{ad}_{\log T}$ has operator norm at least $C_0$.  As $\frac{1}{N} \mathrm{ad}_{\log T}$ is nilpotent, we conclude (on finite Taylor expansion of the logarithm map) that the linear map $e^{\frac{1}{N} \mathrm{ad}_{\log T}}$ has operator norm $\gg C_0^{c}$ for some constant $c>0$.  From this and the singular value decomposition, we conclude that the set
$$ \Omega = \{ x \in \log G: e^{\frac{1}{N} \mathrm{ad}_{\log T}} x = O(1) \}$$
lies in the $O(C_0^{-c})$-neighbourhood of a hyperplane $\Pi$ in $\log G$.  From \eqref{relation} we conclude that for $\alpha \in \Omega_e$, $\log g_\alpha$ lies within $O(C_0^{-c})$-neighbourhood of $\Pi$.  Since we have $g_\alpha = O(1)$ and $g_\alpha = T^{\log \alpha} \mod \Gamma$, we thus have
\begin{equation}\label{al}
T^{\log \alpha} \Gamma = g_\alpha \Gamma \in \{ \kappa \exp(h) \Gamma: \kappa \in G; h \in \Pi; \kappa = O(C_0^{-c}); h = O(1) \}.
\end{equation}
Thus, for $t=O(1/N)$ in a set of measure $\asymp 1/N$, $T^{t} \Gamma$  is contained in the $O(C_0^{-c})$-neighbourhood of the set
$$ \Sigma \coloneqq \{ \exp(h) \Gamma: h \in \Pi; h = O(1) \}.$$
Discretising this using the polynomial size of $T$, we conclude (for $A>0$ a large enough constant) that $T^t \Gamma$ lies in the $O(C_0^{-c})$ neighbourhood of $\Sigma$ for $\gg X^{A}/N$ values of $t = O(1/N)$ with $t \in X^{-A} \Z$.  If $C_0$ is large enough, this implies that the sequence $n \mapsto T^{X^{-A} n} \Gamma$  fails to be $C_0^{-C}$-equidistributed on the interval $[-CX^A/N, CX^A/N] \cap \Z$ for some fixed $C>0$, in the sense of~\cite[Definition 1.2]{gt-nil} (by testing this equidistribution hypothesis against a suitable cutoff function adapted to the $O(C_0^{-c})$-neighbourhood of $\Sigma$).  Applying~\cite[Theorem 1.16]{gt-nil}, this implies that there is a non-trivial horizontal character $\eta \colon G \to \R/\Z$ with $d\eta$ having operator norm\footnote{More precisely,~\cite[Theorem 1.16]{gt-nil} shows that $\eta$, when expressed in Mal'cev coordinates, is given by a linear functional with coefficients $O(C_0^{O(1)})$, from which it is easy to verify that $d\eta$ is also a linear functional with coefficients $O(C_0^{O(1)})$.} $O(C_0^{O(1)})$, such that
$$ \| \eta(T^{X^{-A} n}) - \eta(T^{X^{-A} (n-1)}) \|_{\R/\Z} \ll X^{-A} N$$
for $n \in $$[-CX^A/N, CX^A/N] \cap \Z$, which by \eqref{etag} implies that
$$ X^{-A} d\eta(\log T) = O( X^{-A} N ) \mod \Z.$$
For $A$ large enough, both sides here are less than $1/2$ in magnitude, so we may remove the $\mod \Z$ constraint.  
The claim follows.
\end{proof}

\begin{remark} Proposition~\ref{desc}(2) is the first place where the non-abelian nature of $G$ plays a role. Part (1) of Proposition~\ref{desc} is valid for abelian groups as well. However in part (2), if the group $G$ is abelian, then we can not find a character $\eta$ with the desired properties since the action of $\mathrm{ad}_{\log T}$ is trivial. 
\end{remark}

Having established satisfactory control on the ``continuous'' (or ``Archimedean'') component $\tilde g_{e, 1}$ on the factorization from Proposition~\ref{adi-dense}, we now need to control the ``rational'' (or ``non-Archimedean'') component $\gamma_{e, i}$, with the ultimate aim being to establish an analogue of Proposition~\ref{prop56-analog}.  We begin with a variant of Corollary~\ref{cor54-analog}.  We view $\Gamma$ as a subgroup of $\Poly(\R \to G)$, by identifying each element $\gamma$ of $\Gamma$ with the constant polynomial map $t \mapsto \gamma$.  In particular we may form the quotient space $\Gamma \backslash \Poly(\R \to G)$.

\begin{proposition}\label{cor54-nil} For each $I'' \in {\mathcal I}''$ there exists a set ${\mathcal F}(I'')$ of elements of the quotient space $\Gamma \backslash \Poly(\R \to G)$ of cardinality $O(1)$ such that for any quadruple $e = (I''_1,I''_2,p'_1,p'_2) \in {\mathcal Q}'$ and functions $\gamma_{e,i}$ as in Proposition~\ref{adi-dense}, one has
\begin{equation}\label{gam}
 \Gamma \gamma_{e,i} \in {\mathcal F}(I'')
\end{equation}
if $i=1,2$ and $I''_i = I''$.  Furthermore, each element in ${\mathcal F}(I'')$ is $1$-rational, that is to say it lies in $\Gamma \backslash \Poly( q\Z \to \Gamma)$ for some positive integer $q$ of polynomial size.
\end{proposition}

\begin{proof}  We just prove the claim for $i=1$, as the $i=2$ case is similar, and then we can obtain the joint case $i=1,2$ by taking the union of the two sets ${\mathcal F}(I'')$ thus produced.  We let ${\mathcal F}(I'')$ be the collection of all cosets $\Gamma \gamma_{e,1}$ whenever $e = (I''_1,I''_2,p'_1,p'_2) \in {\mathcal Q}'$ with $I''_1 = I''$.  Since $\gamma_{e,i}$ is $Q$-rational for some natural number $Q$, it is also $1$-rational.

Clearly we have the property \eqref{gam} by definition for $i=1$.  To complete the proof of the proposition, we need to show that ${\mathcal F}(I'')$ has cardinality $O(1)$.  Suppose for contradiction that ${\mathcal F}(I'')$ has cardinality at least $K$ for some large fixed $K$ to be chosen later.  By construction, we can then find $K$ quadruples $e_j = (I'', I''_{2,j}, p'_{1,j}, p'_{2,j}) \in {\mathcal Q}'$ for $j=1,\dots,K$ and associated factorizations
\begin{equation}\label{jfac}
 g_{I''} = \tilde g_{e_j,1} \gamma_{e_j,1}
\end{equation}
for $j=1,\dots,K$, with $\tilde g_{e_j,1} \in \Poly(\R \to G)$ of polynomial size and $\gamma_{e_j,1}$ $\prod {\mathcal P}_{e_j}$-rational for some large set ${\mathcal P}_{e_j}$ of primes in $[P''/2,P'']$, and such that the cosets $\Gamma \gamma_{e_j,1}$ are all distinct.  As each ${\mathcal P}_j$ is large, we have
$$ \sum_{j=1}^K \# {\mathcal P}_{e_j} \gg K \pi_0(P'').$$
The left-hand side can be written as $\sum_{p \in [P''/2,P'']} \# \{ 1 \leq j \leq K: p \in {\mathcal P}_j\}$.  By Cauchy--Schwarz we then have
$$ \sum_{p \in [P''/2,P'']} \# \{ 1 \leq j \leq K: p \in {\mathcal P}_{e_j}\}^2 \gg K^2 \pi_0(P'').$$  
The left-hand side may be written as
$$ \sum_{1 \leq j,j' \leq K} \# ({\mathcal P}_{e_j} \cap {\mathcal P}_{e_{j'}}).$$
For $K$ large enough, we may delete the diagonal contribution $j=j'$ and then use the pigeonhole principle to conclude that there exists $1 \leq j < j' \leq K$ for which ${\mathcal P}_{e_j} \cap {\mathcal P}_{e_{j'}}$ is large.  For this $j,j'$, we use \eqref{jfac} to conclude that
$$ \tilde g_{e_{j'},1}^{-1} \tilde g_{e_j,1} = \gamma_{e_{j'},1} \gamma_{e_j,1}^{-1}.$$
The left-hand side is of polynomial size, and the right-hand side is $\prod ({\mathcal P}_{e_j} \cap {\mathcal P}_{e_{j'}})$-rational.  By Lemma~\ref{disj}, we conclude that $\gamma_{e_{j'},1} \gamma_{e_j,1}^{-1} \in \Gamma$, thus $\Gamma \gamma_{e_j,1} = \Gamma \gamma_{e_{j'},1}$, contradicting the construction of the $e_j$.  The claim follows.
\end{proof}

We can now establish an analogue of Proposition~\ref{prop56-analog} that dramatically improves the bound on $q$.

\begin{proposition}\label{qbound}  There exists a subset ${\mathcal Q}''$ of ${\mathcal Q}'$ of cardinality $\gg dN$, such that for each 
$e \in {\mathcal Q}''$ and functions $\gamma_{e,i}$ as in Proposition~\ref{adi-dense}, one has $\gamma_{e,1} \in \Poly(q\Z \to \Gamma)$ for some $q=O(1)$.
\end{proposition}

\begin{proof}  Let $q_0$ be a sufficiently large fixed quantity to be chosen later.  Suppose for contradiction that the proposition fails, then we can find a subset ${\mathcal Q}''$ of ${\mathcal Q}'$ of cardinality at least $\frac{1}{2} \# {\mathcal Q}' \gg dN$ such that $\gamma_{e,1} \not \in \Poly(q\Z \to \Gamma)$ for any $1 \leq q \leq q_0$ and $e = (I''_1,I''_2,p'_1,p'_2) \in {\mathcal Q}''$.  By Proposition~\ref{cor54-nil}, $\Gamma \gamma_{e,1} \in {\mathcal F}(I''_1)$.  By randomly selecting one element $F_{I''_1}$ from each ${\mathcal F}(I''_1)$ and using the probabilistic method, we conclude that for at least one such choice of elements $F_{I''_1}$, there is a subset ${\mathcal Q}'''$ of ${\mathcal Q}''$ of cardinality $\gg dN$ such that
$$\Gamma \gamma_{e,i} = F_{I''_i}$$ 
for all $e = (I''_1,I''_2,p'_1,p'_2) \in {\mathcal Q}'''$ and $i=1,2$.  In particular, we have
$$ F_{I''_1} \not \in \Gamma \backslash \Poly(q\Z \to \Gamma)$$
whenever $e = (I''_1,I''_2,p'_1,p'_2) \in {\mathcal Q}'''$ and $1 \leq q \leq q_0$.

Let $\ell$ be a bounded integer, large enough so that $d^\ell \geq N d^{10}$.  Viewing ${\mathcal Q}'''$ as a (directed) graph with vertex set ${\mathcal I}''$ and applying the Blakley--Roy inequality~\cite{blakley} (see also~\cite{mulholland-smith}) and Cauchy--Schwarz to count cycles of length $2\ell$ in this graph, we conclude that there exist $\gg d^{2\ell}$ $2\ell$-tuples
\begin{equation}\label{itchy}
 (I''_{j,i})_{0 \leq j\leq \ell-1; i=1,2} \in ({\mathcal I}'')^{2\ell}
\end{equation}
with the property that for each $0 \leq j \leq \ell-1$, there exists primes $p'_{j,1}, p'_{j,2}, p'_{j,3}, p'_{j,4} \in [P',2P']$ such that
$$ (I''_{j,1}, I''_{j,2}, p'_{j,1}, p'_{j,2}), (I''_{j+1,1}, I''_{j,2}, p'_{j,3}, p'_{j,4}) \in {\mathcal Q}'''$$
for $j=0,\dots,\ell-1$, with the periodic convention $I''_{\ell,1} = I''_{0,1}$. In particular, $I''_{j,1}$ lies within $O( \frac{H}{P'P''})$ of $\frac{p'_{j,2}}{p'_{j,1}} I''_{j,2}$, and similarly $I''_{j+1,1}$ lies within $O( \frac{H}{P'P''})$ of $\frac{p'_{j,4}}{p'_{j,3}} I''_{j,2}$.  Iterating this we conclude that $I''_{0,1}$ lies within $O( \frac{H}{P'P''})$ of $\prod_{j=0}^{\ell-1} \frac{p'_{j,4} p'_{j,1}}{p'_{j,3} p'_{j,2}} I''_{0,1}$, which implies that
\begin{equation}\label{ston}
|a-b| \ll \frac{1}{N} (P')^{2\ell}.
\end{equation}
where
$$a \coloneqq  \prod_{j=0}^{\ell-1} p'_{j,4} p'_{j,1} $$
and
$$ b \coloneqq \prod_{j=0}^{\ell-1} p'_{j,3} p'_{j,2}.$$
By the pigeonhole principle, we may find an $I''_{0,1} \in {\mathcal I}''$ which is associated to a family ${\mathcal T}$ of tuples \eqref{itchy} of cardinality
\begin{equation}\label{Scratch}
\# {\mathcal T} \gg d^{2\ell}/N.
\end{equation}
We now fix this interval $I''_{0,1}$ and the family ${\mathcal T}$.

Let $q$ be the least positive integer for which
$$ F_{I''_{0,1}} \in \Gamma \backslash \Poly(q\Z \to \Gamma).$$
By construction we have
$$ q_0 < q \ll X^{O(1)}.$$
Intuitively, the lower bound $q > q_0$ means that polynomials in the coset $F_{I''_{0,1}}$ have at least one coefficient with some ``large denominator'' $n_m$.  The strategy is to locate this coefficient and this denominator, and then to study the equation \eqref{egad-alt} to obtain some non-trivial congruence conditions relating $a$ and $b$ modulo $n_m$ which will restrict the size of ${\mathcal T}$ enough to obtain a contradiction.

We turn to the details.  We arbitrarily select a coset representative $\gamma_{0,1} \in \Poly(q\Z \to \Gamma)$ of $F_{I''_{0,1}}$.  For any $l=1,\dots,k+1$, we let $\gamma_{0,1} \mod G_l$ be the projection to $\Poly(q\Z \to \Gamma G_l/G_l) \subset \Poly(\R \to G/G_l)$, and let $q_l$ be the least positive integer for which $\gamma_{0,1} \mod G_l \in \Poly(q_l\Z \to \Gamma G_l/G_l)$.  Then $q_1=1$, $q_{k+1}=q\geq q_0$, and from Lemma~\ref{bezout-nil} we have $q_i | q_{i+1}$ for $i=1,\dots,k$.  In particular, by the pigeonhole principle we can find $l\in \{1,\dots,k\}$ such that
$$ q_l \leq q_0^{\frac{i-1}{k}}$$
and
\begin{equation}\label{qo}
 q_{l+1} > q_0^{\frac{i}{k}} \geq q_0^{\frac{1}{k}} q_l.
\end{equation}
We now fix this $l$.  By lifting the Taylor coefficients of $\gamma_{0,1} \mod G_l$ from $G/G_l$ back to $G$, we can factor
\begin{equation}\label{tilt}
 \gamma_{0,1} = \gamma'_{0,1} \gamma''_{0,1} 
\end{equation}
where $\gamma''_{0,1} \in \Poly(q_l\Z \to \Gamma)$ and $\gamma'_{0,1} \in \Poly(\R \to G_l)$, hence also $\gamma'_{0,1} \in \Poly(q\Z \to \Gamma_l)$. We then see that $q_{l+1}$ is the least multiple of $q_l$ for which $\gamma'_{0,1} \mod G_{l+1} \in \Poly( q_{l+1} \Z \to \Gamma_l G_{l+1}/G_{l+1})$.  If we perform the Taylor expansion
\begin{equation}\label{dorm}
 \gamma'_{0,1}(t) = g_0 g_1^{t} \dots g_k^{t^k/k!} 
\end{equation}
for $g_0,\dots,g_k \in G_l$, then on setting $t=0$ we conclude that $g_0 \in \Gamma_l$; also, by taking repeated differences with spacing $q_{l+1}$, we see that for each $m=1,\dots,k$ we have $g_m^{a_m} \in \Gamma_l G_{l+1}$ for some positive integer $a_m$ of polynomial size.  Note that $g_1,\dots,g_m$ do not depend on the choice of coset representative $\gamma_{0,1}$. If we let $n_m$ be the least positive integer such that $g_m^{n_m q_l^m/m!} \in \Gamma G_{l+1}$, we see that each $n_m$ is of polynomial size and
$$ \gamma'_{0,1} \mod G_{l+1} \in \Poly( k! n_1 \dots n_k q_{l} \Z \to \Gamma_l G_{l+1}/G_{l+1})$$
and thus
$$ n_1 \dots n_k q_l \gg q_{l+1}$$
so by \eqref{qo} and the pigeonhole principle we can find $m=1,\dots,k$ such that
\begin{equation}\label{nm-big}
 n_m \gg q_0^{1/k^2}.
\end{equation}
Henceforth we fix this $m$.  We will shortly use this large integer $n_m$ as a modulus to which one can apply Lemma~\ref{lem26-analog}.  A key technical point is that this modulus does not depend on the tuples in ${\mathcal T}$.  

Next, we claim that after removing a negligible fraction of tuples from the family ${\mathcal T}$, we may assume that none of the $p'_{j,i}$ divide $n_m$.  For sake of notation let us just remove the contribution where $p'_{0,1}$ divides $n_m$.  There are $O(N)$ choices for $I''_{0,1}$.  As $n_m$ is of polynomial size and $p'_{0,1} \in [P',2P']$, we see that there are only $O(1)$ choices for $p'_{0,1}$.  After fixing this choice, there are at most $O(\pi_0(P')^{2\ell-1}) = O( d^{\ell-1/2} )$ choices for the remaining choices of $p'_{j,4}, p'_{j,1}$, $j=0,\dots,\ell-1$. Then we see from \eqref{ston} and the fundamental theorem of arithmetic that there are $O( \frac{1}{N} (P')^{2\ell} ) = O( d^{\ell+o(1)} / N )$ choices for the $p'_{j,3}, p'_{j,2}$.
After making all these choices, the tuple \eqref{itchy} is fixed, so the total number of tuples generated in this fashion is $O( d^{2\ell-1/2+o(1)} / N)$, which is negligible.  Similarly for the cases when some other prime $p'_{j,i}$ divides $n_m$.

For each $0 \leq j \leq \ell$ and $i=1,2$, let $\gamma_{j,i} \in \Poly(\R \to G)$ be a representative of the coset $f_{I''_{j,i}} \in \Gamma \backslash \Poly(\R \to G)$, thus $f_{I''_{j,i}} =\Gamma \gamma_{j,i}$; for $(j,i)=(0,1)$ we use the same choice $\gamma_{0,1}$ of coset representative that was made earlier.  From \eqref{egad-alt} we have for all $0 \leq j \leq \ell-1$ that
$$
\gamma_{j,1}(p'_{j,2} \cdot) = \gamma_j \gamma_{j,2}(p'_{j,1} \cdot) \gamma^{\dagger}_j 
$$
for some $\gamma_j \in \Gamma$, and some $\gamma^\dagger_j \in \Poly(\Z \to \Gamma)$, and similarly
$$
\gamma_{j+1,1}(p'_{j,4} \cdot) = \tilde \gamma_j \gamma_{j,2}(p'_{j,3} \cdot) \tilde \gamma^{\dagger}_j 
$$
for all $t \in \R$ and some $\tilde \gamma_j \in \Gamma$, and some $\tilde \gamma^\dagger_j \in \Poly(\Z \to \Gamma)$.  Concatenating these estimates, we conclude that
$$ \gamma_{0,1}( a \cdot) = \gamma \gamma_{0,1}( b \cdot ) \gamma^\dagger$$
for some $\gamma \in \Gamma$ and $\gamma^\dagger \in \Poly(\Z \to \Gamma)$.  By \eqref{tilt}, this implies that
$$ \gamma'_{0,1}( a \cdot) = \gamma \gamma'_{0,1}( b \cdot ) \gamma^{-1} \tilde \gamma^\dagger$$
where $\tilde \gamma^\dagger \in \Poly(q_l \Z \to \Gamma)$.  Since $\gamma'_{0,1}( a \cdot)$ and
$\gamma \gamma'_{0,1}( b \cdot ) \gamma^{-1}$ both take values in $G_l$, $\tilde \gamma^\dagger$ does also, thus $\tilde \gamma^\dagger \in \Poly(q_l \Z \to \Gamma_l)$.  If we now project to the torus $G_l / (\Gamma_l G_{l+1})$, we see that
$$ \gamma'_{0,1}( a q_l n) = \gamma'_{0,1}( b q_l n ) \mod \Gamma_l G_{l+1}$$
for all integers $n$.  Using the Taylor expansion \eqref{dorm}, we conclude on taking $m$ divided differences with spacing $q_l$ at the origin that
$$ g_m^{(a q_l)^m} = g_m^{(b q_l)^m} \mod \Gamma_l G_{l+1}$$
and hence by definition of $n_m$
$$ a^m = b^m \mod n_m.$$
Applying Lemma~\ref{lem26-analog}, we can then bound the number $\# {\mathcal T}$ of tuples as
$$ \# {\mathcal T} \ll \frac{d^{2\ell}}{N} \left( \frac{k^{\omega(n_m)}}{\phi(n_m)} + \frac{1}{\log N} \right) $$
which by the divisor bound and \eqref{nm-big} gives
$$ \# {\mathcal T} \ll q_0^{-\frac{1}{2k^2}} \frac{d^{2\ell}}{N} $$
which contradicts the lower bound \eqref{Scratch} if $q_0$ is large enough.
\end{proof}

Note that each $I''_1$ appears in at most $O(d)$ quadruples $e = (I''_1, I''_2, p'_1, p'_2) \in {\mathcal Q}''$.  Combining this observation with Propositions~\ref{qbound}, \ref{desc}, we conclude

\begin{corollary}\label{finally} For all $I''$ in a large subcollection ${\mathcal I}'''$ of ${\mathcal I}''$, we can find a representation
\begin{equation}\label{rep}
 F(g_{I''} \Gamma) = F( \tilde g_{I''} \gamma_{I''} \Gamma )
\end{equation}
for some $\tilde g_{I''}, \gamma_{I''} \in \Poly(\R \to G)$, and $T_{I''} \in G$ of polynomial size such that
\begin{itemize}
\item[(i)]  The map
\begin{equation}\label{sio-2}
 t \mapsto \log( \tilde g_{I''}(t) T_{I''}^{-\log(t / x_{I''})} )
\end{equation}
is bounded by $O(1)$ and has a Lipschitz norm of $O(|I''|^{-1})$ whenever $\langle t \rangle_{I''} \ll 1$.
\item[(ii)] There is a non-trivial horizontal character $\eta_{I''} \colon G \to \R/\Z$ such that the derivative $d\eta_{I''} \colon \log G \to \R$ has operator norm $O(1)$, and such that 
\begin{equation}\label{deat-2}
d\eta_{I''}( \log T_{I''} ) = O( N ).
\end{equation}
\item[(iii)] $\gamma_{I''} \in \Poly(q_{I''}\Z \to \Gamma)$ for some $q_{I''}=O(1)$.
\end{itemize}
\end{corollary}

Let $I'', \tilde g_{I''}, \gamma_{I''}, T_{I''}, \eta_{I''}, q_{I''}$ be as in the above corollary.  Observe that as the number of possible $q_{I''}$ is bounded, we may refine the family ${\mathcal I}'''$ of intervals in the above corollary by a bounded factor to assume that
$$ q_{I''} = q$$
is independent of $I''$ (one could also simply clear denominators here).  In a similar spirit, as $\eta_{I''}$ takes values in the lattice of horizontal characters (which one can identify with the Pontryagin dual of the torus $G / \Gamma[G,G]$) and is a bounded distance away from the origin, there are only finitely many choices for $\eta_{I''}$, so we may assume that
$$ \eta_{I''} = \eta$$
is independent of $I''$.

Now we will be able to descend from $G$ to the lower-dimensional nilpotent group $\mathrm{ker}(\eta)$ as follows.  Since $\eta\colon G \to \R/\Z$ is a homomorphism to the abelian group $\R/\Z$, it annihilates the commutator group $[G,G]$, and hence (by \eqref{lim}) the derivative map $d\eta\colon \log G \to \R$ annihilates the commutator algebra $[\log G, \log G]$.  In particular, from the Baker--Campbell--Hausdorff formula, we have
$$ d\eta \left(\log\left( \tilde g_{I''}(t) T_{I''}^{-\log(t / x_{I''})}\right) \right)
= d\eta( \log \tilde g_{I''}(t) ) - \log(t/x_{I''}) d\eta(\log T_{I''}).$$
If we then apply $d\eta$ to \eqref{sio-2}, we conclude that the map
$$ t \mapsto d\eta( \log \tilde g_{I''}(t) ) - \log(t/x_{I''}) d\eta(\log T_{I''}) $$
has a Lipschitz norm of $O(|I''|^{-1})$ whenever $\langle t \rangle_{I''} \ll 1$.  Combining this with \eqref{deat-2}, we see that the map
$$ t \mapsto d\eta( \log \tilde g_{I''}(t) ) $$
also has a Lipschitz norm of $O(|I''|^{-1})$ in this region.  From the definition of $\Poly(\R \to G)$, this map is also a polynomial of degree $k$, with the $t^j$ coefficient lying in $d\eta(\log G_j)$ for each $j \geq 0$.  By the Bernstein inequality \eqref{ber-2}, we may thus write
$$ d\eta(\log \tilde g_{I''}(t) ) = \sum_{j=0}^k \theta_j (t - x_{I''})^j $$
where the $\theta_j$ are real numbers with $\theta_j \in d\eta(\log G_j)$ and $\theta_j = O( |I''|^{-j} )$.  Lifting this polynomial back to $G$, we may thus write
$$ \log \tilde g_{I''}(t) = \sum_{j=0}^k X_j (t - x_{I''})^j \mod \mathrm{ker}(d \eta)$$
for some $X_j \in \log G_j$ with $X_j = O(|I''|^{-j})$.  If we set
$$ \eps_{I''}(t) \coloneqq \exp( \sum_{j=0}^k X_j (t - x_{I''})^j )$$
then $\eps_{I''} \in \Poly(\R \to G)$ is smooth on $I''$, and if we then define $g^*_{I''} \colon \R \to G$ to be the map for which
$$ \tilde g_{I''}(t) = \eps_{I''}(t) g^*_{I''}(t)$$
then from the Baker--Campbell--Hausdorff formula \eqref{poly} we see that $g^*_{I''} \in \Poly(\R \to \mathrm{ker}(\eta))$ takes values in the kernel $\mathrm{ker}(\eta) = \exp( \mathrm{ker}(d\eta))$ of $G$, which is a proper rational normal subgroup of $G$.  By \eqref{rep}, \eqref{falp} we then have
$$ \left| \sum_{n \in I''} f(n) \overline{F}( \eps_{I''}(n) g^*_{I''}(n) \gamma_{I''}(n) \Gamma) \right| \gg |I''|.$$
Let $H^* \coloneqq c \frac{H}{P'P''}$ for a sufficiently small absolute constant $c>0$.  Then we have
$$ \int_{I''}\left| \sum_{n \in [x,x+H^*]} f(n) \overline{F}( \eps_{I''}(n) g^*_{I''}(n) \gamma_{I''}(n) \Gamma) \right|\ dx \gg |I''| H^*.$$
As $\eps_{I''}$ is smooth on $I''$, $\eps_{I''}(n)$ is $O(1)$ and varies by at most $O(c)$ on $[x,x+H^*]$, hence by the Lipschitz nature of $F$
$$ \int_{I''}\left| \sum_{n \in [x,x+H^*]} f(n) \overline{F}( \eps_{I''}(x) g^*_{I''}(n) \gamma_{I''}(n) \Gamma) \right|\ dx \gg |I''| H^*.$$
Summing over $I'' \in {\mathcal I}'''$, we conclude that
$$
\int_X^{2X} \sup_{\eps \in E; \tilde g \in \Poly( \Z \to \mathrm{ker}(\eta) ); \gamma \in \Poly(q\Z \to \Gamma)} \left| \sum_{n \in [x,x+H]} f(n) \overline{F}(\eps \tilde g(n) \gamma(n) \Gamma)\right| \ dx \gg HX$$
for some compact set $E \subset G$.
But this contradicts Proposition~\ref{major}.  This contradiction (finally!) concludes the proof of Theorem~\ref{mult-pret}.

\begin{remark}\label{rem: loweringH} It seems plausible that the proof of Theorem~\ref{mult-pret}, combined with the quantitative work in Section~\ref{sec: lowering} for lowering the value of $H$, would allow lowering the length of the intervals to $H\geq \exp((\log X)^{1-\delta})$ for some $\delta>0$. We do not pursue this further here, however, as that would further lengthen this paper. Let us note, however, that at least the convenient notion of polynomially large elements in Lie groups used in this section would have to be replaced with a more cumbersome notation in the case where $H$ is no longer polynomially large in terms of $X$. 
\end{remark}

\section{Sign patterns}\label{sec: signpatterns}

\subsection{The Liouville case}

Our main goal in this section is to use Theorem~\ref{mult-pret} to prove Theorem~\ref{superpolynomial}, which asserts a superpolynomial lower bound on the number $s(k)$ of sign patterns of the Liouville function, defined in \eqref{sk-def}. We will also prove a generalization of Theorem~\ref{superpolynomial} to sign patterns of other multiplicative functions (Theorem~\ref{theo-multsigns}), and prove Proposition~\ref{entropy}.  

Regarding Theorem~\ref{superpolynomial}, we will in fact prove a more general implication, which gives a lower bound on $s(k)$ whenever one has local Gowers uniformity of the Liouville function on short intervals:

\begin{theorem}[From local Gowers uniformity to lower bounds on sign patterns]\label{theo-signs}
Let $0 < \kappa < 1/2$. Let $\Psi:\R_{\geq 1}\to \R$ be a strictly increasing function with $X\leq \Psi(X)\leq \exp(X^{1/2-\kappa})$ for all large enough $X$. Suppose that \eqref{fox} holds for $H(X)=\Psi^{-1}(X^{\eta})$ for every fixed $\eta>0$. Then $s(k)\geq \Psi(k)$ for all large enough $k$.
\end{theorem}

Taking $\Psi(X)=X^{A}$ and applying Theorem~\ref{mult-pret}, we see that Theorem~\ref{superpolynomial} follows directly from the above theorem. Furthermore, we have the following conditional corollary.

\begin{corollary}
Let $\varepsilon > 0$. Assuming that \eqref{fox} holds with $H(X)=\exp((\log X)^{1-\delta})$ for some $\delta\in (0,1)$, we have $s(k)\gg_{\varepsilon} k^{(\log k)^{\delta/(1-\delta)-\varepsilon}}$. Further, assuming \eqref{fox} with $H(X)=(\log X)^{C}$ for some $C>2$, we have $s(k)\gg_{\varepsilon} \exp(k^{1/C-\varepsilon})$.
\end{corollary}

\begin{remark}
In the proof of Theorem~\ref{theo-signs} below, one may on first reading want to assume that $\Psi(X)=X^{A}$, which corresponds to $H(X)=X^{o(1)}$, in which case we wish to show that $s(k)\gg_{A}k^{A}$ for all $A$. This simplifies various expressions involved; in particular expressions involving $\Psi$ are just large powers of the argument and expressions involving $\Psi^{-1}$ are small powers of the argument.
\end{remark}

We now begin the proof of Theorem~\ref{theo-signs}.  Fix $\kappa > 0$; we allow all implied constants to depend on $\kappa$.  Suppose for the sake of contradiction that $s(m)< \Psi(m)$ for infinitely many $m$.  We will use this to show that $s(k) = 2^k$ for all $k$.  Since $\Psi(k) < 2^k$ for all sufficiently large $k$, this will give the required contradiction.

Fix $k$; we now allow all implied constants to depend on $k$.  We now select additional parameters $\eps, w, m, R, x$, arranged so that
$$ k \ll \frac{1}{\eps} \ll w \ll m \ll R \ll x,$$
by the following scheme.
\begin{itemize}
\item First, we choose $\eps > 0$ to be a sufficiently small quantity depending on $k,\kappa$.
\item Then we choose a quantity $w > 1$ to be sufficiently large depending on $\eps,k,\kappa$.
\item Next, we choose $m$ to be a natural number with $s(m) < \Psi(m)$ that is sufficiently large (depending on $w,\eps,k,\kappa$).  Such an $m$ always exists by hypothesis.
\item One then sets $R \coloneqq \Psi(m)^{\eps^{-2}}$ and $x \coloneqq \Psi(m)^{\eps^{-3}}$.
\end{itemize}
By construction and the hypothesis $X \leq \Psi(X) \leq \exp(X^{1/2-\kappa})$, we have $R = x^\eps$,
\begin{equation}\label{m-bound}
 (\log x)^{2+\kappa} \leq m \leq x^{\eps^3},
\end{equation}
and
\begin{equation}\label{smo}
 s(m) < x^{\eps^3}.
\end{equation}

Now suppose for contradiction that $s(k) < 2^k$. Then by \eqref{sk-def} there exists a sign pattern $(\eps_1,\ldots, \eps_k)\in \{-1,+1\}^k$ which never occurs in the Liouville sequence, so in particular
\begin{align}\label{eq1}
\E_{n\leq x}^{\log}1_{\lambda(n+1)=\eps_1}\cdots 1_{\lambda(n+k)=\eps_k}=0.
\end{align}
Writing $1_{\lambda(n+j)=\eps_j} = \frac{1 + \eps_j \lambda(n+j)}{2}$, we may expand the left-hand side of \eqref{eq1} as the sum of the $2^k$ quantities of the form
\begin{align*}
(\prod_{l=1}^i \eps_{\ell_l}) 2^{-k} \E_{n\leq x}^{\log} \lambda(n+\ell_1)\cdots \lambda(n+\ell_i),\quad \textnormal{where}\quad \{\ell_1,\ldots, \ell_i\}\subset \{1,2,\ldots, k\}.
\end{align*}
The $i=0$ term is equal to $2^{-k}$.  Thus by the pigeonhole principle, there must exist $1 \leq i \leq k$ and $1 \leq \ell_1 < \ldots < \ell_i\leq k$ for which the correlation
\begin{align}\label{eq11}
C\coloneqq \E_{n\leq x}^{\log} \lambda(n+\ell_1)\cdots \lambda(n+\ell_i)
\end{align}
is such that 
\begin{equation}\label{c-bound}
|C| \gg 1.
\end{equation}
The precise choice of $i,\ell_1,\dots,\ell_i$ may depend on $x$, but this will not concern us. Henceforth let $i,\ell_1,\dots,\ell_i$ be chosen so that \eqref{c-bound} holds.

Set $P \coloneqq \frac{m}{3k}$. By using the multiplicativity relation $\lambda(pn)=-\lambda(n)$ and the fact that the correlation $C$ in \eqref{eq11} involves a logarithmic average, for all primes $p\leq 2P$ we deduce 
\begin{align*}
C&=(-1)^i \E_{n\leq x}^{\log} \lambda(pn+p\ell_1)\cdots \lambda(pn+p\ell_i)\\
&=(-1)^i \E_{n'\leq px}^{\log} \lambda(n'+p\ell_1)\cdots \lambda(n'+p\ell_i)p1_{p\mid n'}+O(\eps^3)\\
&=(-1)^i \E_{n'\leq x}^{\log} \lambda(n'+p\ell_1)\cdots \lambda(n'+p\ell_i)p1_{p\mid n'}+O(\eps^3),
\end{align*}
where the final estimate follows from \eqref{m-bound}.  Hence, by averaging over $p$,
\begin{align*}
C=(-1)^i\E_{P\leq p < 2P} \E_{n\leq x}^{\log} \lambda(n+p\ell_1)\cdots \lambda(n+p\ell_i)p1_{p\mid n}+O(\eps^3).   
\end{align*}
The contribution of $n\leq R = x^{\eps}$ to the average is trivially $\ll \eps$, so
\begin{equation}\label{coo}
C=(-1)^i\E_{P\leq p <2P} \E_{R\leq n\leq x}^{\log} \lambda(n+p\ell_1)\cdots \lambda(n+p\ell_i)p1_{p\mid n}+O(\eps),   
\end{equation}
We will shortly exploit the sign pattern bound \eqref{smo} to obtain the bound
\begin{equation}\label{c-bound-3}
 \E_{P\leq p < 2P} \E_{R\leq n\leq x}^{\log} \lambda(n+p\ell_1)\cdots \lambda(n+p\ell_i) (p1_{p\mid n}-1) \ll \eps.
\end{equation}
Assuming this bound for the moment, we may then simplify \eqref{coo} to
$$ C=(-1)^i\E_{P\leq p < 2P} \E_{R\leq n\leq x}^{\log} \lambda(n+p\ell_1)\cdots \lambda(n+p\ell_i) +O(\eps).$$

For $d \in [P,2P]$, the von Mangoldt function $\Lambda(d)$ is equal to $(1 + O(\eps)) \log P$ when $d$ is prime and is only nonzero (and of size $O(\log P)$) for $O( P^{1/2+\eps})$ other values of $d$.  Since $P$ is large compared to $\eps$, we easily conclude that
$$ C=(-1)^i\E_{P\leq d < 2P} \Lambda(d) \E_{R\leq n\leq x}^{\log} \lambda(n+d\ell_1)\cdots \lambda(n+d\ell_i) +O(\eps)$$
We now apply the ``$W$-trick''. If we set $W \coloneqq \prod_{p \leq w} p$ and split $d$ into residue classes $b \mod W$, then the contribution of the non-primitive classes $(b,W) > 1$ is negligible, and we have
\begin{equation}\label{rhs}
 C=(-1)^i\E_{\substack{1\leq b\leq W\\(b,W)=1}} \E_{P/W\leq d < 2P/W} \Lambda_{W,b}(d) \E_{R\leq n\leq x}^{\log} \lambda(n+(Wd+b)\ell_1)\cdots \lambda(n+(Wd+b)\ell_i) +O(\eps)
\end{equation}
where $\Lambda_{W,b}(d) \coloneqq \frac{\phi(W)}{W} \Lambda(Wd+b)$, and $\phi$ is the Euler totient function.
By splitting the average over $n$ into intervals of length $P/W$ and applying the Gowers uniformity of $\Lambda_{W,b}(d)-1$ (established in~\cite{green-tao}, \cite{gt-mobius}, \cite{gtz}) as in~\cite[Proposition 3.3]{TaoEq}, we find
$$ \E_{P/W\leq d < 2P/W} (\Lambda_{W,b}(d)-1) \E_{R\leq n\leq x}^{\log} \lambda(n+(Wd+b)\ell_1)\cdots \lambda(n+(Wd+b)\ell_i) \ll \eps$$
for any $b \in (\Z/W\Z)^\times$ (here we use the fact that $P$ is large compared to $W,\eps$).  We conclude that
$$ C=(-1)^i\E_{\substack{1\leq b\leq W\\(b,W)=1}} \E_{P/W\leq d < 2P/W} \E_{R\leq n\leq x}^{\log} \lambda(n+(Wd+b)\ell_1)\cdots \lambda(n+(Wd+b)\ell_i) +O(\eps),$$
or equivalently
$$ C=(-1)^i \frac{W}{\phi(W)} \E_{P\leq d < 2P} 1_{(d,W)=1} \E_{R\leq n\leq x}^{\log} \lambda(n+d\ell_1)\cdots \lambda(n+d\ell_i) +O(\eps).$$

Splitting the $n$ sum into intervals of length $m = 3kP$ and using the triangle inequality, we obtain
$$ C \ll \frac{W}{\phi(W)} \E_{P \leq d < 2P} \E_{n \leq x}^{\log} |\E_{n \leq n' \leq n+m} \lambda(n'+d\ell_1)\cdots \lambda(n'+d\ell_i)| + \eps.$$
Embedding $[n,n+m]$ into a cyclic group of prime order, and applying the generalized von Neumann theorem in the form of~\cite[Proposition 7.1]{green-tao}, we have
$$ \frac{W}{\phi(W)} \E_{P \leq d < 2P}|\E_{n \leq n' \leq n+m} \lambda(n'+d\ell_1)\cdots \lambda(n'+d\ell_i)| \ll O_W(\kappa(\|\lambda\|_{U^{k}[n,n+m]}))+\varepsilon$$
for some bounded function $\kappa(x)$ tending to $0$ as $x\to 0$, and so we conclude that
\begin{align}\label{eq5}
C\ll O_W(\E_{n\leq x}^{\log}\kappa(\|\lambda\|_{U^{k}[n,n+m]}))+ \eps.     
\end{align}
Since $m = \Psi^{-1}(x^{\eps^3})$, we conclude from the assumption of the theorem (and the fact that $x$ is sufficiently large depending on $w,k,\eps$) that
$$ C \ll \eps,$$
but this contradicts \eqref{c-bound} for $\eps$ small enough.

To conclude the proof of Theorem~\ref{theo-signs}, it remains to establish the bound \eqref{c-bound-3}.  This is reminiscent of the bounds one can establish by entropy decrement arguments as seen for instance in~\cite{TaoEq}; however the size of $P$ compared to $x$ is too large here for such methods to apply (and furthermore these methods need to exclude an exceptional set of bad scales $P$).  The key observation is that one can instead exploit the small number \eqref{smo} of sign patterns of length $m = 3kP$ to obtain a strong estimate via the moment method.  Firstly, by approximate translation invariance we can write
$$\E_{P\leq p < 2P} \E_{R\leq n\leq x}^{\log} \lambda(n+p\ell_1)\cdots \lambda(n+p\ell_i) (p1_{p\mid n}-1)$$
as 
$$ \E_{P\leq p < 2P} \E_{R\leq n\leq x}^{\log} \lambda(n+j+p\ell_1)\cdots \lambda(n+j+p\ell_i) (p1_{p\mid n+j}-1) + O(\eps)$$
for any $1 \leq j \leq P$, thus on averaging we may also write it as
$$\E_{R\leq n\leq x}^{\log} \E_{P\leq p < 2P} \E_{j \leq P} \lambda(n+j+p\ell_1)\cdots \lambda(n+j+p\ell_i) (p1_{p\mid n+j}-1) + O(\eps).$$
Thus by the triangle inequality, it suffices to show that
$$ \E_{R\leq n\leq x}^{\log}  \left| \E_{P\leq p < 2P} \E_{j \leq P} \lambda(n+j+p\ell_1)\cdots \lambda(n+j+p\ell_i) (p1_{p\mid n+j}-1)  \right| \ll \eps.$$
By the triangle inequality, the quantity inside the absolute values is bounded by $O(1)$.  Thus it will suffice to establish the probability bound
$$ \P_{R\leq n\leq x}^{\log} \left( \left| \E_{P\leq p < 2P} \E_{j \leq P} \lambda(n+j+p\ell_1)\cdots \lambda(n+j+p\ell_i) (p1_{p\mid n+j}-1)  \right| \geq \eps \right) \ll \eps$$
where $\P_{R \leq n \leq x}^{\log}(A) \coloneqq \E_{R \leq n \leq x}^{\log} 1_A(n)$ is the probability measure associated to the averaging operator $\E_{R \leq n \leq x}^{\log}$.

Observe that the numbers $n+j+p\ell_l$ that appear in this expression all lie in the interval $\{n+1,\dots,n+m\}$.  By \eqref{smo}, there are at most $x^{\eps^3}$ possible choices for the sign pattern $(\lambda(n+1),\dots,\lambda(n+m))$.  Thus, by the union bound, it will suffice to show that
\begin{equation}\label{pnx}
 \P_{R\leq n\leq x}^{\log} \left( \left| \E_{P\leq p < 2P} \E_{j \leq P} a_{j+p\ell_1} \cdots a_{j+p\ell_i} (p1_{p\mid n+j}-1)  \right| \geq \eps \right) \ll \eps x^{-\eps^3}
\end{equation}
for each choice of sign pattern $(a_1,\dots,a_m) \in \{-1,+1\}^m$.

Fix $a_1,\dots,a_m$.  Let $2r$ be the largest even integer such that $P^{2r} \leq x^{\eps^2}$.  From \eqref{m-bound} and the definition $P = m/(3k)$ we observe the estimates
\begin{equation}\label{r-bound}
\frac{1}{\eps} \ll r \asymp \eps^2 \frac{\log x}{\log P} \ll \eps^2 \frac{\log x}{\log\log x}.
\end{equation}
From Markov's inequality we may bound the left-hand side of \eqref{pnx} by
$$ \eps^{-2r} \E_{R\leq n\leq x}^{\log} \left| \E_{P\leq p < 2P} \E_{j \leq P} a_{j+p\ell_1} \cdots a_{j+p\ell_i} (p1_{p\mid n+j}-1)  \right|^{2r}$$
which by expanding out the $2r^{\mathrm{th}}$ power and applying the triangle inequality is bounded by
$$ \eps^{-2r} \E_{P \leq p_1,\dots,p_{2r} < 2P} \E_{j_1,\dots,j_{2r} \leq P} |\E_{R\leq n\leq x}^{\log} \xi_{p_1}(n+j_1) \dots \xi_{p_{2r}}(n+j_{2r})|$$
where $\xi_p(n) \coloneqq p 1_{p \mid n} - 1$.  From \eqref{r-bound} we have $\eps^{2r+1} \gg x^{-\eps^3}$, so it will thus suffice to establish the estimate
\begin{equation}\label{target}
\E_{P \leq p_1,\dots,p_{2r} < 2P} \E_{j_1,\dots,j_{2r} \leq P} |\E_{R\leq n\leq x}^{\log} \xi_{p_1}(n+j_1) \dots \xi_{p_{2r}}(n+j_{2r})|
\ll x^{-2\eps^3}.
\end{equation}
For any given $p_1,\dots,p_{2r},j_1,\dots,j_{2r}$, the function $n \mapsto \xi_{p_1}(n+j_1) \dots \xi_{p_{2r}}(n+j_{2r})$ is periodic of period $Q \coloneqq p_1 \dots p_{2r}$ and has magnitude at most $Q$.  
We have
$$ \E_{R\leq n\leq x}^{\log} \xi_{p_1}(n+j_1) \dots \xi_{p_{2r}}(n+j_{2r}) = \E_{R\leq n\leq x}^{\log} \xi_{p_1}(n+h+j_1) \dots \xi_{p_{2r}}(n+h+j_{2r}) + O\left( \frac{Q^2}{R\log x} \right)$$
for any $1 \leq h \leq Q$.  Averaging in $h$ and using the periodicity, we conclude that
$$ \E_{R\leq n\leq x}^{\log} \xi_{p_1}(n+j_1) \dots \xi_{p_{2r}}(n+j_{2r}) = \E_{n \in \Z/Q\Z} \xi_{p_1}(n+j_1) \dots \xi_{p_{2r}}(n+j_{2r}) + O\left( \frac{Q^2}{R\log x} \right)$$
where we view $\xi_{p_1},\dots,\xi_{p_{2r}}$ as functions on $\Z/Q\Z$ in the obvious fashion.  Since
$$ Q^2 \leq (2P)^{4r} \leq 2^{4r} x^{2\eps^2} \ll x^{3\eps^2}$$
(by \eqref{r-bound}) and $R = x^\eps$, we see that the $Q^2/(R\log x)$ error is negligible.  Thus it will suffice to show that
\begin{equation}\label{strong}
\E_{P \leq p_1,\dots,p_{2r} < 2P} \E_{j_1,\dots,j_{2r} \leq P} |\E_{n \in \Z/Q\Z} \xi_{p_1}(n+j_1) \dots \xi_{p_{2r}}(n+j_{2r})|
\ll x^{-2\eps^3}.
\end{equation}
If one of the primes $p_i$ is distinct from all the others, then the inner average $\E_{n \in \Z/Q\Z} \xi_{p_1}(n+j_1) \dots \xi_{p_{2r}}(n+j_{2r})$ vanishes from the Chinese remainder theorem, since $\xi_{p_i}(n+j_i)$ is periodic with mean zero with period $p_i$, and all other factors have period coprime to $p_i$.  Thus we may restrict attention to those tuples $(p_1,\dots,p_{2r})$ in which each prime $p_i$ appears at least twice, hence there are at most $r$ distinct primes in this tuple.  The number of such tuples can then be bounded crudely by $O( r^2 \pi_0(P))^r$, by first selecting $r$ primes in $[P,2P]$ (for which there are $O(\pi_0(P))^r$ choices), and then assigning each $p_1,\dots,p_{2r}$ to one of these primes (for which there are $r^{2r}$ choices).  Thus the proportion of such tuples amongst all primes $P \leq p_1,\dots,p_{2r} < 2P$ is $O( r^2 \pi_0(P)^{-1})^r$. If $(p_1,\dots,p_{2r})$ is such a tuple, then from the triangle inequality one has
\begin{align*}
&\E_{j_1,\dots,j_{2r} \leq P} |\E_{n \in \Z/Q\Z} \xi_{p_1}(n+j_1) \dots \xi_{p_{2r}}(n+j_{2r})|\\
&\leq \E_{n \in \Z/Q\Z} \E_{j_1,\dots,j_{2r} \leq P} |\xi_{p_1}(n+j_1)| \dots |\xi_{p_{2r}}(n+j_{2r})| \\
&= \E_{n \in \Z/Q\Z} \prod_{i=1}^{2r} \E_{j \leq P} |\xi_{p_i}(n+j)|\\
&\leq O(1)^r
\end{align*}
since $\E_{j \leq P} |\xi_{p_i}(n+j)| \ll 1$ for any $i$.  Thus we can bound the left-hand side of \eqref{strong} by $O( r^2 \pi_0(P)^{-1})^r$.  But from \eqref{r-bound}, \eqref{m-bound} we have $r^2 \pi_0(P)^{-1} \ll P^{-c}$ for some $c>0$ depending only on $\kappa$, hence by \eqref{r-bound} the left-hand side of \eqref{strong} is $O( x^{-c' \eps^2} )$ for some $c'>0$ depending on $\kappa$, and the claim follows.
This concludes the proof of Theorem~\ref{theo-signs}.

\subsection{Generalization to other multiplicative functions}

The above proof can be generalized to produce a result about patterns in more general multiplicative functions. 

\begin{theorem}\label{theo-multsigns}Let $g:\mathbb{N}\to \mu_{\ell}$ be a multiplicative function taking values in the roots of unity of order $\ell\geq 2$, and suppose that $\mathbb{D}(g^j,\chi;X)\xrightarrow{X\to \infty}\infty$ for all Dirichlet characters $\chi$ and for all $1\leq j\leq \ell-1$. Then the number
\begin{align*}
s_g(k)\coloneqq \{v\in \mu_{\ell}^k:\,\, v=(g(n+1),\ldots, g(n+k))\,\, \textnormal{for some}\,\, n\in \mathbb{N}\}
\end{align*}
of value patterns of $g$ of length $k$ satisfies $s_{g}(k)\gg_{A}k^{A}$.
\end{theorem}

We remark that a similar result holds (with essentially the same proof) for any $1$-bounded multiplicative function $g:\mathbb{N}\to \mathbb{C}$ such that $\inf_{|t|\leq X^{k+1}}\mathbb{D}(g^j,\chi(n)n^{it};X)\xrightarrow{X\to \infty}\infty$ for all $j\geq 1$. In this case, the ``sign patterns'' would be defined as occurrences of a pattern $(g(n+1),\ldots, g(n+k))\in I_1\times \cdots \times I_k$, where $I_i$ are arcs of the unit circle of the form $[e(m_i/\ell),e((m_i+1)/\ell)]$ with $0\leq m_i\leq \ell-1$. We leave the details of this generalization to the interested reader.

\begin{proof} (Sketch)
The proof follows along similar lines as that of Theorem~\ref{theo-signs}. We assume for the sake of contradiction that $s_g(m)\leq m^{A}$ for infinitely many $m$ and aim to deduce that
\begin{align}\label{eqqn1}
C\coloneqq \E_{n\leq x}^{\log} g^{a_1}(n+\ell_1)\cdots g^{a_j}(n+\ell_j)=o(1)
\end{align}
for any nonempty set $\{\ell_1,\ldots, \ell_j\}\subset \{1,2,\ldots, k\}$ with the $\ell_i$ distinct, and for any integers $a_1,\ldots, a_j\in [1,\ell-1]$. Once we have proved \eqref{eqqn1}, we use the expansion
\begin{align*}
1_{g(n)=e(a/\ell)}=\frac{1}{\ell}\sum_{j=0}^{\ell-1}g(n)^je(-\frac{aj}{\ell})    
\end{align*}
for the indicator functions of the level sets to obtain $s_g(k)=\ell^{k}$ for any $k$, which gives the desired contradiction.

The main difficulty\footnote{A much more minor difficulty is that $g$ is now only assumed to be multiplicative rather than completely multiplicative, so that the identity $g(n) = g(p)^{-1} g(pn)$ only holds when $n$ is not divisible by $p$. However, as we will be working with moderately large primes $p$, the contribution of those $n$ which are divisible by $p$ can easily be seen to be negligible.} is that the factor $(-1)^i$ that appeared in the proof of Theorem~\ref{theo-signs} must now be replaced by $g(p)^{-a_1-\dots-a_j}$.  One can still repeat the proof of Theorem~\ref{theo-signs} with obvious modifications down to \eqref{rhs}, where the right-hand side is now up to $O(\varepsilon)$ equal to
$$
\E_{\substack{1\leq b\leq W\\(b,W)=1}} \E_{P/W\leq d < 2P/W} g(d)^{-a_1-\dots-a_j} \Lambda_{W,b}(d) \E_{R\leq n\leq x}^{\log} g^{a_1}(n+(Wd+b)\ell_1)\cdots g^{a_j}(n+(Wd+b)\ell_j).$$
The weight $g(d)^{-a_1-\dots-a_j}$ now prevents one from applying the Gowers uniformity theory for the von Mangoldt function~\cite{green-tao}, \cite{gt-mobius}, \cite{gtz}.  However, the function $g(d)^{-a_1-\dots-a_j} \Lambda_{W,b}(d)$ is still dominated pointwise by $\Lambda_{W,b}(d)$, which is a pseudorandom majorant in the sense of~\cite{green-tao}.  One can then apply the generalized von Neumann theorem (essentially in the form of~\cite[Proposition 7.1]{green-tao}), and reduce matters to showing that
\begin{align*}
\E_{n\leq x}^{\log}\|g^j\|_{U^{k-1}[n,n+m]}=o(1)    
\end{align*}
whenever $1\leq j\leq \ell-1$ and $m\gg x^{\theta}$ for some $\theta>0$. This Gowers norm bound then follows from Theorem~\ref{mult-pret}, once we show that $M(f;x^{k+1},Q)\to \infty$ as $x\to \infty$ for any given $k$ and $Q$. By~\cite[Lemma 3.1]{klurman-mangerel-ANT} (which is a pretentious triangle inequality argument), and the fact that $\mathbb{D}(f,g;x)=\mathbb{D}(f,g;x^{k+1})+O_k(1)$, we have
\begin{align*}
M(g;x^{k+1},Q)\geq\inf_{\substack{\chi \mod q\\q\leq Q\\|t|\leq x^{k+1}}}\mathbb{D}(g\overline{\chi},n\mapsto n^{it};x)\geq \frac{1}{2kQ}\min\{(\log \log x)^{1/2},\mathbb{D}(g\overline{\chi},1;x)\}-O_{k,Q}(1), \end{align*}
and the right-hand side is tending to infinity with $x$ by assumption. This completes the proof.
\end{proof}

\subsection{Uniformity at very small scales}\label{uss}

We now give a proof of Proposition~\ref{entropy} that states that the estimate \eqref{log-liouville} at scale $H=(\log x)^{\eta}$ is enough to deduce the logarithmic Chowla conjecture (and hence in fact \eqref{log-liouville} for \emph{any} $H=H(X)$ tending to infinity, thanks to the results in~\cite{TaoEq}).

\begin{proof}[Proof of Proposition~\ref{entropy}] Let $k$ be a natural number, and let be  $h_1,\ldots, h_k$ given shifts. Let $x$ be large enough, and denote the correlation along these shifts by
\begin{align*}
C\coloneqq \E_{n\leq x}^{\log} \lambda(n+h_1)\cdots \lambda(n+h_k).
\end{align*}
For any fixed $\eps>0$, we wish to show that $|C|\ll \eps$. We begin by applying the entropy decrement argument in the slightly refined form given in~\cite[Theorem 3.1]{tt-chowla} (the original argument from~\cite{Tao} is able to locate a good scale on any interval $I$ with $\sum_{m\in I}\frac{1}{m\log m}\gg \eps^{-10}$, whereas the refined one is able to locate a good scale on any interval with $\sum_{m\in I}\frac{1}{m}\gg \eps^{-10}$).

By~\cite[Theorem 3.1]{tt-chowla}, we deduce that
\begin{align}\label{eq10a}
C=(-1)^k\E_{2^m\leq p\leq 2^{m+1}}\E_{n\leq x}^{\log}\lambda(n+ph_1)\cdots \lambda(n+ph_k)+O(\eps)    
\end{align}
for all $m \leq \log\log X$ outside of an exceptional set $\mathcal{M}\subset [1,\log \log x]\cap \mathbb{N}$ with
\begin{align*}
\sum_{m\in \mathcal{M}}\frac{1}{m}\ll \eps^{-3}.
\end{align*}
In particular, we can locate some $m$ with the property \eqref{eq10a} belonging to the range $m\in [\eps'\log \log x,\frac{1}{10}\log \log x]$ with $\eps'\coloneqq \exp(-\eps^{-10})$. Let $P=2^m\geq (\log x)^{\eps'/2}$, where $m$ has this value. Then, by introducing the von Mangoldt weight, we have
\begin{align*}
 C=(-1)^k\E_{P\leq d\leq 2P}\Lambda(d)\E_{n\leq x}^{\log}\lambda(n+dh_1)\cdots \lambda(n+dh_k)+O(\eps)    
\end{align*}
As in the proof of Theorem~\ref{theo-signs}, we may split $d$ into residue classes $\pmod W$ with $W=\prod_{p\leq w}p$ and $w=w(x)$ tending to infinity slowly enough, and then apply the Gowers uniformity of the $W$-tricked von Mangoldt function and the generalized von Neumann theorem (as in~\cite[Section 5]{tt-chowla}) to conclude that
\begin{align*}
C=(-1)^k\frac{W}{\phi(W)}\E_{P\leq d\leq 2P}1_{(d,W)=1}\E_{n\leq x}^{\log} \lambda(n+dh_1)\cdots \lambda(n+dh_k)+O(\eps).  
\end{align*}
Arguing as in the proof of \eqref{eq5}, we have
$$ C \ll O_W( \E_{n\leq x}^{\log}\kappa(\|\lambda\|_{U^{k}[n,n+3kP]} )) + \eps.$$
Since $P\geq (\log x)^{\eta}$ where $\eta=\eps'/2$, the hypothesis of the theorem will then give $C = O(\eps)$ if we assume $x$ sufficiently large depending on $w$.
\end{proof}

\section{Reducing the length of the intervals}\label{sec: lowering}

In this section we indicate the changes needed to the proof of Theorem~\ref{lam-poly} to obtain Theorem~\ref{lower}. Up to Proposition~\ref{prop41-analog} (corresponding to the work up to~\cite[Section 4]{mrt-fourier}), everything works for smaller $H$ as well, except in the statement of Proposition~\ref{prop32-analog} the range for $P', P''$ is now $[H^{\eps^2/2}, H^\eps]$.

To proceed, we will need the following variant of Lemma~\ref{lem26-analog} in which the implied constants do not depend on the number of primes in the product. Crude bounds suffice here and stronger bounds would not be useful as we in any case lose factors like $\ell!$ in our arguments.

\begin{lemma}[Counting nearby products of primes]\label{le:countNPLowH} Let $m, \ell, q \in \mathbb{N}$ and $P', N \geq 3$. Then the number of $2\ell$-tuples $(p'_{1,1},\dots,p'_{1,\ell},p'_{2,1},\dots,p'_{2,\ell})$ of primes in $[P', 2P']$ obeying the conditions
\[ 
\left | \prod_{j=1}^\ell p'_{2,j} - \prod_{j=1}^\ell p'_{1,j} \right | \leq C\cdot \frac{(2P')^\ell}{N}
\]
and
\[
\prod_{j=1}^{\ell} (p'_{2,j})^m = \prod_{j=1}^{\ell} (p'_{1,j})^m \mod q
\]
for some $C \geq 1$ is bounded by
$$ \ll C \ell !^2 (2P')^{\ell}m^{\omega(q)} \left(\frac{(2P')^\ell}{Nq} + 1 \right).$$
\end{lemma}

\begin{proof}
Since every integer has at most $\ell!$ representations as a product of $\ell$ primes, the number of prime tuples we need to count is at most $\ell !^2$ times the number of integers $n_1, n_2 \leq (2P')^\ell$ for which
\[
|n_1 - n_2| \leq C\cdot \frac{(2P')^\ell}{N} \qquad \text{and} \qquad n_1^m = n_2^m \mod{q}.
\]
The claim follows by noticing that there are $(2P')^\ell$ choices for $n_1$, and after fixing it, there are at most $m^{\omega(q)}$ choices for $n_2 \mod{q}$.
\end{proof}

Let us now get back to Proposition~\ref{prop41-analog} corresponding to~\cite[Proposition 4.1]{mrt-fourier}. In our setting we obtain the following variant, where we for simplicity restrict to the case $\ell_1 = \ell_2 = \ell$ and a single quadruple $\vec a$ corresponding to each $e \in \mathcal{Q}$ as this is sufficient for the polynomial phase case.

\begin{proposition}[Local structure of $\phi''$]\label{prop41-analogLowerH} Let the hypotheses be as in Theorem~\ref{lower}, and let $\eps,X,P',P'',{\mathcal I}'', \phi''_{I''}, {\mathcal Q}$ be as in Proposition~\ref{prop32-analog} (except now $P', P'' \in [H^{\eps^2/2}, H^\eps]$). 
Let $\ell$ be an even integer such that
\begin{equation}
N^2 d^{10} \leq d^{\ell} = O(N^{O(1)}).
\end{equation}
We allow implied constants to depend on $\eps, \eta$ and $\theta$. There exists a constant $c = c(\eps, \eta, \theta)$ such that, for a subset ${\mathcal Q}'$ of the quadruples $e = (I''_1, I''_2, p'_1, p'_2)$ in ${\mathcal Q}$ of cardinality $\gg c^\ell dN$, one can find a quadruple $\vec a = (a_1,a_2,b_1,b_2)$ of natural numbers, and a collection ${\mathcal P}_{e}$ of primes in $[P''/2,P'']$ with $|{\mathcal P}_e| \gg (\log X)^{-10\ell} \pi_0(P')$, with the following properties:
\begin{itemize}
\item[(i)] One has
\begin{equation}
\frac{1}{p'_2} \circ \phi''_{I''_1} \sim_{\frac{1}{\prod {\mathcal P}_{e}}} \frac{1}{p'_1} \circ \phi''_{I''_2}.
\end{equation}
\item[(ii)]  For $i=1,2$, $a_i,b_i$ are products of $\ell_i$ primes in $[P',2P']$; in particular
\begin{equation}
(P')^{\ell} \leq a_i, b_i \leq (2 P')^\ell. 
\end{equation}
Furthermore we have
\begin{equation}
0 \neq a_i - b_i \ll \frac{C^\ell}{N} a_i,
\end{equation}
where $C$ is an absolute constant.
\item[(iii)]  For $i=1,2$, we have the approximate dilation invariance
\begin{equation}
\frac{1}{a_i} \circ \phi''_{I''_i} \sim_{\frac{1}{\prod {\mathcal P}_{e,\vec a}}} \frac{1}{b_i} \circ \phi''_{I''_i}.
\end{equation}
Here (abusing the notation) the implied constants depend linearly on $\ell$.
\end{itemize}
\end{proposition}

\begin{proof}[Sketch of proof]
The proof is very similar to the proofs of Propositions~\ref{prop41-analog} and~\cite[Proposition 4.1]{mrt-fourier}: 
One makes two cycles of length $\ell$ joined by a "middle edge". The main difference is that now $\ell \asymp (\log x)^{1-\theta}$, so $\ell$ is no longer a constant. 

Since the number of edges in the graph is $\gg \frac{X}{H} P'^2/(\log P')^2$, the number of such constellations gets reduced by a factor $c'^\ell$ (with certain constant $c' \in (0, 1)$). Hence the Cauchy-Schwarz argument at the end naturally only gives us $\gg c'^\ell X/H \cdot \pi_0(P')^2$ middle edges.\footnote{This might be fixable through arguing more carefully removing some edges before running the argument but this would be of no importance.} Since $P'$ is larger than $(c \ell \log P')^{O(\ell)}$, Lemma~\ref{le:countNPLowH} is sufficient to show that degenerate cases involving repeating primes or products are negligible as before.

Since the constellation involves $2\ell +1$ edges, the intersection 
\[
\mathcal{P}(\overrightarrow{I''}) := \mathcal{P}(\{I''_{0, 1}, I''_{0, 2}\}) \cap \bigcap_{j = 1}^k \bigcap_{i = 1, 2} \mathcal{P}(\{I''_{j, i}, I''_{j+1, i}\})
\]
that appears in~\cite[(52)]{mrt-fourier} is now expected to be only of size $c^\ell \pi_0(P')$ for some constant $c > 0$, so $\delta$ in~\cite[(52))]{mrt-fourier} cannot anymore be taken to be a constant but can be at most $c^\ell$. In fact to compensate for losses in Lemma~\ref{le:countNPLowH} we choose $\delta$ in~\cite[(52))]{mrt-fourier} to be $(\log X)^{-10\ell}$. Then in the argument below~\cite[(52))]{mrt-fourier} the number of candidate tuples is at most $\ell!^4 P'^{4\ell+1}/N$ and so the expected number of good tuples obeying~\cite[(52))]{mrt-fourier} is $\ll (\log X)^{-10\ell} \ell!^4 P'^{4\ell+1}/N \ll (\log X)^{-\ell} d^{2\ell+1}/N$ whereas with probability $\gg 1$, there are $\gg c^\ell d^{2\ell+1}/N$ non-degenerate good tuples. Hence one can indeed find a deterministic choice of $\mathbf{p}$ such that there are $\gg c^\ell d^{2\ell+1}/N$ very good tuples, i.e. tuples for which
\[
\# \mathcal{P}(\overrightarrow{I''}) > (\log X)^{-10\ell} \pi_0(P')
\]
as desired.
\end{proof}

Lowering $H$ does not affect solving the approximate dilation invariance in Proposition~\ref{sold}, except that the bounds for $T$ and the smoothness of $\eps^{(j)}_i(t)$ get worsened by $C^\ell$ for a constant $C$. Since $\mathcal{P}(\overrightarrow{I''})$ now of size $\gg (\log X)^{-10\ell} \pi_0(P')$, we now need to take $K \gg (\log X)^{10 \ell}$ in Proposition~\ref{prop53-analog}, so in Corollary~\ref{cor54-analog} we now have $\# \mathcal{F}(I'') \ll (\log X)^{10\ell}$. Proposition~\ref{prop55-analog} works without changes but now it provides only $\gg c^\ell X/H \pi_0(P')^2$ pairs $(I_1'', I_2'')$.

To proceed, we need an adequate version of the mixing lemma:

\begin{lemma}[Mixing lemma]\label{le:mixlemShH} Let $X, V \geq 3, 2\leq P \leq H$.  Let $\mathcal{A}_1, \mathcal{A}_2$ be two $(X, H)$-families of intervals. Write
\[
\mathcal{V} = \left\{\xi \in [-X/H, X/H] \colon \big|\sum_{P\leq p\leq 2P} p^{2 \pi i \xi}\big| \geq PV^{-1} \right\}.
\]

Then the number of quadruplets $(J_1, J_2, p_1, p_2)$ with $J_1 \in \mathcal{A}_1, J_2 \in \mathcal{A}_2$, $p_1, p_2$ primes in $[P, 2P]$, and $I_1$ lying within $100 H$ of  $\frac{p_2}{p_1} I_2$ is
\begin{equation}\label{mixbound}
 \ll |\mathcal{V}| (\# \mathcal{A}_1) (\# \mathcal{A}_2) \frac{H}{X} \left (\frac{P}{\log P} \right )^2 +(\# \mathcal{A}_1)^{1/2} (\# \mathcal{A}_2)^{1/2} P^2 V^{-2}.
\end{equation}
\end{lemma}

\begin{proof} As in~\cite[Proof of Lemma 5.1]{mrt-fourier}, the number of quadruplets in question is bounded by
\begin{equation}
\label{eq:mixintLowH}
\ll \frac{H}{X} \int_{|\xi| \leq \frac{X}{H}} |S_1(\xi)| |S_2(\xi)| |T(\xi)|^2\ d\xi 
\end{equation}
where
$$ S_i(\xi) := \sum_{I \in \mathcal{A}_i} e( \xi \log x_{I} ) $$
for $i=1,2$ and
\begin{align}\label{eq: Tchi}
 T(\xi) := \sum_{P\leq p\leq 2P} p^{2\pi i \xi}.
\end{align}

Splitting the integral in~\eqref{eq:mixintLowH} according to whether $\xi \in \mathcal{V}$, we obtain that~\eqref{eq:mixintLowH} is at most
\[
\begin{split}
&\frac{H}{X} |\mathcal{V}| \sup_{|\xi| \in \mathcal{V}} |S_1(\xi) S_2(\xi) T(\xi)^2| + \frac{H}{X} P^2V^{-2} \int_{|\xi| \leq \frac{X}{H}} |S_1(\xi)| |S_2(\xi)|\ d\xi \\
&\ll \frac{H}{X} |\mathcal{V}| (\# \mathcal{A}_1) (\# \mathcal{A}_2) \left (\frac{P}{\log P} \right )^2 + \frac{H}{X} P^2V^{-2} \left(\int_{|\xi| \leq \frac{X}{H}} |S_1(\xi)|^2 \ d\xi \int_{|\xi| \leq \frac{X}{H}} |S_2(\xi)|^2\ d\xi\right)^{1/2}.
\end{split}
\]
From the large sieve inequality (see e.g.~\cite[Lemma 2.3]{mrt-fourier}) we have
\begin{equation}
\label{eq:mixLS}
 \int_{|\xi| \leq \frac{X}{H}} |S_i(\xi)|^2 \ll \# {\mathcal A}_i \frac{X}{H},
 \end{equation}
and the claim follows.
\end{proof}

Note that the size of $\mathcal{V}$ above is at most twice the size of the maximal one-spaced subset of $\mathcal{V}$ (meaning a set where any two points are at least one apart). The needed bound for $|\mathcal{V}|$ in our situation is provided by the following lemma. The requirement $\theta > 5/8$ comes from it as for smaller $\theta$ we do not know how to obtain $|\mathcal{V}| = P^{o(1)}$.

\begin{lemma}\label{le: 5/8}
Let $\theta \in (5/8, 1)$ be fixed, $H = \exp((\log X)^\theta)$ and $P = \exp(\varepsilon (\log X)^\theta)$ for some $\eps > 0$, and let $V = (\log X)^{100\ell}$, where $\ell \asymp (\log X)^{1-\theta}$. Let $\mathcal{U}$ be a set of one-spaced points $\xi \in [-X/H, X/H]$ for which 
\[
\big|\sum_{p \sim P} p^{2 \pi i \xi}\big| \geq PV^{-1}.
\]
Then, for some $\eps' > 0$, we have
\[
\# \mathcal{U} \ll \exp((\log X)^{\theta-\eps'}) = P^{o(1)}.
\]
\end{lemma}

\begin{remark}\label{rem: largevalues}
From the proof of Lemma~\ref{le: 5/8}, it will be clear that the larger $\theta>5/8$ is, the better the bound we can obtain on $\#\mathcal{U}$. In fact, for $\theta=2/3+\varepsilon$ the Vinogradov--Korobov bound  (see~\cite[Lemma 2]{mr-p}) directly gives $\mathcal{U}\subset [-V^2,V^2]$, so that $\# \mathcal{U}\ll V^2\ll \exp((\log X)^{1-\theta+\varepsilon^2})$, say. Nevertheless, here the main interest is in the smallest value of $\theta$ for which $\#\mathcal{U}\ll \exp((\log X)^{\theta-\varepsilon'})$ holds, so this aspect is not optimized.  
\end{remark}

\begin{proof}
Let $T(\chi)$ be as in \eqref{eq: Tchi}. We apply~\cite[Lemma 4.4]{mrPart2}, which is a variant of the  Hal\'asz--Montgomery estimate that uses Vinogradov's bound on $\sum_{P\leq n\leq 2P}n^{it}$ as an input (see also Lemma~\ref{lem:primeshalasz} below with $q = 1$). This gives that uniformly for $\eta \in (0, 1)$ and integers $k \geq 0$ we have
\begin{equation} \label{eq:Hmtapp}
\# \mathcal{U} \cdot \Big ( \frac{P}{V} \Big )^{2k} \ll \sum_{t \in \mathcal{U}} |T(\xi)|^{2k} \ll \Big ( (2P)^{k} + \# \mathcal{U} \cdot X^{5\eta^{3/2}} (\log X)^{2/3} \cdot (2P)^{k (1 - \eta / 4)} \Big ) k! \cdot (2P)^{k}.
\end{equation}
This means that we have the bound
\begin{align*}
\# \mathcal{U} \ll (4kV)^{2k}
\end{align*}
whenever $X^{5\eta^{3/2}}(\log X)^{3/2}k^{2k}P^{k(2-\eta/4)}=o((P/V)^{2k})$. The latter holds whenever
\[
X^{5\eta^{3/2}} k^{2k} \cdot \exp(k (\log X)^{1-\theta}(\log \log X)^2) = o (\exp(\tfrac{\varepsilon}{4} \eta k (\log X)^\theta)),
\]
which in turn follows from
\[ 
5\eta^{3/2} \log X + 2k \log k + k (\log X)^{1-\theta}(\log \log X)^2 <\tfrac{\varepsilon}{5} \eta k (\log X)^\theta.
\]
This holds (assuming already $k = (\log X)^{O(1)}$ and letting $\delta$ be a small positive constant) if 
\[
\begin{cases}
k \geq \eta^{1/2} (\log X)^{1-\theta + \delta} & \\
\eta \geq (\log X)^{-\theta + \delta} &\\ 
\eta \geq (\log X)^{1-2\theta+\delta}.
\end{cases}
\]
For $\theta < 1$, the third condition is more demanding than the second and thus we can set $\eta = (\log X)^{1-2\theta+\delta}$ and $k = (\log X)^{3/2 - 2\theta + 2\delta}$. With these choices the first term dominates in~\eqref{eq:Hmtapp} and we obtain the upper bound
\[
\# \mathcal{U} \ll (4kV)^{2k} \ll (\log X)^{300 \ell k} \ll \exp((\log X)^{5/2-3\theta+3\delta})
\] 
The claim follows as $5/2-3\theta < \theta$ since $\theta > 5/8$.
\end{proof}

Now this leads to approximate ergodicity~\cite[Corollary 5.2]{mrt-fourier} except that now we have either 
\[
\frac{MK^3}{\delta} \gg (\log X)^{100\ell}
\]
or a collection $\mathcal{T}$ as in~\cite[Corollary 5.2]{mrt-fourier} but with
\begin{align}\label{eq: T}
\# \mathcal{T} \gg \exp(-(\log X)^{\theta-\varepsilon}) \frac{\delta}{MK^3} \frac{X}{H}.
\end{align}
We can apply this with $\delta = c^\ell$, $K \asymp (\log X)^{10\ell}$, $M= 100$ and $r = 1/10$ to get conclusions between Proposition~\ref{prop55-analog} and Lemma~\ref{lem26-analog}, except that now have the weaker lower bound $\# \mathcal{T} \gg \exp(-(\log X)^{\theta-\varepsilon}) X/H$. 

As for the analogue of Proposition~\ref{prop56-analog}, we can use the same argument as in its proof to obtain upper and lower bounds for the number of certain tuples $(Q_0, \dotsc, Q_{\ell -1}) \in \mathcal{T}^\ell$: The lower bound we get is $\gg c^\ell d^\ell$ (with $d:=(P'/\log P')^2$) and the upper bound (from Lemma~\ref{le:countNPLowH}) is
\[
\ll \ell!^2 (2P')^\ell k^{\omega(q)} \left(\frac{(2P')^\ell}{q_0^{1/k} N} + 1\right)
\]
Combining the lower and upper bounds, we obtain $q_0 \ll (\log X)^{O(\ell)}$.

Now, repeating the arguments after Proposition~\ref{prop56-analog}, we see that there are at least $\gg \exp(-(\log X)^{\theta-\varepsilon}) X/(P'P'')$ integers $X/(2P'P'')\leq x \leq X/(P' P'')$ for which
\begin{align}\label{eq: Hast}
\big|\sum_{n \in [x, x+H^\ast]} f(n)n^{-iT} e(-\gamma(n))\big| \gg H^\ast 
\end{align}
with $H^{\ast}:=C^{-\ell}H/(P'P'')$ and $\gamma(t) = \sum_{j = 0}^k c_j \binom{t/q_0}{j}$, where $c_j$ are integers. 

Now we will obtain a contradiction as in Section~\ref{sec: polyphases}, except due to worse bounds for $\mathcal{T}^\ast$ and $q_0$ we need to use results from~\cite{mrPart2} where one obtains a polynomial saving in the exceptional set for averages of multiplicative functions in short intervals (in the special case $f = \lambda$ and $\theta = 2/3+\eps$ arguments of~\cite{mr} actually suffice --- see Remark~\ref{rem:specShorter} below). Also since $q_0$ is not bounded, we need to treat the $q$-aspect non-trivially. 

As in~\cite{mr, mrPart2} we first restrict $n$ to a set of numbers with factors of convenient sizes. For this, let $\delta$ be small in terms of the implied constant above and define $\mathcal{S}$ as in~\cite[Proof of Theorem 1.7 in Section 11]{mrPart2}, i.e. choose in~\cite[Section 9]{mrPart2} the parameters $\eta =1/150, \nu_1 = \delta^2/4000, \nu_2 = 1/10, Q_1 = H^\ast$ and $P_1 = Q_1^{\delta/4}$, so that $J = 1$, $P_2 = X^{\nu_1}, Q_2 = P_3 = X^{\sqrt{\nu_1 \nu_2}}$ and $Q_3 = X^{\nu_2}$ and $\mathcal{S}$ consists of numbers with a prime factor on each interval $(P_j, Q_j]$ with $j=1, 2, 3$.

Using the linear sieve (cf.~\cite[Proof of Theorem 1.7 in Section 11]{mrPart2}), we see that $n \not \in \mathcal{S}$ make a negligible contribution of 
\begin{align*}
H^{\ast}\sum_{1\leq i\leq 3}\frac{\log P_i}{\log Q_i}\ll \delta H^{\ast},
\end{align*}
to \eqref{eq: Hast} and so we have $\gg \exp(-(\log X)^{\theta-\varepsilon}) X/(P'P'')$ integers $X/(2P'P'')\leq x \leq X/(P' P'')$ for which
\[
\big|\sum_{\substack{n \in [x, x+H^\ast] \\ n \in \mathcal{S}}} f(n)n^{-iT} e(-\gamma(n))\big| \gg H^\ast.
\]

Splitting into residue classes $a \mod{q_0}$ and then splitting according to $q_2 = \gcd(a, q_0)$, we see that 
\[
\sum_{q_2:\,\, q_0 = q_1 q_2} \big|\sum_{\substack{b \mod{q_1} \\ (b, q_1) = 1}}  e(-\gamma(bq_2))\sum_{\substack{n \in \mathcal{S} \\ n \in [x/q_2, (x+H^\ast)/q_2] \\ n = b \mod{q_1}}} f(n)n^{-iT}\big| \gg H^\ast
\]
for $\gg \exp(-(\log X)^{\theta-\varepsilon}) X/(P'P'')$ integers $X/(2P'P'')\leq x \leq X/(P' P'')$. This implies that for some choice of $q_0 = q_1 q_2$, we have
\[
\big|\sum_{\substack{b \mod{q_1} \\ (b, q_1) = 1}}  e(-\gamma(bq_2))\sum_{\substack{n \in \mathcal{S} \\ n \in [x, x+H^\ast/q_2] \\ n = b \mod{q_1}}} f(n)n^{-iT}\big| \gg \frac{\phi(q_1)}{q_1q_2} H^\ast
\]
for $\gg \exp(-(\log X)^{\theta-\varepsilon}) X/(q_2 P'P'')$ integers $X/(2q_2P'P'')\leq x \leq X/(q_2P' P'')$.
Moving into characters, the left-hand side is at most
\[
\frac{1}{\phi(q_1)} \sum_{\chi \mod{q_1}} \big|\sum_{\substack{b \mod{q_1} \\ (b, q_1) = 1}}  e(-\gamma(bq_2)) \overline{\chi(b)}\big| \cdot \big| \sum_{\substack{n \in \mathcal{S} \\ n \in [x, x+H^\ast/q_2]}} f(n) \chi(n) n^{-iT}\big|.
\]
Recall that $\gamma$ is a polynomial phase of degree $k$. By~\cite[Corollary 1.1]{CZ99} and the Chinese reminder theorem we have, for every $\chi$,
\begin{equation}
\big|\sum_{\substack{b \mod{q_1} \\ (a, q_1) = 1}}  e(-\gamma(bq_2)) \overline{\chi(b)}\big| = O(q_1^{1-1/(k+1)}), 
\end{equation}
so that
\begin{equation}
\label{eq:chilowbound}
\sum_{\chi \mod{q_1}} \big| \sum_{\substack{n \in \mathcal{S} \\ n \in [x, x+H^\ast/q_2]}} f(n) \chi(n) n^{-iT}\big| \gg q_1^{1/(k+2)} H^\ast/q_2
\end{equation}
for $\gg \exp(-(\log X)^{\theta-\varepsilon}) X/(q_2 P'P'')$ integers $X/(2q_2 P' P'')\leq x \leq X/(q_2 P' P'')$. 

Now, if $q_1 \leq Q$ for a constant $Q \ll_{k, \eta, \theta, \rho} 1$ to be determined later, we have, for some $\chi \pmod{q_1}$,
\[
\sum_{\chi \mod{q_1}} \big| \sum_{\substack{n \in \mathcal{S} \\ n \in [x, x+H^\ast/q_2]}} f(n) \chi(n) n^{-iT}\big| \gg_{k, \eta, \theta, \rho} H^\ast/q_2
\]
for $\gg \exp(-(\log X)^{\theta-\varepsilon}) X/(q_2 P'P'')$ integers $X/(2q_2 P' P'')\leq x \leq X/(q_2 P' P'')$. 
By~\cite[Theorem 9.2(i)]{mrPart2} this implies that 
\[
\big| \sum_{\substack{n \in \mathcal{S} \\ X < n \leq 2X}} f(n) \chi(n) n^{-iT+it_0}\big| \gg_{k, \eta, \theta, \rho} X,
\]
for some $|t_0| \leq X$, which in turn by inclusion-exclusion and Hal\'asz's theorem implies \eqref{eq:lowerMClaim} since $|T| \leq C^\ell (X/H)^{k+1} \leq X^{k+1}/H^{k+1-\rho/2}$.

Let us now turn to the case $q_1 \geq Q$. The proof of~\cite[Proposition 8.3]{mrPart2} (taking $\mathcal{V}_1 = \emptyset$ in the proof of~\cite[Proposition 8.3]{mrPart2} and bounding $R_C(1+it)$ trivially) gives
\begin{equation}
\label{eq:sumProp83deco}
\begin{split}
&\frac{1}{H^\ast/q_2} \sum_{\substack{n \in \mathcal{S} \\ n \in [x, x+H^\ast/q_2]}} f(n) \chi(n) n^{-iT} = A(x, H^\ast/q_2, \mathcal{U}) + O\left(\frac{1}{H^\ast/q_2}\right) \\
 & \qquad + O\Bigl(\Bigl(\sum_{\substack{A = 2^j \\ P_3/2 \leq A \leq Q_3}} \sum_{t \in \mathcal{W}^\ast(\chi)} |Q_{3,A}(\chi, 1 + it)|^2 \sum_{\substack{B = 2^j \\ P_2/2 \leq B \leq Q_2}} \sum_{t \in \mathcal{W}^\ast(\chi)} |Q_{2,B}(\chi, 1 + it)|^2 \Bigr)^{1/2}\Bigr),
\end{split}
\end{equation}
where
\[
\mathcal{W}^\ast(\chi) \subset \{|t| \leq X \colon \max_B |Q_{2,B}(\chi, 1+it)| \geq X^{-\nu_1^3/320} \}
\]
is one-spaced,
\[
Q_{j,D}(\chi, s) := \sum_{\substack{D < p \leq 2D \\ P_{j} < p \leq Q_{j}}} \frac{\chi(p)}{p^s},
\]
and $A(x, H^\ast/q_2, \mathcal{U})$ satisfies~\cite[(46)]{mrPart2}. 

As in~\cite[Proof of Theorem 9.2(ii)]{mrPart2} with same choices of $\mathcal{U}$ and $d_n$, we have $|A(x, H^\ast/q_2, \mathcal{U})|$ $\ll {H^\ast}^{-\delta/5000}$ except for $\ll X{H^\ast}^{-\delta/5000}$ values $X/(2q_1P'P'')\leq x \leq X/(q_1P' P'')$. Summing over $\chi \mod{q_1}$ and taking the union bound, the contribution from $A(x, H^\ast/q_2,  \mathcal{U})$ is acceptable.

Given all this, \eqref{eq:chilowbound} implies that
\begin{equation}
\label{eq:QAQBbound}
\sum_{\chi \mod{q_1}} \sum_A \sum_{t \in \mathcal{W}^\ast(\chi)} |Q_{3,A}(\chi, 1 + it)|^2 \sum_{\chi \mod{q_1}} \sum_B \sum_{t \in \mathcal{W}^\ast(\chi)} |Q_{2,B}(\chi, 1 + it)|^2 \gg Q^{2/(k+2)}
\end{equation}
In~\cite{mrPart2} this sort of term with $q_1 = 1$ is dealt with using~\cite[Lemma 4.4]{mrPart2} which is a large values result of Hal\'asz --Montgomery type that uses Ford's bound (see~\cite[Theorem 1]{FordZeta})
\[
|\zeta(\sigma + it)| \ll 1 + |t|^{\tfrac{9}{2} (1 - \sigma)^{3/2}} (\log (|t|+2))^{2/3} \quad \text{for $1/2 \leq \sigma \leq 1$}.
\]
for $\zeta(s)$. As pointed out by Ford, $L(s, \chi) = q^{-s} \sum_{m = 1}^q \chi(m) \zeta(s, m/q)$, where $\zeta(s, u) = \sum_{n=0}^\infty (n+u)^{-s}$ is the Hurwitz zeta function, so that~\cite[Theorem 1]{FordZeta} also gives 
\[
|L(\sigma + it, \chi)| \ll q^{1-\sigma} |t|^{\tfrac{9}{2} (1 - \sigma)^{3/2}} (\log (|t|+2))^{2/3} + \frac{q^{1-\sigma}}{1-\sigma} \quad \text{for $1/2 \leq \sigma < 1$}.
\]
Using this in the proof of~\cite[Lemma 4.4]{mrPart2}, we get the following variant.
\begin{lemma} \label{lem:primeshalasz}
Let $T \geq 3$ $q \geq 1$ and let $\mathcal{T}$ be a set of pairs $(\chi, t)$, where $\chi$ is a Dirichlet character $\mod q$ and $t \in [-T, T]$ such that if $(\chi, t_1), (\chi, t_2) \in \mathcal{T}$, then $|t_1-t_2| \geq 1$. Let $P(s, \chi) = \sum_{N < p \leq 2N} a(p)\chi(p) p^{it}$ be a Dirichlet polynomial of length $N \leq T^2$ whose coefficients are supported on primes. Then, for any $\varepsilon', \eta \in (0, 1/2)$,
\[
\sum_{(\chi, t) \in \mathcal{T}} |P(\chi, it)|^2 \ll_{\varepsilon'} \Big ( \frac{N}{\log N} + |\mathcal{T}| \cdot (q^\eta T^{\frac{9}{2} \eta^{3/2}} (\log T)^{2/3} + q^\eta/\eta) \cdot N^{1 - \eta(1-\varepsilon')} \Big ) \sum_{N < p \leq 2N} |a(p)|^2.
\]
\end{lemma}
Using this and arguing as in~\cite[Proof of Proposition 8.3]{mrPart2}, we obtain
\[
\sum_{\chi \mod{q_1}} \sum_A \sum_{t \in \mathcal{W}^\ast(\chi)} |Q_{3,A}(\chi, 1 + it)|^2 \sum_{\chi \mod{q_1}} \sum_B \sum_{t \in \mathcal{W}^\ast(\chi)} |Q_{2,B}(\chi, 1 + it)|^2 \ll 1
\]
which contradicts~\eqref{eq:QAQBbound} once the constant $Q$ is large enough. Hence Theorem~\ref{lower} follows.

\begin{remark}
\label{rem:specShorter}
We remark that the special case $f = \lambda$  of Theorem~\ref{lower} with the weaker value $\theta=2/3+\varepsilon$  can be proved more simply by relying only on~\cite{mr-p} as follows. Firstly note that, by Remark~\ref{rem: largevalues}, we can replace $\exp(-(\log X)^{\theta-\varepsilon})$ with $\exp(-(\log X)^{1-\theta-\varepsilon^2})$ in \eqref{eq: T} and on later occurrences. Note also that in this case $q_1,q_2 \ll \exp((\log X)^{1/3-\varepsilon/2})$.

We must then show that \eqref{eq:chilowbound} with $f = \lambda$ cannot hold for $\gg \exp(-(\log X)^{1-\theta-\varepsilon^2}) X/(q_2 P'P'')$ integers $X/(2q_2 P' P'')\leq x \leq X/(q_2 P' P'')$. We have the Vinogradov--Korobov zero-free region for $L$-functions of the form
\begin{align}\label{eq: VinKor}
L(s,\chi)\neq 0,\quad \sigma\geq 1-\frac{c_0}{\log q_1+(\log(|t|+3))^{2/3}(\log \log(|t|+3))^{1/3}}
\end{align}
for all $\chi\pmod {q_1}$, apart from possibly one real zero corresponding to one real character. In case an exceptional character exists, $q_1 \gg_A (\log X)^A$. The contribution of an exceptional character to \eqref{eq:chilowbound} is trivially negligible, so we may assume that in \eqref{eq:chilowbound} we only sum over characters $\chi\pmod{q_1}$ satisfying \eqref{eq: VinKor}. Moreover, we may assume that the set $\mathcal{S}$ in \eqref{eq:chilowbound} is instead simply defined as the set of $n$ having a prime factor from $[P,Q]$, with $Q=H^{*}$, $P=Q^{\delta/4}$. We again claim that \eqref{eq:chilowbound} fails, which will then provide the desired contradiction.

To show this claim, we apply the proof method of~\cite{mr-p} to the multiplicative function $\lambda(n)\chi(n)n^{-iT}$, summed over $n\in \mathcal{S}$. Reducing matters from short sums to Dirichlet polynomials by Parseval-type arguments, as in~\cite[Section 4]{mr-p}, we can reduce the claim to 
\begin{align}\label{eq: PQ}
\int_{[-T_1,T_1]\setminus[T-T_0,T+T_0]}|P(1+it)|^2|Q(1+it)|^2\, dt\ll \exp(-(\log X)^{1/3-\varepsilon/10}),
\end{align}
where $T_0=\exp((\log X)^{1/3-\varepsilon/10})$ and $T_1=X\exp((\log X)^{1/3-\varepsilon/10})/H^{*}$, and we have $P(s)=\sum_{p\in I}\chi(p)p^{s-iT}$ for some interval $I\subset [P,Q]$ and $Q(s)=\sum_{X/Q\leq n\leq X}a_nn^{s}$ for some $|a_n|\leq 1$. As in~\cite{mr-p}, applying the pointwise Vinogradov--Korobov bound to  $P(s)$ and the mean value theorem to $Q(s)$, \eqref{eq: PQ} follows.
\end{remark}

\section{Polynomial averages of the Liouville function}\label{sec:pattern}

In this section, we prove Theorems~\ref{poly1} and~\ref{poly2}. Note that Corollary~\ref{cor_chowla} is a special case\footnote{\label{foot:polypattern}This special case could in fact be proved more directly without considering polynomial progressions, instead combining the generalized von Neumann theorem with Corollary~\ref{cor: mult-pret}.} of Theorem~\ref{poly1} where we take $P_i(m)=a_im$.

\begin{proof}[Proof of Theorems~\ref{poly1} and~\ref{poly2}]

We borrow notation from~\cite{tao-ziegler2}. Note that the claim of Theorem~\ref{poly1} follows from
\begin{align}\label{eq105}
\mathbb{E}_{\m\in [X^{\varepsilon}]^r}\mathbb{E}_{n\leq X}c_X(\m) \lambda(n+P_1(\m)) \lambda(n+P_2(\m)) \dotsm \lambda(n+P_k(\m))=o(1)   
\end{align}
for an arbitrary unimodular sequence $c_X(\m)$. Denoting $W=\prod_{p\leq w}p$, where $w$ tends to infinity very slowly in terms of $X$, and splitting $n$ and $m$ into residue classes $\pmod W$ in the statement of Theorem~\ref{poly2}, that theorem in turn reduces to 
\begin{align}\label{eq106}
\mathbb{E}_{\m\in [L]^r}\mathbb{E}_{n\leq X/W}c_X(\m) \lambda_{b_1,W}(n+P_1'(\m)) \Lambda_{b_2,W}(n+P_2'(\m)) \dotsm \Lambda_{b_k,W}(n+P_k'(\m))=o(1)
\end{align}
uniformly for unimodular sequences $c_X(\m)$,  for $X^{\varepsilon}/W \ll L\ll X^{\varepsilon}$, for $1\leq b_1,\ldots, b_k\leq W$ coprime to $W$, and for $P_1',\ldots, P_k'$ polynomials in $\mathbb{Z}[x_1,\ldots, x_r]$ with $P_i'-P_j'$ non-constant for $i\neq j$, and $\deg{P_i'}\leq d$, and the coefficients of $P_i'$ bounded by $W^{1/\kappa}$ in absolute value for some constant $\kappa>0$ (cf.~\cite[Section 5]{tao-ziegler2} for this reduction). Here we have denoted $\lambda_{b,W}(n):=\lambda(Wn+b)$, and recall that $\Lambda_{b,W}(n):=\phi(W)/W\cdot \Lambda(Wn+b)$. We now see that in fact both Theorem~\ref{poly1} and~\ref{poly2} will follow once we prove \eqref{eq106} in a form where some copies of $\Lambda$ are allowed to be replaced with $\lambda$.

Let $A=W^{1/\kappa}$, so that the absolute values of the coefficients of $P_i'$ are bounded by $A$.  Recall $d=\max_i \deg{P_i'}$. We set $N=\lfloor X/W\rfloor$, so that $L=o(N^{1/d})$. Consider functions $f_1,\ldots, f_k:[N]\to \mathbb{C}$ with $|f_i|\ll \Lambda_{b_i,W}+1$ and $|f_1| \leq 1$. Extend the $f_i$ to functions $\tilde{f_i}:\mathbb{Z}/N\mathbb{Z}\to \mathbb{C}$ by making them $N$-periodic. Observe that 
\begin{align}\label{eq:PolAv1}
\mathbb{E}_{\m \in [L]^r}\mathbb{E}_{n\leq N}c_X(\m)f_1(n+P_1'(\m))\cdots f_k(n+P_k'(\m))    
\end{align}
is up to $o(1)$ error equal to
\begin{align}\label{eq:PolAv6}
 \mathbb{E}_{\m\in [L]^r}c_X(\m)\mathbb{E}_{n\in \mathbb{Z}/N\mathbb{Z}}\tilde{f_1}(n+P_1'(\m))\cdots \tilde{f_k}(n+P_k'(\m)),   
\end{align}
since the components of the $\m$ variable in \eqref{eq:PolAv1} are bounded by $\eta X^{1/d}$ for some $\eta>0$ small enough in terms of $A$, so that wraparound issues are negligible.

This latter expression is in turn bounded using van der Corput's inequality (see e.g.~\cite[Formula (4.1)]{gt-U3}) by
\begin{align*}
\ll (\mathbb{E}_{h\in \mathbb{Z}/N\mathbb{Z}}|\mathbb{E}_{\m\in [L]^r}\mathbb{E}_{n\in \mathbb{Z}/N\mathbb{Z}}\Delta_{h}\tilde{f_1}(n+P_1'(\m))\cdots \Delta_h\tilde{f_k}(n+P_k'(\m))|)^{1/2},  
\end{align*}
where $\Delta_hf(x):=f(x+h)\overline{f(x)}$.

By~\cite[Theorem 13]{tao-ziegler2}, for any polynomials $P_i'$ as in Theorems~\ref{poly1}, \ref{poly2}, we have 
\begin{align*}
|\mathbb{E}_{n\in \mathbb{Z}/N\mathbb{Z}}\mathbb{E}_{\m\in [L]^r}\Delta_h\tilde{f_1}(n+P_1'(\m))\cdots \Delta_h\tilde{f_k}(n+P_k'(\m))|=o(1),
\end{align*}
provided that    
\begin{align}\label{eq:PolAv3}
\mathbb{E}_{\mathbf{t}\in [A^{-1}L]^r}\|\Delta_h\tilde{f_1}\|_{\square^{D'}_{Q_1(\mathbf{t})[-A^{-1}L,A^{-1}L],\ldots, Q_{D'}(\mathbf{t})[-A^{-1}L,A^{-1}L]}}=o(1)
\end{align}
for any fixed $D'\geq 1$ and any polynomials $Q_1,\ldots, Q_{D'}\in \mathbb{Z}[t_1,\ldots, t_r]$ not identically zero and with coefficients of size $O(A^{O(1)})$,
where 
\begin{align*}
\|f\|_{\square^{d}_{C_1,\ldots, C_{d}}}:=\left(\mathbb{E}_{x\in \mathbb{Z}/N\mathbb{Z}}\mathbb{E}_{h_1\in C_1-C_1}\cdots \mathbb{E}_{h_d\in C_d-C_d}\prod_{\omega\in \{0,1\}^d}\mathcal{C}^{|\omega|}f(x+\omega\cdot \mathbf{h})\right)^{1/2^d}    
\end{align*}
is a Gowers box norm of order $d$ and $\mathcal{C}$ is the complex conjugation operator, and we used the notation $q[-N,N]:=[-qN,qN]\cap q\mathbb{Z}$. Thus we may control polynomial averages with averaged Gowers box norms. Further, by a concatenation theorem, namely~\cite[Theorem 9]{tao-ziegler2} (with $d_0=1$ there), we have \eqref{eq:PolAv3} provided that
\begin{align}\label{eq:PolAv4}
\|\Delta_h\tilde{f_1}\|_{U^{D''}_{q[1,A^{-2D''}L]}}=o(1)
\end{align}
holds for all fixed $D''\geq 1$ and all $1\leq q\leq A^{D''}$, where $\|f\|_{U^{d}_C}:=\|f\|_{\square^{d}_{C,\ldots, C}}$.

Averaging this over $h$, we now conclude that the desired bound for \eqref{eq:PolAv6} follows from
\begin{align*}
\mathbb{E}_{h\in \mathbb{Z}/N\mathbb{Z}}\|\Delta_h\tilde{f_1}\|_{U^{D''}_{q[1,A^{-2D''}L]}}^{2^{D''}}=o(1).    
\end{align*}
Expanding out the Gowers norm above, we see that this claim in turn reduces to
\begin{align}\label{eq:PolAv7}
\|\tilde{f_1}\|_{U^{D''+1}_{\mathbb{Z}/N\mathbb{Z},q[1,A^{-2D''}L],\ldots, q[1,A^{-2D''}L]}}=o(1).   
\end{align}
Since wraparound issues are again negligible, we can split the average over $\mathbb{Z}/N\mathbb{Z}$ implicit in \eqref{eq:PolAv7} into intervals of length $\asymp L$ and apply the generalized von Neumann theorem, thus reducing the proof of \eqref{eq:PolAv7} to
\begin{align}\label{eq:PolAv8}
\sup_{A^{-c}L\leq M\leq A^{c}L}\mathbb{E}_{n\leq N-M} \|f_1\|_{U^{D''+1}[n,n+M]}=o(1)   
\end{align}
for any constant $c\geq 1$.

Now specialize to the case where $f_1$ is the ($W$-tricked) Liouville function $\lambda_{b,W}(n)1_{[N]}(n)$ (and $N=\lfloor X/W\rfloor$ as before). By making a change of variables, and extending the range of the supremum in $W1_{m\equiv b\pmod W}$, we reduce \eqref{eq:PolAv8} to
\begin{align}\label{eq:PolAv9}
\sup_{N^{\varepsilon/2}\leq M\leq N^{2\varepsilon}}\mathbb{E}_{n\leq W(N-M)} \|\lambda \cdot W 1_{\cdot\equiv b\pmod W}1_{[WN]}\|_{U^{D''+1}[n,n+M]}=o(1).   
\end{align}
The factor $1_{[WN]}$ can be removed, since the contribution to the $n$ average from the range $WN-O(M)\leq n\leq W(N-M)$ is negligible. By Fourier expanding $1_{\cdot\equiv b\pmod W}$ in terms of additive characters, and applying the triangle inequality (and recalling that $w$ tends to infinity arbitrarily slowly) we reduce\footnote{Note that even though the Fourier expansion of $1_{\cdot\equiv b\pmod W}$ followed by the triangle inequality loses a multiplicative factor of $W$, this loss is harmless, since $w$, and hence $W$, can be assumed to tend to infinity much slower than the decay rate of \eqref{eq:PolAv9} without the $W1_{\cdot\equiv b\pmod W}1_{[WN]}$ factor.} to proving \eqref{eq:PolAv8} also without the factor $W1_{\cdot\equiv b\pmod W}$.

By our main theorem, Theorem~\ref{mult-pret}, we have \eqref{eq:PolAv9} without the term $W1_{m\equiv b\pmod W}1_{[WN]}$, and therefore taking above $f_i\in \{\lambda_{b_i,W},\Lambda_{b_i,W}\}$ for $1\leq i\leq k$, both Theorem~\ref{poly1} and Theorem~\ref{poly2} follow.\end{proof}

\appendix
\section{Bernstein inequality for exponential polynomials}

\label{app:a}

In this appendix we establish the Bernstein inequality for exponential polynomials, Lemma~\ref{ber-exp}.  We begin with a bound for the number of zeroes of such polynomials:

\begin{lemma}\label{lom} Let $\alpha_1,\dots,\alpha_k$ be real numbers, let $d_1,\dots,d_k$ be non-negative integers, and let $P: \R \to \R$ be a real linear combination of the exponential monomials $t \mapsto t^j \exp(\alpha_i t)$ for $i=1,\dots,k$ and $0 \leq j \leq d_i$.  Then if $P$ is not identically zero, it has at most $k+\sum_{i=1}^k d_i$ zeroes.
\end{lemma}

\begin{proof} The claim is trivial for $k=0$, so suppose that $k \geq 1$ and that the claim has already been proven for $k-1$.  We now fix $k$ and induct on $\sum_{i=1}^k d_i$. By reordering we may assume that $d_1\leq d_2\leq \cdots \leq d_k$.  By multiplying $P$ by $t \mapsto \exp(-\alpha_1 t)$ we may assume that $\alpha_1=0$. If $d_1$ vanishes, then the derivative $P'$ is a linear combination of the exponential monomials $t\mapsto t^j\exp(\alpha_it)$ with $2\leq i\leq k$ and $0\leq j\leq d_i$, so the claim follows from the outer induction hypothesis on $k$ and Rolle's theorem. If instead $d_1$ does not vanish, then $P'$ is of the same form as $P$ but with $d_1$ replaced by $d_1-1$, thus by the induction hypothesis it either vanishes identically or has at most $k+(\sum_{i=1}^k d_i) - 1$ zeros.  The claim now follows from Rolle's theorem. 
\end{proof}

\begin{proof}[Proof of Lemma~\ref{ber-exp}] We allow all implied constants to depend on $k,d_1,\dots,d_k,m,I$. Let $N_0$ be large enough in terms of $k,d_1,\ldots, d_k$.  We may normalize $\sup_{n=1,\dots,N_0} |P(n)| = 1$.  The claim is trivial if $P$ is constant, so we may assume that $P$ is non-constant.  By Lemma~\ref{lom} the exponential polynomial $P(t)$ then attains the values $\pm 1$ at most $O(1)$ times, so the set $\{ t \in \R: |P(t)| \leq 1 \}$ is the union of $O(1)$ intervals (possibly of infinite or zero length).  As this set contains $\{1,\dots,N_0\}$, we conclude from the pigeonhole principle (for $N_0$ large enough in terms of $d_1, \dotsc, d_k$) that this set also contains an interval $[n,n+1]$ for some $n=1,\dots,N_0-1$.  

Now observe that $P$ solves the ordinary differential equation
$$ \prod_{i=1}^k \left( \frac{d}{dt} - \alpha_i \right)^{d_i+1} P(t) = 0.$$
Writing $D \coloneqq \sum_{i=1}^k (d_i+1) = O(1)$ and $\eps \coloneqq \sup_{1 \leq i \leq k} |\alpha_i|$ (where, by assumption, $\varepsilon$ is small enough in terms of $k,d_1,\ldots, d_k,N_0$), we can write this equation as
\begin{equation}\label{pdt}
 P^{(D)}(t) + c_{D-1} P^{(D-1)}(t) + \dots + c_0 P(t) = 0
\end{equation}
where the coefficients $c_0,\dots,c_{D-1}$ are of size $O(\eps)$.  In terms of the $D$-dimensional vector
$$ v(t) \coloneqq \begin{pmatrix} P(t) \\ \vdots \\ P^{(D-1)}(t) \end{pmatrix}$$
one can write this differential equation as a first-order system
$$ \frac{d}{dt} v(t) = (U+E) v(t)$$
where $U$ is the shift matrix
$$ U \coloneqq \begin{pmatrix} 0 & 1 & \dots & 0 \\ 0 & 0 & \dots & 0 \\ \vdots & \vdots & \ddots & \vdots \\ 0 & 0 & \dots & 1 \\ 0 & 0 & \dots & 0 \end{pmatrix}$$
and $E$ is a $t$-independent matrix of dimension $D$ with all entries being of size $O(\eps)$.
The solution of this equation is $$v(t) =  \exp( (t-n) (U+E )) v(n).$$ 
 By the continuity of the matrix exponential we then have 
\begin{equation}\label{vt}
 v(t) = \exp( (t-n) U ) v(n) + O(\eps \|v(n)\| )
\end{equation}
whenever $|t-n|=O(1)$ (here $\|\cdot\|$ denotes the Euclidean norm of a vector).  In particular, we have the approximate Taylor expansion
$$ P(t) = \sum_{j=0}^{D-1} \frac{(t-n)^j}{j!} P^{(j)}(n) + O( \eps \|v(n)\| ).$$
Since $|P(t)| \leq 1$ for $t \in [n,n+1]$, we conclude that
$$ \sum_{j=0}^{D-1} \frac{(t-n)^j}{j!} P^{(j)}(n) \ll 1 + \eps \|v(n)\|$$
for $t \in [n,n+1]$.  From \eqref{ber-2} applied to the polynomial in $t$ on the left-hand side we have that
$$ |P^{(j)}(n)| \ll 1 + \eps \|v(n)\|.$$
We conclude that
$$ \|v(n)\| \ll 1 + \eps \|v(n)\|$$ 
and hence for $\eps$ small enough we see that all the components of $v(n)$ are $O(1)$.  Inserting this back into \eqref{vt} we conclude that \eqref{pmt} holds for all $m \leq D-1$; the remaining cases then follow by differentiating the equation \eqref{pdt} $m-D$ times and using induction on $m$.
\end{proof}
\section{The Baker--Campbell--Hausdorff formula and its consequences}\label{bch}

In this section, we review some standard facts about connected, simply connected nilpotent Lie groups $G$ and their Lie algebras $\log G$.  As mentioned in Section~\ref{nilseq}, all connected, simply connected nilpotent Lie groups are isomorphic to matrix algebras, so we shall abuse notation in this appendix by viewing elements of $G$ and $\log G$ as matrices (in particular we identify the Lie group exponential with the matrix exponential).

If $G$ is a simply connected nilpotent Lie group with some filtration $(G_i)_{i \geq 0}$ with $G_i=0$ for $i>k$, we can define the operation $\ast: \log G \times \log G \to \log G$ by the formula
\begin{equation}\label{ast-def}
 X \ast Y \coloneqq \log( \exp(X) \exp(Y) )
\end{equation}
for all $X,Y \in \log G$, or equivalently
$$ \log(gh) = \log g \ast \log h$$
for all $g,h \in G$.  For instance, in the Heisenberg group example from Example~\ref{heisen}, we have
$$ \begin{pmatrix} 0 & x_1 & z_1 \\ 0 & 0 & y_1 \\ 0 & 0 & 0 \end{pmatrix} \ast \begin{pmatrix} 0 & x_2 & z_2 \\ 0 & 0 & y_2 \\ 0 & 0 & 0 \end{pmatrix} = \begin{pmatrix} 0 & x_1+x_2 & z_1 + z_2+\frac{x_1y_2-x_2y_1}{2} \\ 0 & 0 & y_1+y_2 \\ 0 & 0 & 0 \end{pmatrix}.$$

The operation $\ast$ is clearly a group operation on $\log G$ (with identity $0$ and inverse map $X \mapsto -X$).  The \emph{Baker--Campbell--Hausdorff formula} gives an explicit description of this operation.  As is well known, $\log G$ is a nilpotent Lie algebra, using the usual matrix commutator $[X,Y] \coloneqq XY - YX$ as the Lie bracket; see~\cite[Corollary 11.2.7]{hilgert}.  For any $X \in \log G$, we can then define the adjoint representation $\mathrm{ad}_X\colon \log G \to \log G$ to be linear map
$$ \mathrm{ad}_X(Y) \coloneqq [X,Y].$$
As $\log G$ is a nilpotent Lie algebra, $\mathrm{ad}_X$ is a nilpotent linear transformation, thus $\mathrm{ad}_X^m = 0$ for some natural number $m$; more generally, for any $X,Y \in \log G$, any word in $\mathrm{ad}_X, \mathrm{ad}_Y$ of length greater than or equal to some threshold $m$ will vanish (in fact, by the inclusion \eqref{gij-inc} established below, one can take $m$ to equal the degree $k$ of the filtration).  The Baker--Campbell--Hausdorff formula then states
$$ X \ast Y = X + \int_0^1 \psi( e^{\mathrm{ad}_X} e^{t\mathrm{ad}_Y} ) Y\ dt,$$
where $e^{\mathrm{ad}_X} = \sum_{n=0}^\infty \frac{1}{n!} \mathrm{ad}_X^n$ is the matrix exponential of $\mathrm{ad}_X$, and $\psi$ is the function
$$ \psi(x) \coloneqq \frac{x \log x}{x-1} = 1 + \frac{x-1}{2} - \frac{(x-1)^2}{6} + \dots;$$
see for instance~\cite[Theorem 3.3]{hall} or~\cite[Proposition 3.4.4]{hilgert}.  Note that from the nilpotent nature of $\log G$ that we can truncate the Taylor series for the matrix exponential and the function $\psi$ to some finite threshold $m$, so that
\begin{equation}\label{poly}
X \ast Y = X + Y + P( \mathrm{ad}_X, \mathrm{ad}_Y ) Y
\end{equation}
for some (non-commutative) polynomial $P$ of two variables of bounded degree and coefficients that are rational numbers of bounded height, where the constant term of $P$ vanishes and the linear term is equal to $\frac{1}{2} \mathrm{ad}_X$ (the contribution of $\mathrm{ad}_Y$ can be deleted from the linear term since $\mathrm{ad}_Y Y = 0$).  The first few terms of this formula are
\begin{align*}
 X \ast Y &= X + Y + \frac{1}{2} \mathrm{ad}_X Y + \frac{1}{12} (\mathrm{ad}_X^2 - \mathrm{ad}_Y \mathrm{ad}_X) Y + \dots \\
&= X + Y + \frac{1}{2} [X,Y] + \frac{1}{12} ([X,[X,Y]] - [Y,[X,Y]]) + \dots,
\end{align*}
although we will not need the explicit form of these terms beyond the quadratic case.  From \eqref{poly} we conclude in particular that $X \ast Y$ is a polynomial combination of $X, Y$, with bounded degree and coefficients. As one particular consequence of this formula, we see that
$$ (tX) \ast (tY) \ast (-tX) \ast (-tY) = t^2 [X,Y] + O(t^3)$$
as $t \to 0$ for any $X,Y \in \log G$, so the Lie bracket can be recovered from $\ast$ by the limiting formula
\begin{equation}\label{lim}
[X,Y] = \lim_{t \to 0} \frac{(tX) \ast (tY) \ast (-tX) \ast (-tY) }{t^2},
\end{equation}
which can also be established directly from \eqref{ast-def} and Taylor expansion of the matrix exponential (this is also~\cite[(3.14)]{hilgert}).

Another closely related identity to the Baker--Campbell--Hausdorff formula is
$$ e^{\mathrm{ad}_X} Y = \exp(X) Y \exp(-X)$$
for any $X,Y \in \log G$; see~\cite[Proposition 2.25]{hall}.  
As $\exp (C^{-1}YC) = C^{-1} \exp (Y)C$ for any invertible $C$, we have 
$$ \exp( \exp(X) Y \exp(-X) ) = \exp(X) \exp(Y) \exp(-X)$$
and thus
$$ \exp( e^{\mathrm{ad}_X} Y ) = \exp(X) \exp(Y) \exp(-X)$$
for all $X,Y \in \log G$, or equivalently
\begin{equation}\label{hgh}
 \log( h g h^{-1} ) = e^{\mathrm{ad}_{\log h}} \log g
\end{equation}
for all $g,h \in G$.

By definition, the groups $G_i$ in the filtration $(G_i)$ are closed subgroups of $G$, and thus are themselves Lie groups with a Lie algebra $\log G_i$ that are subalgebras of $\log G$; see~\cite[Proposition 9.3.9]{hilgert}, \cite[Proposition 3.14]{hall}.  In particular, the exponential map $\exp: \log G \to G$ descends to a diffeomorphism $\exp: \log G_i \to G_i$, so $G_i$ is simply connected. 
The group $G_i$ is nilpotent simply connected, and $G_{i+1}$ is a closed simply connected nilpotent subgroup, thus $G_i/G_{i+1}$ is simply connected. 
  If $X \in \log G_i$ and $Y \in \log G_j$, then from the filtration property $[G_i,G_j] \subset G_{i+j}$ and \eqref{ast-def} we see that $(tX) \ast (tY) \ast (-tX) \ast (-tY) \in \log G_{i+j}$ for any $t>0$; inserting this into \eqref{lim} we conclude that $[X,Y] \in \log G_{i+j}$, thus we have the Lie algebra filtration property
\begin{equation}\label{gij-inc}
[\log G_i, \log G_j] \subset \log G_{i+j}.
\end{equation}
In particular, each of the $\log G_i$ are normal Lie subalgebras of $\log G$.  From the Baker--Campbell--Hausdorff formula \eqref{poly} and \eqref{gij-inc} we then also have
$$ X \ast Y = X + Y \mod \log G_{i+1}$$
whenever $i \geq 1$ and $X,Y \in \log G_i$, or equivalently
\begin{equation}\label{gh}
 \log(gh) = \log(g) + \log(h) \mod \log G_{i+1}
\end{equation}
whenever $i \geq 1$ and $g,h \in G_i$.  Thus $\log G_i / \log G_{i+1}$ is an abelian Lie algebra for any $i \geq 1$, and the logarithm map descends to a homomorphism from the multiplicative group $G_i/G_{i+1}$ to the additive group $\log G_i / \log G_{i+1}$.

\begin{lemma}[Taylor expansion]\label{taylo}  Let $d \geq 1$ be a natural number, and let $g \in \Poly(\Z \to G)$.  Then there exist unique \emph{Taylor coefficients} $g_j \in G_{j}$ such that
$$ g(n) = \prod_{j} g_j^{\binom{n}{j}}.$$
\end{lemma}

\begin{proof} This is a special case of~\cite[Lemma B.9]{gtz}.
\end{proof}

Now we can prove Lemma~\ref{uniq}.

\begin{proof}[Proof of Lemma~\ref{uniq}]  We may rescale $\delta=1$.  The fact that $\Poly(\Z \to G)$ forms a group is the Leibman--Lazard theorem; see e.g.,~\cite[Corollary B.4]{gtz}. Now suppose that $\tilde g \in \Poly(\R \to G)$, thus we have a Taylor expansion
$$ \log \tilde g(t) = \sum_{i=0}^k X_i t^i$$
for some $X_i \in \log G_i$.  For any $j \geq 0$, let $V_j$ denote the vector space of polynomial maps $p\colon \R \to \log G$ of the form
$$ p(t) = \sum_{0 \leq i \leq k-j} Y_i t^i$$
where $Y_i \in \log G_{i+j}$ for all $i$, thus $\log \tilde g \in V_0$.  One can check that the $V_j$ are decreasing with 
\begin{equation}\label{vij}
[V_i,V_j] \subset V_{i+j}
\end{equation}
and $V_i=0$ for $i > k$; in particular, the $V_i$ are each Lie algebras.  We now claim by induction that
$$ \log \partial_{h_1} \dots \partial_{h_j} \tilde g \in V_j$$
for all $j \geq 0$ and $h_1,\dots,h_j \in \R$.  This claim is already established for $j=0$.  If it holds for some $j$, and $h_{j+1} \in \R$, then by using the fact that $(t+h_{j+1})^i$ differs from $t^i$ by a polynomial of degree at most $i-1$ in $t$, we see that
$$ \log \partial_{h_1} \dots \partial_{h_j} \tilde g(\cdot + h_{j+1}) = \log \partial_{h_1} \dots \partial_{h_j} \tilde g \mod V_{j+1}$$
and hence by the Baker--Campbell--Hausdorff formula \eqref{poly}
$$ \log \partial_{h_1} \dots \partial_{h_j} \tilde g(\cdot + h_{j+1}) \ast (-\log \partial_{h_1} \dots \partial_{h_j} \tilde g) \in V_{j+1}.$$
But by \eqref{ast-def} the left-hand side is equal to $\log \partial_{h_1} \dots \partial_{h_{j+1}} \tilde g$, closing the induction.
Applying this with $j=k$ and $h_1,\dots,h_j,t \in \Z$, we conclude that the restriction of $\tilde g$ to $\Z$ lies in $\Poly(\Z \to G)$.

Now suppose that $g \in \Poly(\Z \to G)$.  Any such element can be expressed uniquely as a Taylor expansion
$$ g(n) = g_0 g_1^{\binom{n}{1}} \dots g_k^{\binom{n}{k}}$$
for all $n \in \Z$ and some $g_j \in G_j$; see~\cite[Lemma B.9]{gtz}.  Using the real exponentiation \eqref{realexp}, we can extend $g$ to the map
\begin{equation}\label{tg}
 \tilde g(t) = g_0 g_1^{\binom{t}{1}} \dots g_k^{\binom{t}{k}}
\end{equation}
and from many applications of the Baker--Campbell--Hausdorff formula \eqref{ast-def}, \eqref{poly}, \eqref{gij-inc} one sees that $\tilde g$ is now an element of $\Poly(\R \to G)$.  This establishes existence.  To show uniqueness, it suffices by the group property to check the case $g=1$.  Then any extension $\tilde g$ is such that $\log \tilde g(n) = 0$ for every integer $n$; since $\log \tilde g$ is also a polynomial, $\log \tilde g$ vanishes identically, hence $\tilde g$ must be $1$, giving uniqueness.
\end{proof}

As a corollary we obtain: 
\begin{lemma}[Non-abelian Discrete Taylor expansion]\label{na-dte}  For any $\delta > 0$, the space $\Poly(\delta \Z \to G)$ consists precisely of those functions $\gamma: \R \to G$ of the form
$$ \gamma(t) \coloneqq \prod_{j = 0}^k g_j^{\binom{t/\delta}{j}}$$
for some $g_j \in G_j$, where $\binom{x}{j} \coloneqq \frac{x(x-1)\dots(x-j+1)}{j!}$.
\end{lemma}

If $\Gamma$ is a cocompact discrete subgroup of $G$ with each $\Gamma_i \coloneqq \Gamma \cap G_i$ cocompact in $G_i$, then there exists a \emph{Mal'cev basis} for $\Gamma$, by which we mean a linear basis $X_1,\dots,X_{\dim G}$ for $\log G$ with the property that $X_{\dim G - \dim G_i + 1},\dots,X_{\dim G}$ form a basis for $\log G_i$ for each $i$ (so in particular $[X_i,X_j]$ lies in the span of $X_{\max(i,j)+1},\dots,X_{\dim G}$ for any $1 \leq i,j \leq \dim G$), and
$$ \Gamma = \{ \exp( n_1 X_1 ) \dotsm \exp( n_{\dim G} X_{\dim G} ): n_1,\dots,n_{\dim G} \in \Z \}.$$
See~\cite[\S 2]{gt-nil} for details.  From this and many applications of the Baker--Campbell--Hausdorff formula, we see that for any $1 \leq i,j \leq \dim G$, the coefficients of $[X_i,X_j]$ in the basis $X_{\max(i,j)+1},\dots,X_{\dim G}$ are rational numbers with denominator $O(1)$, and thus every  element of $\Gamma$ can be written in the form
\begin{equation}\label{gamma-rat} \exp( \frac{1}{Q_1} (n_1 X_1 + \dots + n_{\dim G} X_{\dim G}) )\end{equation}
for some integers $n_1,\dots,n_{\dim G}$ and some natural number $Q_1=O(1)$ depending only on $G$ and the Mal'cev basis; conversely, there exists a natural number $Q_2 = O(1)$ such that every expression of the form
$$ \exp( Q_2 (n_1 X_1 + \dots + n_{\dim G} X_{\dim G}) )$$
with $n_1,\dots,n_{\dim G} \in \Z$ lies in $\Gamma$.  One consequence of this and the Baker--Campbell--Hausdorff formula is that, for any fixed natural number $q = O(1)$, the set $\{ \gamma \in G: \gamma^q \in \Gamma \}$ generates a group, all of whose elements are of the form
$$ \exp( \frac{1}{Q} (n_1 X_1 + \dots + n_{\dim G} X_{\dim G}) )$$
for some $Q$ depending on $G$, $q$, and the Mal'cev basis; in particular, this group contains only finitely many cosets of $\Gamma$, so that $\Gamma$ is a finite index subgroup of it.  As one particular corollary of this, we see that if $\gamma_1, \gamma_2 \in G$ are such that $\gamma_1^{q_1}, \gamma_2^{q_2} \in \Gamma$ for some natural numbers $q_1,q_2 = O(1)$, then one has $(\gamma_1 \gamma_2)^q \in \Gamma$ for some $q = O(1)$.\\

\section{Bezout's identity and the Chinese remainder theorem for polynomial spaces}

\label{app:b}

In this section, we prove various versions of Bezout's identity and the Chinese remainder theorem for polynomial maps, either into the circle $\R/\Z$ or into more general filtered nilpotent Lie groups.

\subsection{Bezout-type identities}

\begin{proof}[Proof of Lemma~\ref{bezout}] We may normalize $\lambda=1$. We begin with the first claim.  It suffices to establish the inclusion
$$ \Poly_{\leq k}\left(\frac{1}{a} \Z \to \Z\right) + \Poly_{\leq k}\left(\frac{1}{b} \Z \to \Z\right) \supset \Poly_{\leq k}(\Z \to \Z)$$
as the opposite inclusion is trivial.  That is, it suffices to show that every $\gamma \in \Poly_{\leq k}(\Z \to \Z)$ may be split as $\gamma = \gamma_a + \gamma_b$ where $\gamma_a \in \Poly_{\leq k}\left(\frac{1}{a} \Z \to \Z\right)$ and $\gamma_b \in \Poly_{\leq k}\left(\frac{1}{b}\Z \to \Z\right)$.

We prove this by induction on $k$.  The claim is trivial for $k=0$, so suppose that $k \geq 1$ and that the claim has already been proven for $k-1$.  From Lemma~\ref{dte} we can write $\gamma(t) = c \binom{t}{k} + \gamma^*(t)$ for some integer $c$ and $\gamma^* \in \Poly_{\leq k-1}(\R \to \R)$.  By Bezout's identity we may write $c = qa^k + rb^k$ for some integers $q,r$, thus
$$ \gamma(t) = q \binom{at}{k} + r \binom{bt}{k} + \gamma^{**}(t)$$
for some $\gamma^{**} \in \Poly_{\leq k-1}(\R \to \R)$.  As $\gamma(\Z) \subset \Z$, also $\gamma^{**}(\Z) \subset \Z$; so by the induction hypothesis we may write $\gamma^{**}(t) = \gamma^{**}_a(t) + \gamma^{**}_b(t)$ where $\gamma^{**}_a \in \Poly_{\leq k-1}\left(\frac{1}{a}\Z \to \Z\right)$ and $\gamma^{**}_b \in \Poly_{\leq k-1}\left(\frac{1}{b}\Z \to \Z\right)$.  Setting $\gamma_a(t) \coloneqq q \binom{at}{k} + \gamma^{**}_a(t)$ and $\gamma_b(t) \coloneqq r \binom{bt}{k} + \gamma^{**}_b(t)$ closes the induction.

Now we prove the second claim.  Again it suffices to prove the inclusion
$$ \Poly_{\leq k}\left(\frac{1}{a} \Z \to \Z\right) \cap \Poly_{\leq k}\left(\frac{1}{b} \Z \to \Z\right) \subset \Poly_{\leq k}\left(\frac{1}{ab} \Z \to \Z\right)$$
as the opposite inclusion is trivial, and we may again inductively assume that $k \geq 1$ and that the claim has already been proven for $k-1$.

If $\gamma \in \Poly_{\leq k}\left(\frac{1}{a} \Z \to \Z\right) \cap \Poly_{\leq k}\left(\frac{1}{b} \Z \to \Z\right)$, then from Lemma~\ref{dte} we see that the derivative $\gamma^{(k)}$ (which is a constant) is an integer multiple of both $a^k$ and $b^k$, hence can be written as $c(ab)^k$ for some integer $c$.  Thus we may write $\gamma(t) = c \binom{abt}{k} + \gamma^*(t)$ for some integer $c$ and $\gamma^* \in \Poly_{\leq k-1}(\R \to \R)$.  One then easily checks that
$$\gamma^* \in \Poly_{\leq k-1}\left(\frac{1}{a} \Z \to \Z\right) \cap \Poly_{\leq k-1}\left(\frac{1}{b} \Z \to \Z\right) $$
and the claim now follows from the induction hypothesis and Lemma~\ref{dte}.
\end{proof}

\begin{proof}[Proof of Lemma~\ref{bezout-nil}]
We again normalize $\lambda=1$. We begin with the first claim.  As $\Poly(\Z \to \Gamma)$ is a group that contains\footnote{We remind here that, by Lemma~\ref{uniq}, the group $\Poly(\delta \Z \to \Gamma)$ can be (by an abuse of notation) interpreted as a subgroup of $\Poly(\R \to \Gamma)$.} $\Poly(\frac{1}{a} \Z \to \Gamma), \Poly(\frac{1}{b} \Z \to \Gamma)$, we clearly have the inclusion
$$ \Poly(\frac{1}{a} \Z \to \Gamma) \cdot \Poly(\frac{1}{b} \Z \to \Gamma) \subset \Poly(\Z \to \Gamma)$$
and it now suffices to show that any $\gamma \in \Poly(\Z \to \Gamma)$ can be factored as $\gamma = \gamma_a \gamma_b$, where $\gamma_a \in \Poly(\frac{1}{a} \Z \to \Gamma)$ and $\gamma_b \in \Poly(\frac{1}{b}\Z \to \Gamma)$.

Set $\Gamma_i \coloneqq G_i \cap \Gamma$ for all $i$.  If $\gamma$ lies in $\Poly(\Z \to \Gamma_{k+1})$ then the claim is trivial since $\Gamma_{k+1}=\{1\}$, so now suppose by downward induction that $\gamma$ lies in $\Poly(\Z \to \Gamma_i)$ for some $1 \leq i \leq k$, and that the claim has already been proven for $\gamma$ in $\Poly(\Z \to \Gamma_{i+1})$.  By Lemma~\ref{na-dte} we have a Taylor expansion of the form
$$ \gamma(t) = \prod_{j}\gamma_j^{\binom{t}{j}}.$$
Since for $t \in \Z$ we have  $\gamma(t)  \in \Gamma_i$ we get by induction on $n$ that $\gamma_j \in \Gamma_i$. If we let $\pi_i \colon \Gamma_i \to \Gamma_i / \Gamma_{i+1}$ be the quotient map, then since  $\Gamma_i/\Gamma_{i+1}$ is abelian we get  for $t \in \Z$
$$ \pi_i(\gamma(t)) = \prod_{j=0}^i \pi_i(\gamma_j)^{\binom{t}{j}}.$$

By Lemma~\ref{bezout}, we can split each $\binom{t}{j}$ as $P_{a,j}(t) + P_{b,j}(t)$ for $t \in \R$ and some $P_{a,j} \in \Poly_{\leq j}(\frac{1}{a}\Z \to \Z)$ and $P_{b,j} \in \Poly_{\leq j}(\frac{1}{b}\Z \to \Z)$.  Setting
$$ \gamma'_a(t) \coloneqq \prod_{j=0}^i \gamma_j^{P_{a,j}(t)}; \quad 
\gamma'_b(t) \coloneqq \prod_{j=0}^i \gamma_j^{P_{b,j}(t)}$$
for all $t \in \R$, we see that $\gamma'_a \in \Poly(\frac{1}{a} \Z \to \Gamma)$, $\gamma'_b \in \Poly(\frac{1}{b} \Z \to \Gamma)$, and
$$ \gamma = \gamma'_a \sigma \gamma'_b$$
for some $\sigma \in \Poly(\Z \to \Gamma_{i+1})$.  The claim now follows from the induction hypothesis.

Now we prove the second claim.  We show by downwards induction on $k$ that for each $1 \leq i \leq k+1$ and $\gamma \in \Poly(\frac{1}{a} \Z \to \Gamma_i) \cap \Poly(\frac{1}{b} \Z \to \Gamma_i)$ one has $\gamma \in \Poly(\frac{1}{ab}\Z \to \Gamma_i)$.  The claim is trivially true for $i=k+1$, so suppose that $1 \leq i \leq k$ and that the claim has already been proven for $i+1$.  From two applications of Lemma~\ref{na-dte} and with $\pi_i$ as above, we have
\begin{equation}\label{pig}
 \pi_i(\gamma(t)) = \prod_{j=0}^i \pi_i(\gamma_{j,a})^{\binom{at}{j}}
\end{equation}
for all $t \in \frac{1}{a} \Z$ and some $\gamma_{j,a} \in \Gamma_i$, and
\begin{equation}\label{pig-2}
 \pi_i(\gamma(t)) = \prod_{j=0}^i \pi_i(\gamma_{j,b})^{\binom{bt}{j}}
\end{equation}
for all $t \in \frac{1}{b} \Z$ and some $\gamma_{j,b} \in \Gamma_i$.  Specializing to $t \in \Z$ and comparing the top order coefficients of these polynomials (using the uniqueness of the Taylor expansion) in the abelian group $\Gamma_i/\Gamma_{i+1}$, we conclude that
$$\pi_i(\gamma_{i,a})^{a^i} = \pi_i(\gamma_{i,b})^{b^i}.$$
As $a^i,b^i$ are coprime, the Bezout identity allows one to express $1$ as an integer combination of $a^i,b^i$.  We conclude that there exists $\gamma_i \in \Gamma_i$ such that $\pi_i(\gamma_{i,a}) = \pi(\gamma_i)^{b^i}$ and $\pi_i(\gamma_{i,b}) = \pi(\gamma_i)^{a^i}$.  If one then divides out the polynomial $t \mapsto \gamma_i^{\binom{abt}{i}}$ (which lies in $\Poly(\frac{1}{ab}\Z \to \Gamma_i)$) from $\gamma$ (either on the right or left), one ends up with a polynomial in $\gamma \in \Poly(\frac{1}{a} \Z \to \Gamma_i) \cap \Poly(\frac{1}{b} \Z \to \Gamma_i)$ which has an expansion similar to that of \eqref{pig}, \eqref{pig-2} but with the $j=i$ term absent.  Repeating this argument we may eliminate all the other factors in \eqref{pig}, \eqref{pig-2} by dividing out appropriate sequences in $\Poly(\frac{1}{ab}\Z \to \Gamma_i)$, until $\pi_i(\gamma(n))$ is identically equal to $1$ on both $\frac{1}{a}\Z$ and $\frac{1}{b}\Z$, so that $\gamma$ now lies in $\Poly(\frac{1}{a} \Z \to \Gamma_{i+1}) \cap \Poly(\frac{1}{b} \Z \to \Gamma_{i+1})$, and the claim now follows from the induction hypothesis.
\end{proof}

\subsection{Chinese remainder theorems}

\begin{proof}[Proof of Proposition~\ref{chinese}]
We begin by proving an auxiliary claim, namely that if $a_1,\dots,a_m$ are coprime natural numbers, and $\gamma_1,\dots,\gamma_m \in \Poly_{\leq k}(\Z \to \Z)$, then there exists $\gamma \in \Poly_{\leq k}(\Z \to \Z)$ such that $\gamma_i - \gamma \in \Poly_{\leq k}(\frac{1}{a_i} \Z \to \Z)$ for $i=1,\dots,m$.  It suffices to verify this when $m=2$, as this also implies the $m=1$ case, and the higher $m$ cases also follow from induction.  From the first claim of Lemma~\ref{bezout} we can write $\gamma_1 - \gamma_2 = \gamma^*_1 - \gamma^*_2$ where $\gamma^*_1 \in \Poly_{\leq k}(\frac{1}{a_1}\Z \to \Z)$ and $\gamma^*_2 \in \Poly_{\leq k}(\frac{1}{a_2}\Z \to \Z)$.  The claim now follows by setting $\gamma \coloneqq \gamma_1 - \gamma^*_1 = \gamma_2 - \gamma^*_2$.

Now we prove (i).  Write $\phi = (I,P)$ and $\phi_p = (I_p,P_p)$.  From Definition~\ref{poly-comp}, we have
$$ P_p = \eps_p + P + \gamma_p$$
where $\eps_p \in \Poly_{\leq k}(\R \to \R)$ obeys the smoothness bounds in Definition~\ref{poly-comp}(i), and $\gamma_p \in \Poly_{\leq k}( \Z \to \Z)$.  From the previous claim, there exists $\gamma \in \Poly_{\leq k}(\Z \to \Z)$ such that $\gamma_p - \gamma \in \Poly_{\leq k}(\frac{1}{p} \Z \to \Z)$ for each $p$.  If one then sets $\tilde \phi \coloneqq (I, P + \gamma)$, one obtains the claim (i).

Now we prove (ii).  Write $\phi = (I,P)$ and $\phi' = (I', P')$. From hypothesis we may write
$$ P(t) = \epsilon_p(t) + P'(t) + \gamma_p(t)$$
for all $p \in {\mathcal P}$ and some $\epsilon_p, \gamma_p \in \Poly_{\leq k}(\R \to \R)$ obeying the properties in Definition~\ref{poly-comp}.  In particular, we see that $\epsilon_p(t) + \gamma_p(t)$ is independent of $p$.  Setting $n_I$ to be an integer point in $I$, we then have that $\epsilon_p(n_I) \mod 1$ is independent of $p$.  Since also $\epsilon_p(n_I) = O(1)$, we may subtract a bounded integer from each $\epsilon_p$ and add it to $\gamma_p$ to assume without loss of generality that $\epsilon_p(n_I)$ is independent of $p$.  Since $\epsilon_p(n+1) =\epsilon_p(n) + O(1/|I|)$ for all $n \in I \cap \Z$, and $\epsilon_p(n) \mod 1$ is independent of $p$, we conclude from induction (for $|I|$ large enough) that $\epsilon_p(n)$ is independent of $p$ for all $n \in I \cap \Z$, which by Lagrange interpolation (or Lemma~\ref{ber-exp}) implies that $\epsilon_p = \epsilon$ is independent of $p$.  This implies that $\gamma_p = \gamma$ is also independent of $p$.  Since $\gamma \in \Poly_{\leq k}(\frac{1}{p}\Z \to \Z)$ for all $p \in {\mathcal P}$, we see from iterating the second claim of Lemma~\ref{bezout} that $\gamma \in \Poly_{\leq k}(\frac{1}{\prod {\mathcal P}}\Z \to \Z)$, and the claim follows.
\end{proof}

\begin{proof}[Proof of Proposition~\ref{chinese-nil}]
As with the proof of Proposition~\ref{chinese}, we begin by proving an auxiliary claim, namely that if $a_1,\dots,a_m$ are coprime natural numbers, and $\gamma_1,\dots,\gamma_m \in \Poly(\Z \to \Gamma)$, then there exists $\gamma \in \Poly(\Z \to \Gamma)$ such that $\gamma^{-1} \gamma_i \in \Poly(\frac{1}{a_i} \Z \to \Gamma)$ for $i=1,\dots,m$.  As before it suffices from induction to verify the $m=2$ case. From the first claim of Lemma~\ref{bezout-nil} we can write $\gamma_1^{-1} \gamma_2 = (\gamma^*_1)^{-1} \gamma^*_2$ where $\gamma^*_1 \in \Poly(\frac{1}{a_1}\Z \to \Gamma)$ and $\gamma^*_2 \in \Poly(\frac{1}{a_2}\Z \to \Gamma)$.  The claim now follows by setting $\gamma \coloneqq \gamma_1 (\gamma^*_1)^{-1}= \gamma_2 (\gamma^*_2)^{-1}$.

Now we prove (i).  From Definition~\ref{nil-comp}, if we write $\phi = (I,g)$ and $\phi_p = (I_p,g_p)$, we have
$$ g_p = \eps_p g \gamma_p$$
where $\eps_p \in \Poly(\R \to G)$ obeys the smoothness bounds in Definition~\ref{poly-comp}(i), and $\gamma_p \in \Poly( \Z \to \Gamma)$.  From the previous claim, there exists $\gamma \in \Poly(\Z \to \Gamma)$ such that $\gamma^{-1} \gamma_p \in \Poly(\frac{1}{p} \Z \to \Z)$ for each $p$.  If one then sets $\phi' \coloneqq (I, g \gamma)$, one obtains the claim (i).

Now we prove (ii).  Write $\phi = (I,g)$ and $\phi' = (I',g')$.  From hypothesis we may write
\begin{equation}\label{ga}
 g = \epsilon_p g' \gamma_p
\end{equation}
for all $p \in {\mathcal P}$ and some $\epsilon_p, \gamma_p \in \Poly(\R \to G)$ obeying the properties in Definition~\ref{nil-comp}.  Let $n_I$ be an integer point in $I$.  The points $\log \epsilon_p(n_I)$ take values in a ball of size $O(1)$ around the origin in $\log G$.  Let $\delta > 0$ be a small, fixed constant (depending on $k,\eps,\theta,G/\Gamma,F$).  By the pigeonhole principle, one can find a subcollection ${\mathcal P}'$ of ${\mathcal P}$ with $\# {\mathcal P}' \gg_\delta \# {\mathcal P}$ such that $\log \epsilon_p(n_I) = \epsilon_0 + O(\delta)$ for some $\epsilon_0 = O(1)$.  From Bernstein's inequality \eqref{ber} (applied to the function that expresses the distance between $\log \varepsilon_p(t)$ and $\varepsilon_0$) we also have $\log \epsilon_p(t) = \epsilon_0 + O(\delta)$ whenever $t = n_I + O(\delta |I|)$.  From \eqref{ga} one has
\begin{equation}\label{proj}
 (g')^{-1} \epsilon_p^{-1} \epsilon_{p'} g' = \gamma_p \gamma_{p'}^{-1}.
\end{equation}
Now suppose that $t$ is an integer with $t = n_I + O(\delta |I|)$.  By the Baker--Campbell--Hausdorff formula \eqref{poly}, the quantity
$$\epsilon_p(t)^{-1} \epsilon_{p'}(t) = \exp( (-\log \epsilon_p(t)) \ast \log \epsilon_{p'}(t) ) = \exp( (-\epsilon_0+O(\delta)) \ast (\epsilon_0 + O(\delta) ) )$$ 
lies within $O(\delta)$ of the identity, hence the conjugate $g'(t)^{-1} \epsilon_p(t)^{-1} \epsilon_{p'}(t) g'(t)$ lies within $O(\delta)$ of the identity when projected to the abelian group $G/G_2$.  On the other hand by \eqref{gamma-rat}, the projection of $\gamma_p(t) \gamma_{p'}(t)^{-1}$ to $G/G_2$ is rational in the sense that it lies in the image of $\Gamma$ when raised to some power $q=O(1)$.  For $\delta$ small enough, these facts are only compatible if the projection of both sides of \eqref{proj} to $G/G_2$ is trivial, that is to say both sides of \eqref{proj} lie in $G_2$, so $\epsilon_p(t)^{-1} \epsilon_{p'}(t)$ also lies in $G_2$.  Now one can project to the abelian group $G_2/G_3$ and repeat the above arguments to show that both sides of \eqref{proj} lie in $G_3$ (for $\delta$ small enough).  Continuing this argument we conclude that both sides of \eqref{proj} are in fact trivial for all integers $t = n_I + O(\delta |I|)$, and hence by Lagrange interpolation (for $|I|$ large enough) for all real $t$ also.  In particular, $\gamma_p = \gamma$ is independent of $p$.  From the second part of Lemma~\ref{bezout-nil} we conclude that $\gamma \in \Poly(\frac{1}{\prod {\mathcal P}'}\Z \to \Z)$, and the claim follows.
\end{proof}

\bibliography{refs-higher}
\bibliographystyle{plain}

\end{document}